\documentclass[11pt]{article}

\usepackage{amsthm,amsmath,bbm,amssymb,graphicx,booktabs,array,fullpage,url,mathtools,wrapfig,lipsum,mathrsfs,dsfont,titling,epstopdf,bm,relsize,caption,subfig, xcolor,algorithm,enumitem,multirow,array}

\RequirePackage[authoryear]{natbib}

\usepackage[noend]{algpseudocode}
\usepackage[colorlinks=true,
linkcolor=blue,
urlcolor=blue,
citecolor=blue]{hyperref}

\numberwithin{equation}{section}
\newtheorem{conj}{Conjecture}
\newtheorem{thm}[conj]{Theorem}
\newtheorem{cor}[conj]{Corollary}

\newtheorem{prop}[conj]{Proposition}
\newtheorem{lemma}[conj]{Lemma}
\newtheorem{ass}{Assumption}

\newtheorem{example}{Example}

\theoremstyle{definition}\newtheorem{remark}{Remark}
\allowdisplaybreaks

\def\red{\color{red}}

\def\Dt{D^{\text{topic}}}
\def\Dtb{{\bm D}^{\text{topic}}}
\def\hDt{\wh D^{\text{topic}}}
\def\hDtb{\wh {\bm D}^{\text{topic}}}
\def\Dw{D^{\text{word}}}
\def\Dwb{{\bm D}^{\text{word}}}

\def\dtop{d_\text{topic}}

\def\b0{\bm 0}
\def\be{\bm e}
\def\bPi{\bm \Pi}
\def\bT{\bm T}
\def\bX{\bm X}

\def\1{\bm{1}}
\def\bI{\bm{I}}

\def\cO{\mathcal O}
\def\cT{\mathcal{T}}
\def\cC{\mathcal{C}}
\def\cX{\mathcal{X}}
\def\cA{\mathcal{A}}

\def\cN{\mathcal{N}}
\def\M {\mathcal{M}}
\def\H{\mathcal{H}}

\def\E{\mathcal{E}}

\def\P{\mathcal{P}}

\def\diag{\textrm{diag}}

\def\PP{\mathbb{P}}
\def\EE{\mathbb{E}}
\def\RR{\mathbb{R}}
\def\R{\mathbb{R}}

\def\wh{\widehat}
\def\eps{\varepsilon}
\def\I{\mathcal{I}}

\def\supp{\textrm{supp}}
\def\op{\textrm{op}}

\def\rI{{\rm I}}
\def\rII{{\rm II}}
\def\rIII{{\rm III}}

\def\i{\infty}
\def\D{\Delta}
\def\wt{\widetilde}
\let\emptyset\varnothing
\def\rs{\!\!\!}
\def\0{{\ell_0}}
\def\tr{\text{tr}}
\def\T{\top}

\def\oJ{\overline{J}}
\def\uJ{\underline{J}}
\def\olp{\overline{p}}

\def\ua{\underline{\alpha}}
\def\oa{\overline{\alpha}}

\def\mle{\wh T_{\rm mle}}
\def\Tm{T_{\min}}
\def\Pm{\Pi_{\min}}

\DeclareMathOperator*{\argmax}{arg\,max}
\DeclareMathOperator*{\argmin}{arg\,min}

\begin{document}

\title{
 Likelihood estimation of sparse topic distributions in topic models and its  applications to  Wasserstein document distance calculations
} 

\author{Xin Bing\thanks{Department of Statistics and Data Science, Cornell University, Ithaca, NY. E-mail: \texttt{xb43@cornell.edu}.}~~~~~Florentina Bunea\thanks{Department of Statistics and Data Science, Cornell University, Ithaca, NY. E-mail: \texttt{fb238@cornell.edu}.}~~~~~ Seth Strimas-Mackey\thanks{Department of Statistics and Data Science, Cornell University, Ithaca, NY. E-mail: \texttt{scs324@cornell.edu}.}~~~~~Marten Wegkamp\thanks{Departments of Mathematics, and   of Statistics and Data Science, Cornell University, Ithaca, NY. E-mail: \texttt{mhw73@cornell.edu}.} }

\date{}

\maketitle

{\bf Abstract:}  This paper studies the estimation of high-dimensional, discrete, possibly sparse, mixture models in the context  topic models. The data consists of   observed multinomial counts of $p$ words across $n$ independent documents. In topic models, the $p\times n$ expected word frequency matrix is assumed to be factorized as a $p\times K$ word-topic matrix $A$ and a $K\times n$ topic-document matrix $T$. Since columns of both matrices represent conditional probabilities belonging to probability simplices, columns of $A$ are viewed as   $p$-dimensional mixture components that are common to all documents while columns of $T$ are viewed as the $K$-dimensional mixture weights that are document specific and are allowed to be sparse. 

The main interest is to provide sharp, finite sample, $\ell_1$-norm convergence rates for estimators of the mixture weights $T$ when $A$ is either known or unknown. For known $A$, we suggest MLE estimation of $T$. Our non-standard analysis of the MLE not only establishes its $\ell_1$ convergence rate, but also reveals a remarkable property: the MLE, with no extra regularization, can be exactly sparse and
contain  the true zero pattern of  $T$.  We further show that the MLE is both minimax optimal and adaptive to the unknown sparsity in a large class of sparse topic distributions. When $A$ is unknown, we estimate $T$ by optimizing the likelihood function corresponding to a plug in, generic, estimator $\wh A$ of $A$. For any estimator $\wh A$ that satisfies carefully detailed conditions for proximity to $A$,  we show that the resulting estimator of $T$ retains the properties established for the  MLE. Our theoretical results allow the ambient dimensions $K$ and $p$ to grow with the sample sizes. 

Our main application is to the estimation of 1-Wasserstein distances between document generating distributions. We propose, estimate and analyze  new 1-Wasserstein distances between alternative  probabilistic document representations, at the word and topic level, respectively.  We derive finite sample bounds on the estimated proposed 1-Wasserstein distances. For word level document-distances,  we provide contrast with existing rates on the  1-Wasserstein distance between  standard empirical frequency estimates. The effectiveness of the proposed 1-Wasserstein distances is illustrated by an analysis of an IMDB movie reviews data set. 
Finally, our theoretical results are supported by extensive simulation studies.\\ 


    	\section{Introduction}

    We consider the problem of estimating high-dimensional, discrete, mixture distributions, in the context of topic models. The focus of this work is  the estimation, with sharp finite sample convergence rates, of the distribution of the latent topics within the documents of a corpus. Our main application is to the  estimation of  Wasserstein distances between document generating distributions.
    
    In the framework  and traditional jargon of topic models, one has access to a corpus of $n$ documents generated from a common set of $K$ latent topics. Each document $i\in [n]:=\{1,\ldots,n\}$ is modelled as a set of  $N_i$ words drawn from a discrete distribution $\Pi^{(i)}_*$ on $p$ points,  where $p$ is the dictionary size.  We observe the $p$-dimensional word-count vector $Y^{(i)}$ for each document $i\in[n]$, where we assume
    \[
    Y^{(i)} \sim  \text{Multinomial}_p(N_i,\Pi^{(i)}_{*}).
    \]
    The topic model assumption is that the matrix of expected word frequencies  in the corpus, $\bPi_* \coloneqq (\Pi^{(1)}_{*}, \ldots, \Pi^{(n)}_{*})$ can be factorized as
    \begin{equation}\label{topic}
    	\bPi_* = A \bT_*
    \end{equation} 
    Here $A$ represents the $p \times K$ matrix of conditional probabilities of a word, given a topic, and therefore each column of $A$ belongs to the $p$-dimensional probability simplex 
    \[ 
    \Delta_p\coloneqq\{ x\in \RR^p\mid x\succeq {\mathbf 0},\ \1_p^\T x=1\}.
    \]
    The notation $x\succeq \b0$ represents $x_j \ge 0$ for each $j\in[p]$, and $\1_p$ is the vector of all ones. 
    The $K \times n$ matrix $\bT_* \coloneqq  (T^{(1)}_{*},\ldots, T^{(n)}_*)$ collects the probability vectors  $T^{(i)}_* \in \Delta_K$, the simplex in $\R^K$. 
    The entries of  $T^{(i)}_*$ are  probabilities with which each of the $K$ topics occurs within document $i$, for each 
    $i\in [n]$. Relationship (\ref{topic}) would be a very basic application of Bayes' Theorem if $A$ also depended on $i$. A matrix $A$ that is common across documents is the topic model assumption, which we will make in this paper. 
    
    Under model (\ref{topic}), each distribution on words, $\Pi^{(i)}_* = A T^{(i)}_* \in \Delta_p$, is a  discrete mixture of $K < p$ distributions.  The mixture components correspond to the   columns of $A$, and  are therefore common to the entire corpus,
    while the weights, given by the entries of  $T^{(i)}_{*}$, are document specific. Since not all topics  are expected to be covered by all documents, the mixture weights are potentially sparse, in that $T^{(i)}_{*}$ may be sparse. Using their dual interpretation, throughout the paper  we will refer to a vector $T_*^{(i)}$ as either the topic distribution or  the vector of  mixture weights,  in document $i$. 
    
    The observed  word frequencies are collected in a $p\times n$ data matrix $\bX = (X^{(1)}, \ldots, X^{(n)})$ with independent columns  $X^{(i)} = Y^{(i)}/N_i$ corresponding to  the $i$th document. Our main interest is to estimate $\bT_*$ when either the matrix $A$ is known or unknown. We allow for the ambient dimensions $K$ and  $p$ to depend  on  the sizes of the samples $\{N_1,\ldots, N_n\}$ and $n$ throughout the paper.
    
    While, for ease of reference to the existing literature, we will continue to employ the text analysis jargon for the remainder of this work, and our main application will be to the analysis of a movie review data set, our results apply to any data set generated from a model satisfying (\ref{topic}), for instance in biology \citep{cisTopic,Chen2020}, hyperspectral unmixing \citep{ma2013signal} and collaborative filtering \citep{kleinberg2008using}.
    
    The specific problems treated in this work are listed below, and expanded upon in the following subsections.
    
    \begin{enumerate}
    	\item The main focus of this paper is on the derivation  of  sharp, finite-sample,  $\ell_1$-error bounds for estimators $\wh T^{(i)}$ of the potentially sparse topic distributions  $T^{(i)}_*$, under model  (\ref{topic}),  for each $i \in [n]$.  The finite sample analysis covers two cases,  corresponding to whether the components of the mixture, provided by the columns of $A$,  are either (i) known,  or  (ii) unknown, and estimated by $\wh A$ from the corpus data $\bX$.   As a corollary, we derive corresponding finite sample $\ell_1$-norm error bounds for mixture model-based estimators of $\Pi^{(i)}_*$.

    	
    	\item  The main application of our work is to  the construction  and analysis of similarity measures between the documents of a corpus, for measures corresponding to  estimates of the Wasserstein distance between different probabilistic representations of a document.  
    	

    \end{enumerate}

    \subsection{A finite sample analysis of topic and word distribution estimators}\label{sec:top word est}
    
    Finite sample error bounds for estimators  $\wh A$  of $A$ in topic models (\ref{topic}) have been studied in \cite{arora2012learning,arora2013practical,Tracy,TOP,bing2020optimal}, while the finite sample properties of estimators of $T^{(i)}_*$ and, by extension,  those of mixture-model-based estimators 
    of $\Pi_*^{(i)}$, are much less understood, even when $A$ is known beforehand, and therefore $\wh A = A$.


    
    When  $\Pi_*^{(i)}\in \Delta_p$ is a probability vector parametrized as $\Pi_*^{(i)} = g(T)$, with  $T \in \RR^K$, $K < p$ and some {\it known}  function $g$,  provided that  $T$ is identifiable, the study of the  asymptotic properties of the  maximum likelihood estimator (MLE) of $T$, derived  from the $p$-dimensional vector of observed counts $Y^{(i)}$,  is over eight decades old.  Proofs of  the consistency and asymptotic normality of the MLE, when the ambient dimensions $K$ and $p$ do not depend on the sample size, can be traced back to \cite{Rao55, Rao58} 
    and later to the  seminal work of \cite{birch}, and are reproduced, in updated forms, in standard textbooks on categorical data \citep{BFH1975,Agresti2012}.

    The mixture parametrization  treated in this work, when $A$ is known,  is an instance of these well-studied low-dimensional parametrizations. Specialized to our context, for document $i\in [n]$, the parametrization is  $\Pi_{*j}^{(i)}= A_{j\cdot }^\top T_*^{(i)}$ with $T_*^{(i)}\in \Delta_K$, for each component $j \in [p]$ of $\Pi_*^{(i)}$. 
    However, even when $p$ and $K$ are fixed,  the aforementioned classical asymptotic results  are not applicable, as they are established under the following key  assumptions that typically do not  hold for topic models: 
    \begin{enumerate}\setlength\itemsep{0mm}
    	\item [(1)] $ 0 < T^{(i)}_{*k} < 1$,  for {\em all}  $k\in [K]$,
    	\item[(2)] $\Pi_{*j}^{(i)} =A_{j.}^\top T^{(i)}_{*}>0$ for {\em all} $j\in [p]$.
    \end{enumerate}
    The regularity assumption (1)  is crucial in classical analyses \citep{Rao55,Rao58}, and stems from the basic requirement of  $M$-estimation that  $T_{*}^{(i)}$ be an interior point in its appropriate parameter space.  In effect, since $\sum_{k=1}^{K}T_{*k}^{(i)} = 1$, this is a requirement on  only a $(K-1)$ sub-vector of it. 
    In the context of topic models, a given document $i$ of the corpus may not touch upon all $K$ topics, and in fact is expected not to. Therefore,  it is expected that $T^{(i)}_{*k} = 0$, for some $k$.  Furthermore, $K$ represents the number of topics common to the entire corpus, and although topic $k$ may not appear in document $i$, it may be the leading topic of some other document $j$. Both presence and absence of a topic in a document are subject to discovery, and are not known prior to estimation. Moreover,
    one does not observe the topic proportions $T^{(i)}_*$ per document $i$ directly.  Therefore, one cannot use background knowledge, for any given document, to reduce $K$ to a smaller dimension in order to satisfy assumption (1).  
    
    The classical assumption (2) also typically does not hold for topic models. To see this, note that the matrix $A$ is also expected to be sparse: conditional on a topic $k$, some of the words in a large  $p$-dimensional 
    dictionary will not be used in that topic. Therefore, in each column $A_{\cdot k}$, we expect that  $A_{jk} = 0$, for many rows $j\in [p]$.  When the supports of $A_{j\cdot}$ and $T_{*}^{(i)}$ do not intersect, the corresponding probability of word $j$ in document $i$ is zero, $\Pi_{*j}^{(i)} = A_{j\cdot}^\T T_*^{(i)} = 0$. Since zero word probabilities are induced by unobservable sparsity in the topic distribution (or, equivalently, in the mixture weights), one once again cannot reduce the dimension $p$ a priori in a theoretical analysis. Therefore, the assumption (2) is also expected to fail. 
    
    The analysis on the MLE of $T^{(i)}_*$ is 
    thus an open problem  with $A$ being known even for fixed $p$ scenarios, when the standard assumptions (1) and (2) do not hold and  when the problem cannot be artificially reduced to a framework in which  they do. 
    
    \subsubsection*{Finite sample analysis of the rates of the MLE of topic distributions, for known $A$}
    In  Section \ref{sec_est_T_known_A}, we  provide a novel analysis  of the MLE  of  $T_{*}^{(i)}$ for known $A$,  
    under a sparse discrete mixture framework, in which both the ambient dimensions $K$ and $p$ are allowed to grow with the sample sizes $N_i$ and $n$. \cite{kleinberg2008using} refer to the assumption of $A$ being known as the semi-omniscient setting in the context of collaborative filtering and note that even this setting  is, surprisingly,  very  challenging for estimating the mixture weights. 
    By  studying the MLE  of $T_{*}^{(i)}$ when $A$ is known, one  gains appreciation of  the intrinsic difficulty of this problem, that is present even before one further takes into account the estimation of the entire $p\times K$ matrix $A$.

    To the best of our knowledge, the only existing work that treats the aspect of our problem is \cite{arora2016},
    under the assumptions that
    \begin{enumerate}[itemsep = 0mm]
    	\item[(a)] the support $S_*$ of $T_{*}^{(i)}$ is known and $
    	\Tm \coloneqq \min_{k\in S_*}T_{*k}^{(i)} \ge c/s$ with $s = |S_*|$ and $c\in (0,1]$,
    	\item [(b)] the matrix $A$ is known and $\kappa = \min_{\|x\|_1 = 1}\|Ax\|_1 > 0$. 
    \end{enumerate}
    The parameter $\kappa$ is called the $\ell_1\to\ell_1$ condition number of $A$ \citep{kleinberg2008using} which measures the amount of linear independence between columns of $A$ that belong to the simplex $\D_p$. 
    Under (a) and (b),  the problem framework is very close to the classical one, and the novelty in \cite{arora2016} resides in
    the provision of  a finite sample $\ell_1$-error bound of the difference between the restricted MLE (restricted to the known support $S_*$) and the true $T_*^{(i)}$,  a bound that is 
    valid for growing ambient dimensions. However, assumption (a) is rather strong, as the support of $T_*^{(i)}$ is typically unknown.  Furthermore, the restriction $\sum_{k\in S_*} T_{*k}^{(i)}=1$ implies that $\Tm \le 1/s$. Hence (a) essentially requires $T_{*}^{(i)}$ 
    to be  approximately uniform on its a priori known support. 
    This does not hold in general. For instance, even if the support were  known,  many documents will primarily cover a  very small number of topics, while only  mentioning the rest, and thus some topics will be much more likely to occur than others, per document.

    Our novel finite sample analysis in Section \ref{sec_est_T_known_A} avoids the strong condition (a) in \cite{arora2016}.  For notational simplicity, we pick one $i\in [n]$ and drop the superscripts $(i)$ in $X^{(i)}$, $T^{(i)}_*$ and $\Pi_*^{(i)}$ within this section. 
    In Theorem \ref{thm_mle} of Section \ref{sec_theory_genral_bound}, we  first establish  a general  bound for the $\ell_1$-norm of the error $(\mle - T_*)$, with $\mle$ being the MLE of $T_*$. Then, in Section \ref{sec_est_T_known_A_dense_T}, we use this bound as a preliminary result to characterize the regime in which the Hessian matrix of the loss in (\ref{def_MLE}), evaluated at $\mle$, is close to its population counterpart (see condition (\ref{cond_T_neigh}) in Section \ref{sec_est_T_known_A_dense_T}).
    When this is the case, we prove a potentially faster rate of $\|\mle - T_*\|_1$ in Theorem \ref{thm_mle_fast}. A consequence of both Theorem \ref{thm_mle} and Theorem \ref{thm_mle_fast} is summarized in Corollary \ref{cor_mle_fast_dense} of Section \ref{sec_est_T_known_A_dense_T} for the case when
    $T_*$ is dense such that $S_*=[K]$. For dense $T_*$, provided that $ T_{\min}^3 \ge C \log (K) / (\kappa^4 N_i)$ for some sufficiently large constant $C>0$,  $\|\mle - T_*\|_1$ achieves the parametric rate $\sqrt{K/N_i}$, up to a multiplicative factor $\kappa^{-1}$. 
    
    As mentioned earlier, since $T_*$ is not necessarily an interior point,  we cannot appeal to the standard theory of the MLE, nor can we rely on having a zero gradient of the log-likelihood at $\wh T_{\rm mle}$. 
    Instead, our proofs of Theorem \ref{thm_mle} and \ref{thm_mle_fast} consist of the following key steps: 
    \begin{itemize}[leftmargin=5mm, itemsep = 0mm]
    	\item  We prove that the
    	KKT conditions of maximizing the log-likelihood under the restriction that $\mle \in \Delta_K$ lead to a quadratic inequality in $(\mle - T_*)$ of the form 
    	$
    	(\mle - T_*)^\T \wt H (\mle - T_*) \le (\mle - T_*)^\T E,
    	$
    	where (the infinity norm of) $E$ is defined in the next point, and
    	\[
    	\wt H = \sum_{ j: X_j>0} \frac{ X_j}{\Pi_{*j}A_{j\cdot}^\T \mle } A_{j\cdot}A_{j\cdot}^\T.
    	\]
    	
    	\item We bound  the linear term of this inequality by $\|E\|_\i \|\mle - T_*\|_1$ together with a sharp concentration inequality (Lemma \ref{lem_oracle_error} of Appendix \ref{app_tech_lemma}) for 
    	\[
    	\|E\|_\i = \max_{k\in [K]}\left|\sum_{j:\Pi_{*j}>0} {A_{jk} \over \Pi_{*j}} (X_j-\Pi_{*j})\right|.
    	\]
    	\item 
    	We prove that the quadratic term can be bounded from below by
    	$(\kappa^2/2)\|\mle - T_*\|_1^2$, using the definition of the $\ell_1\to\ell_1$ condition number of $A$, and control of the ratios $X_j/\Pi_{*j}$ over a suitable subset of indices $j$ such that $X_j>0$.
    	\item The faster rate in Theorem \ref{thm_mle_fast} requires a more delicate control of $\wt H$, 
    	and its analysis is
    	complicated by the division by $A_{j\cdot}^\T \wh T_{\rm mle}$.
    	To this end, we use the bound in Theorem \ref{thm_mle} to first prove that $A_{j\cdot}^\T \wh T_{\rm mle}\le (1+c)\Pi_{*j}$, for all $j$ with $\Pi_{*j}>0$ and some constant $c\in(0,1)$. 
    	We then  prove a sharp concentration bound (Lemma \ref{lem_I_deviation} of Appendix \ref{app_tech_lemma}) for the operator norm of the matrix $H^{-1/2} (\wh H - H) H^{-1/2}$ for $\wh H= \sum_j X_j \Pi_{*j}^{-2} A_{j\cdot} A_{j\cdot}^\T$
    	and $H= \sum_j  \Pi_{*j}^{-1} A_{j\cdot} A_{j\cdot}^\T$.
    	This will lead to an improved quadratic inequality 
    	\begin{align*}
    		(\mle - T_*)^\T H (\mle - T_*) &\le (1+c)(\mle - T_*)^\T E\\ 
    		&\le (1+c)\|H^{1/2}(\mle - T_*)\|_2\|H^{-1/2}E\|_2.
    	\end{align*}
    	Finally, a sharp concentration inequality for $\|H^{-1/2}E\|_2$ gives the desired faster  rates on $\|\mle - T_*\|_1$.
    \end{itemize}

    \subsubsection*{Minimax optimality and adaptation to sparsity of the MLE of  topic distributions,  for known $A$}
    In Section \ref{sec_est_T_sparse_known_A} we show that the MLE of $T_*$ can be sparse, without any need for extra regularization, a remarkable property that holds in the topic model set-up. Specifically, we introduce in Theorem \ref{thm_supp} a new incoherence condition on the matrix $A$ under which $\{\supp(\wh T_{\rm mle}) \subseteq \supp(T_*)\}$ holds with high probability.  Therefore, if the vector  $T_*$ is sparse, its zero components  will be among those of $\mle$. Our analysis uses a primal-dual witness approach based on the KKT conditions from solving the MLE. To the best of our knowledge, this is the first work proving that the MLE of sparse mixture weights can be exactly sparse, without extra regularization, and determining conditions under which this can happen. 
    Since  $\supp(\wh T_{\rm mle}) \subseteq \supp(T_*)$ implies that if $T_{*k} =0$ for some $k$, so is $[\wh T_{\rm mle}]_k$, this sparsity recovery property further leads to  a faster $\sqrt{s/N_i}$ rate (up to a logarithmic factor) for $\|\mle - T_*\|_1$ with $s = |S_*|$, as summarized in Corollaries \ref{cor_mle_sparse_supp} and \ref{cor_mle_final}  of Section \ref{sec_est_T_sparse_known_A}. In Section \ref{sec_minimax} we prove  that $\sqrt{s/N_i}$
    in fact is the minimax rate of estimating $T_*$ over a large class of sparse topic distributions, implying the minimax optimality of the MLE as well as its adaptivity to the unknown sparsity $s$.

    \subsubsection*{Finite sample analysis of the estimators of topic distributions, for unknown $A$}
    We study the estimation of $T_*$ when $A$ is unknown in Section \ref{sec_est_A_T_unknown_A}. Our procedure of estimating $T_*$ is valid for any estimator $\wh A$ of $A$ with columns of $\wh A$ belonging to $\Delta_p$. For any such estimator $\wh A$, we propose to plug it into the log-likelihood criterion 
    $ \sum_{j} X_j \log ( \wh A_{j\cdot}^\T T)$
    for estimating $T_*$.
    While the proofs are more technical, we can prove that the resulting estimate $\wh T$ of $T_*$ by using $\wh A$ retains all the properties proved for the MLE $\wh T_{\rm mle}$ based on the known $A$ in Section \ref{sec_est_T_known_A}, provided that  the error
    $\| \wh A- A\|_{1,\i} \coloneqq \max_k \|\wh A_{\cdot k}-A_{\cdot k}\|_1$ is sufficiently small. In fact, all bounds of $\|\wh T-T_*\|_1$ in
    Theorems \ref{thm_mle_unknown} and \ref{thm_mle_fast_unknown} and Corollary \ref{cor_mle_unknown} of Section \ref{sec_est_T_unknown_A}, have an extra additive term $\|\wh A- A\|_{1,\i}$ reflecting the effect of estimating $A$. In Theorem \ref{thm_supp_unknown} of Section \ref{sec_est_A_T_unknown_A}, we also show that the estimator $\wh T$ retains the sparsity recovery property despite using $\wh A$. 
    Essentially, our take-home message is that 
    the rate for $\| \wh T- T_*\|_1$ is the same as $\|\mle - T_*\|_1$ plus the  additive error $\|\wh A- A\|_{1,\i}$, provided that
    $\wh A$ estimates $A$ well in $\| \cdot\|_{1,\i}$ norm, with one instance given by the estimator in \cite{TOP} and fully analyzed in Section \ref{sec_est_A_T_unknown_A_eg}.
    
    \subsubsection*{Finite sample analysis of the estimators of  word distributions}
    In Section \ref{sec:est pi} we compare the mixture-model-based estimator $\wt \Pi_{A} = A \mle $ of $\Pi_*$ with the empirical estimator $\wh \Pi=X$ (we drop the document-index $i$), which is simply the $p$-dimensional observed word frequencies, in two aspects: the $\ell_1$ convergence rate and the estimation of probabilities corresponding to zero observed frequencies. 
    For the empirical estimator $\wh\Pi$,   
    we find
    $\EE[ \|\wh\Pi-\Pi_*\|_1 ]\le  \sqrt{\|\Pi_*\|_0 / N}$ with $\|\Pi_*\|_0 = \sum_{j} 1\{ \Pi_{*j}>0\}$, while
    $\EE[ \| \wt\Pi_A - \Pi_*\|_1 ] \le \EE[ \| \wh T_{\rm mle} - T_* \|_1 ] = \cO(\sqrt{ K\log(K) /N})$. We thus expect a faster rate for the model-based estimate $\wt\Pi_A$ whenever $K\log(K) = \cO(\|\Pi_*\|_0)$. 
    Regarding the second aspect, we note that we can have  zero observed frequency ($X_j=0$) for some word $j$ that has
    strictly positive word probability ($\Pi_{*j}>0$). The probabilities of these words are estimated incorrectly by zeroes by the empirical estimate $\wh \Pi$  whereas the model-based estimator $\wt \Pi_A$ can produce strictly positive estimates, for instance, under conditions stated in Section \ref{sec:est pi}.  On the other hand, for the words that have zero probabilities in $\Pi_*$ (hence zero observed frequencies),  the empirical estimate $\wh \Pi$ makes no mistakes in estimating their probabilities while the estimation error of $\wt \Pi_A$ tends to zero at a rate that is no slower than $\sqrt{K\log(K)/N}$. In the case that $\mle$ has correct one-sided sparsity recovery, detailed in Section \ref{sec_est_T_sparse_known_A},  $\wt \Pi_A$ also estimates zero probabilities by zeroes.

    \subsection{Estimates  of the  1-Wasserstein document distances  in topic models} 
    
    In Section \ref{sec:wass} we introduce two alternative probabilistic representations of a document $i\in [n]$: via the word generating probability vector, $\Pi^{(i)}_*$,  or via the topic generating probability vector $T^{(i)}_*$.
    We  use either the 1-Wasserstein distance (see Section \ref{sec:wass} for the definition) between the word distributions,  $ W_1(\Pi_*^{(i)}, \Pi_*^{(j)};\Dw)$,  or the 1-Wasserstein distance between the topic distributions, 
    $W_1(T_*^{(i)}, T_*^{(j)}; \Dt)$, in order to evaluate the proximity of a pair of documents $i$ and $j$, for  metrics $\Dw$ and  $\Dt$ between words and topics, defined in displays  (\ref{def_D_word}) and  (\ref{d top w}) -- (\ref{def_D_topic_TV}), respectively. In particular, in  Section \ref{sec:wass theory} we explain in detail that we regard a topic as a distribution on words, given by a column of $A$, and therefore distances between topics are distances between discrete distributions in $\Delta_p$, and need to be estimated when $A$ is not known.
    
    In Section \ref{sec:w bounds} we propose to estimate the two 1-Wasserstein distances by plug-in estimates $ W_1( \wt \Pi^{(i)},  \wt\Pi^{(j)}; \Dw)$ and 
    $W_1(\wh T^{(i)}, \wh T^{(j)}; \hDt)$, respectively, where 
    $\wt \Pi^{(i)} = \wh A\wh T^{(i)}$ is the model-based estimator of $\Pi_*^{(i)}$ based on a generic estimator $\wh A$ of $A$ and the estimator $\wh T^{(i)}$ of $T_*^{(i)}$ that uses the same $\wh A$, as studied in Section \ref{sec_est_T}.
    We prove in   Proposition \ref{allbounds} of Section \ref{sec:w bounds} that the absolute values of the  errors of both estimates 
    can be bounded by  
    $$ 
    \max_{\ell \in \{i,j\}}\|\wh T^{(\ell)} - T_*^{(\ell)}\|_1 + \|\wh A - A\|_{1,\i}.
    $$
    A main theoretical application
    of the  $\ell_1$-error bounds for the topic distributions derived in Section \ref{sec_est_T} can be used to bound the first term while the second term  reflects the order of the error in estimating $A$, and therefore vanishes  if $A$ is known. For completeness, we take the estimator $\wh A$ proposed in \cite{TOP} and provide in Corollary \ref{cor_allbounds} of Section \ref{sec:w bounds} explicit rates of convergence of both errors of estimating two 1-Wasserstein distances by using this $\wh A$. 
    The  practical implications of the corollary are that a short document length (small $N$)  can be compensated for, in terms of speed of convergence, by having a relatively small number of topics $K$ covered by the entire corpus, whereas working with a very large dictionary (large $p$) will not be detrimental to the rate in a very large corpus (large $n$).

    
    

    To the best of our knowledge, this rate analysis of the estimates of 1-Wasserstein distance corresponding to estimators of discrete distributions in topic models is new.  The only related results, discussed in  Section \ref{sec:wass theory},  have  been established relative to empirical frequency estimators of discrete distributions, from an asymptotic perspective \citep{sommerfeld2017inference,tameling2018empirical} or in finite samples \citep{weed2017sharp}. 
    
    In Remark \ref{rem:allbounds} of Section \ref{sec:w bounds} we discuss the net computational benefits of representing documents in terms of their $K$-dimensional topic distributions, for 1-Wasserstein distance calculations. Using an IMBD movie review corpus as a real data example, we illustrate in Section \ref{sec:imdb} the practical benefits of these distance estimates, relative to the more commonly used  earth(word)-mover's distance \citep{kusner2015wmd}  between observed empirical word-frequencies, $ W_1( \wh \Pi^{(i)}, \wh \Pi^{(j)} ;D^{word}  )$, with $\wh \Pi^{(i)} \coloneqq X^{(i)}$, for all $i \in [n]$. Our analysis reveals that all our proposed 1-Wasserstein distance estimates successfully capture differences in the relative weighting of topics between documents, whereas the standard $W_1(\wh\Pi^{(i)},\wh\Pi^{(j)};\Dw)$  is substantially less successful, likely owing in part to the fact noted in Section \ref{sec:top word est} above, that when the dictionary size $p$ is large, but the document length $N_i$ is relatively small, 
    the quality of $\wh\Pi^{(i)}$ as an estimator of $\Pi^{(i)}_*$ will deteriorate, and the quality of $W_1(\wh\Pi^{(i)},\wh\Pi^{(j)};\Dw)$ as an estimator of  (\ref{wmd pop}) will deteriorate accordingly.\\

    The remainder of the paper is organized as follows. In Section \ref{sec_est_T_known_A} we study the estimation of $T_*$ when $A$ is known. A general bound of $\|\mle - T_*\|_1$ is stated in Section \ref{sec_theory_genral_bound} and is improved in Section \ref{sec_est_T_known_A_dense_T}. The sparsity of the MLE is discussed in Section \ref{sec_est_T_sparse_known_A} and the minimax lower bounds of estimating $T_*$ are established in Section \ref{sec_minimax}. Estimation of $T_*$ when $A$ is unknown is studied in Section \ref{sec_est_A_T_unknown_A}. In Section \ref{sec:est pi} we discuss the comparison between model-based estimators and the empirical estimator of $\Pi_*$. Section \ref{sec:wass} is devoted to our main application: the 1-Wasserstein distance between documents. In Section \ref{sec:wass theory} we introduce alternative Wasserstein distances between probabilistic representations of documents with their estimation studied and analyzed in Section \ref{sec:w bounds}.  
    Section \ref{sec:imdb} contains the analysis of a real data set of IMDB movie reviews. 
    The Appendix contains all proofs, auxiliary results and all simulation results.

    \subsubsection*{Notation} 	For any positive integer $d$, we write $[d] := \{1,\ldots, d\}$. For two real numbers $a$ and $b$, we write $a\vee b = \max\{a,b\}$ and $a\wedge b = \min\{a, b\}$. For any set $S$, its cardinality is written as $|S|$. 
    For any vector $v\in \RR^d$, we write its $\ell_q$-norm as $\|v\|_q$ for $0\le q\le \i$.
    For a subset $S\subset [d]$, we define $v_S$ as the subvector of $v$ with corresponding indices in $S$.
    Let $M\in \RR^{d_1\times d_2}$ be any matrix. For any set $S_1\subseteq [d_1]$ and $S_2 \subseteq [d_2]$, we use $M_{S_1S_2}$ to denote the submatrix of $M$ with corresponding rows $S_1$ and columns $S_2$. In particular, $M_{S_1 \cdot }$ ($M_{\cdot S_2}$) stands for the whole rows (columns) of $M$ in $S_1$ ($S_2$). Sometimes we also write $M_{S_1} = M_{S_1\cdot}$ for succinctness. 
    We use $\|M\|_{\rm op}$ and $\|M\|_q$ to denote the operator norm and elementwise $\ell_q$ norm, respectively. We write $\|M\|_{1,\i} = \max_{j}\|M_{\cdot j}\|_1$. The $k$-th canonical unit vector in $\RR^d$ is denoted by $\be_k$ while $\1_d$ represents the $d$-dimensional vector of all ones. $\bI_d$ is short for the $d\times d$ identity matrix. For two sequences $a_n$ and $b_n$, we write $a_n \lesssim b_n$ if there exists $C>0$ such that $a_n \le Cb_n$ for all $n\ge 1$. For a metric $D$ on a finite set $\cX$, we use boldface ${\bm D}\coloneqq (D(a,b))_{a,b\in \cX}$ to denote the corresponding $|\cX|\times |\cX|$ matrix. The set $\H_d$ contains all $d\times d$ permutation matrices.

    \section{Estimation of topic distributions under topic models}\label{sec_est_T}

    We consider the estimation of the topic distribution vector,  $T^{(i)}_*\in \Delta_K$,  for  each $i\in [n]$.
    Pick any $i\in [n]$; for notational simplicity, we write $T_*=T_*^{(i)}$, $X = X^{(i)}$ and $\Pi_* = \Pi_*^{(i)}$ as well as $N = N_i$ throughout this section. 
    
    We allow, but do not assume, that  the vector $T_*$ is sparse, as sparsity is expected in topic models: a document will cover some, but most likely not all, topics under consideration.  We therefore introduce 
    the following parameter space for $T_*$:  
    \[
    \cT(s) 
    = \left\{T\in \Delta_K: |\supp(T)|  =  s
    \right\},
    \]
    with $s$ being any integer between $1$ and $K$. From now on, we let  $S_* \coloneqq \supp(T_*)$ and write $|S_*|$ for its cardinality.

    In Section \ref{sec_est_T_known_A} we study  the  estimation of $T_*$  from  the observed data $X$, generated from background probability vector $\Pi_*$ parametrized as $\Pi_* = AT_*$, with known matrix $A$. The  intrinsic difficulties associated with the optimal estimation of $T_*$ are already visible when $A$ is known, and we treat this in detail before providing, in Section  \ref{sec_est_A_T_unknown_A}, a full analysis that includes the estimation of $A$.  We remark that assuming $A$ known is not purely unrealistic in topic models used for text data, since then one  typically has access to a large corpus (with $n$ in the order of tens of thousands). When the corpus can be assumed to share the same $A$, this matrix  can be very accurately estimated.

    The results of Section  \ref{sec_est_T_known_A} hold for any known $A$, not required  to have any specific structure: in particular, we do not assume that it follows a topic model with anchor words (Assumption \ref{ass_sep} stated in Section \ref{sec_est_A} below).  We will make this assumption when we consider optimal estimation of $T_*$ when $A$ itself is unknown, in which case Assumption \ref{ass_sep} serves as both a needed identifiability condition and a condition under which estimation of both $A$ and $T_*$, in polynomial time, becomes possible. This is covered in detail in  Section \ref{sec_est_A_T_unknown_A}.

    \subsection{Estimation of $T_*$ when $A$ is known}\label{sec_est_T_known_A}

    When $A$ is known and given, with columns $A_{\cdot k}\in \Delta_p$, the data has a multinomial distribution, 
    \begin{equation}\label{model_single}
    	N X \sim \textrm{Multinomial}_p(N; AT_*),
    \end{equation}
    where $T_*\in \Delta_K$ is  the topic distribution vector,
    with entries corresponding to the  proportions of  the $K$ topics, respectively.  Under (\ref{model_single}), it is natural to consider the  Maximum Likelihood Estimator (MLE) $\wh T_{\rm mle}$ of $T_*$. The log-likelihood, ignoring terms independent of $T$, is proportional to 
    \[  \sum_{j=1}^p  X_j\log\left(A_{j\cdot}^\T T\right) =   \sum_{j\in J}  X_j\log\left(A_{j\cdot}^\T T\right),\] 
    where the last summation is taken over the index set of observed relative frequencies,
    \begin{equation}\label{def_J}
    	J \coloneqq \{j\in [p]: X_j > 0\},
    \end{equation}   
    and using  the convention that  $0^0=1$. Then 
    \begin{equation}\label{def_MLE}
    	\wh T_{\rm mle}
    :=\argmax_{T\in \Delta_K }
\sum_{j\in J}  X_j\log\left(A_{j\cdot}^\T T\right).
\end{equation}
This optimization problem is also known as the log-optimal investment strategy, see for instance \cite[Problem 4.60]{boyd2004convex}.  It can be computed efficiently, since the loss function in (\ref{def_MLE}) is concave
on its domain, the open half space $\cap_{j\in J} \{x\in \RR^K \mid A_{j\cdot}^\T x > 0\}$,  and the constraints $T\succeq0$ and $\1_K^\T T=1$ are convex.

The following two subsections state the theoretical properties of the MLE in (\ref{def_MLE}), and include a study of its adaptivity to the potential sparsity of $T_*$ and minimax optimality. In Section \ref{sec:est pi} we show that although $\wh T_{\rm mle}$ is constructed only from observed, non-zero, frequencies, 
$A_{j\cdot}^\T\wh T_{\rm mle}$ can be a non-zero estimate of $\Pi_{*j}$ for those indices $j \in J^c$ for which  we observe $X_j = 0$.

\subsubsection{A general finite sample bound for $\|\wh T_{\rm mle} - T_*\|_1$}\label{sec_theory_genral_bound}

To analyze $\wh T_{\rm mle}$, we 
first introduce two deterministic sets that control $J$ defined in (\ref{def_J}).  Recalling  $\Pi_* = AT_*$, we collect the words with non-zero probabilities in the set
\begin{equation}\label{def_oJ}
\oJ  \coloneqq \{j\in [p]: \Pi_{*j} > 0\}.
\end{equation}
We will also consider the set 
\begin{equation}\label{def_uJ}
\uJ \coloneqq \left\{
j\in[p]: \Pi_{*j} > 2\eps_j
\right\}
\end{equation}
where 
\begin{equation}\label{def_eps_j}
\eps_j \coloneqq 2\sqrt{\Pi_{*j} \log p \over N} + {4\log p\over 3N}, \qquad \forall ~ 1\le j\le p.
\end{equation}
The sets $\oJ$ and $\uJ$ are appropriately defined such that 
$
\uJ \subseteq J \subseteq \oJ
$
holds with probability at least $1-2p^{-1}$ (see Lemma \ref{lem_basic} of Appendix \ref{app_tech_lemma}).  Define
\begin{equation}\label{def_rho}
\rho \coloneqq \max_{j\in  \oJ} {\|A_{j\cdot}\|_\i \over \Pi_{*j}}.
\end{equation}
We note that $\oJ$, $\uJ$ and $\rho$ all depend on $T_*$ implicitly via $\Pi_*$.
Another important quantity is the following $\ell_1\to\ell_1$ restricted condition number of the submatrix $A_{\uJ}$ of $A$, defined as
\begin{equation}\label{def_kappa_A}
\kappa(A_{\uJ}, s) \coloneqq \min_{S\subseteq[K]: |S|\le s}\min_{v\in \cC(S)}{\|A_{\uJ}v\|_1 \over \|v\|_1},
\end{equation}
with
$$\cC(S) \coloneqq \{ v \in \RR^{K} \setminus \{\b0\}: \|v_S\|_1 \ge \|v_{S^c}\|_1\}.$$
We make the following simple, but very important,  observation that  
\begin{equation}\label{cone_mle}
\wh T_{\rm mle} - T_* \in \cC(S_*)
\end{equation}
with  $S_*= \supp(T_*)$, by using the fact that both $\mle$ and $T_*$ belong to $\Delta_K$. 
In fact, (\ref{cone_mle}) holds generally for any estimator $\wh T\in \Delta_K$ as
\begin{equation*}
0 = \|T_*\|_1 - \|\wh T\|_1 = \|(T_*)_{S_*}\|_1 - \|\wh T_{S_*}\|_1 - \|\wh T_{S_*^c}\|_1 \le \|(\wh T - T_*)_{S_*}\|_1 - \|(\wh T - T_*)_{S_*^c}\|_1.
\end{equation*}
Display (\ref{cone_mle}) implies that the ``effective'' $\ell_1$ error bound of $\wh T_{\rm mle} - T_*$ arises mainly from the estimation of $(T_*)_{S_*}$. Also because of this property, we need the condition number of $A$ to be positive only over the cone $\cC(S_*)$ rather than the whole $\RR^K$.

The following theorem states the convergence rate of $\|\wh T_{\rm mle} - T_*\|_1$. Its proof can be found in Appendix \ref{app_proof_thm_mle}.
\begin{thm}\label{thm_mle}
Assume $\kappa(A_{\uJ},s)>0$. 
For any $\epsilon\ge 0$, with probability $1-2p^{-1}-2\epsilon$, one has 
\begin{equation} \label{firstbound} 
\|\wh T_{\rm mle} - T_*\|_1 ~ \le   ~  {2 \over \kappa^{2}(A_{\uJ}, s)}
\left\{\sqrt{2\rho\log (K/\epsilon) \over N} + {2\rho \log (K/\epsilon)\over N}\right\}.
\end{equation} 
\end{thm} 

Theorem \ref{thm_mle} is a general result that only requires $\kappa(A_{\uJ},s)>0$. The rates depend on two important quantities: $\kappa(A_{\uJ},s)$ and $\rho$, which we discuss below in detail. In the next section we will show that the bound in Theorem \ref{thm_mle} serves as an initial result, upon which one could obtain a faster rate of the MLE in certain regimes.

\begin{remark}[Discussion on $\kappa(A_{\uJ}, s)$]\label{rem_kappa}

The $\ell_1\to\ell_1$ condition number, $\kappa(A, K)$, is commonly used to quantify the linear independence
of the  columns belonging to $\D_p$  of the matrix $A \in \R_+^{p\times K}$ 
\citep{kleinberg2008using}. As remarked in \cite{kleinberg2008using}, the $\ell_1\to\ell_1$ condition number $\kappa(A,K)$ plays the   role of the smallest singular value, $\sigma_K(A) = \inf_{v\ne 0}\|Av\|_2/\|v\|_2$, but it is more appropriate for matrices with columns belonging to a probability simplex. Because of the chain inequalities
\[
{\kappa(A,K) \over \sqrt p} \le  \sigma_K(A) \le \sqrt{K} ~ \kappa(A,K),
\]
and the fact that $K \ll p$, having
$\kappa^{-1}(A,K)$ appear in the bound loses at most a $\sqrt{K}$ factor comparing to $\sigma_K^{-1}(A)$. But using $\sigma_K^{-1}(A)$ potentially yields a much worse bound than using $\kappa^{-1}(A,K)$: there are instances for which $\kappa(A,K)$ is lower bounded by a constant whereas $\sigma_K(A)$ is only of order $o(1)$ (see, for instance, \citet[Appendix A]{kleinberg2008using}). 

The restricted $\ell_1\to\ell_1$  condition number $\kappa(A, s)$ in (\ref{def_kappa_A}) for $1\le s\le K$ generalizes $\kappa(A, K)$ by requiring the condition of $A$ over the cones $\cC(S)$ with $S\subseteq [K]$ and $|S|\le s$. We thus view $\kappa(A, s)$ as the analogue of the restricted eigenvalue \citep{bickel2009simultaneous} of the Gram matrix in the sparse regression settings. In topic models, it has been empirically observed that the (restricted) condition number of $A$ is oftentimes bounded from below by some absolute constant \citep{arora2016}.

To understand why $\kappa(A_{\uJ}, s)$ appears in the rates, recall that the MLE in (\ref{def_MLE}) only uses the words in $J$ as defined in (\ref{def_J}).
Intuitively, only the condition number of $A_{J}$ should play a role as we do not observe any information from words in $J^c \coloneqq [p]\setminus J$. Since $\uJ \subseteq J \subseteq \oJ$ holds with high probability, we can thus bound  $\kappa(A_{J}, s)$ from below by $\kappa(A_{\uJ},s)$. For the same reason, $\rho$ in (\ref{def_rho}) is defined over $j\in \oJ$ rather than $j\in J$.  
\end{remark}

\begin{remark}[Discussion on $\rho$]\label{rem_rho}
Define the smallest non-zero entry in $T_*$ as 
$$
T_{\min} \coloneqq \min_{k\in S_*} T_{*k}.
$$
Recall $\Pi_{*j} = A_{j\cdot}^\T T_* = A_{jS_*}^\T T_{*S_*}$. 
We have
$\rho = \max\left\{\rho_{S_*}, \rho_{S_*^c}\right\}$
where 
\begin{align}\label{eqn_rho_S}
\rho_{S_*} &= \max_{k \in S_*} \max_{j\in \oJ} {A_{jk}\over \sum_{a\in S_*} A_{ja}T_{*a}} \le {1\over T_{\min}},\\\label{eqn_rho_Sc}
\rho_{S_*^c} &= \max_{k \in S_*^c} \max_{j\in \oJ} {A_{jk}\over \sum_{a\in S_*} A_{ja}T_{*a}} \le 
{1\over T_{\min}} \cdot \max_{k\in S_*^c}\max_{j\in \oJ} {A_{jk} \over \sum_{a\in S_*} A_{ja}}.
\end{align}
The magnitudes of both $\rho_{S_*}$ and $\rho_{S_*^c}$ closely depend on  $T_{\min}$ while $\rho_{S_*^c}$ also depends  on
\begin{equation}\label{def_xi}
\xi \coloneqq \max_{j\in \oJ} {\|A_{jS_*^c}\|_\i \over \|A_{jS_*}\|_1},
\end{equation}
a quantity that essentially balances the entries of $A_{jS_*}$ and those of $A_{j S_*^c}$.
Clearly, when $T_*$ is dense, that is, $|S_*|=K$, we have $\xi = 0$.
In general, we have 
\begin{equation}\label{bd_rho}
\rho  \le  (1\vee \xi) / \Tm.
\end{equation}
We further remark that if $A$ has a special structure such that there exists at least one anchor word for each topic $k\in S_*$, that is, for each $k\in [K]$, there exists a row $A_{j\cdot} \propto \be_k$ (see Assumption \ref{ass_sep} in Section \ref{sec_est_A} below), it is easy to verify that the inequality for $\rho_{S_*}$ in (\ref{eqn_rho_S}) is in fact an equality. 
\end{remark}

\subsubsection{Faster rates of $\|\mle - T_*\|_1$}\label{sec_est_T_known_A_dense_T}

In this section we state conditions under which the general bound stated in  Theorem \ref{thm_mle}  can be improved. We begin by noting that  one of the main difficulties in deriving a faster rate for $\|\mle - T_*\|_1$ is in establishing a link between the  Hessian matrix (the second order derivative) of the loss function in (\ref{def_MLE}) evaluated at $\mle$ to that evaluated at $T_*$.  

To derive this link, we prove 
in Appendix \ref{app_proof_thm_mle}
that a relative weighted error of estimating $T_*$ by $\mle$ stays bounded in probability, in the precise sense that  
\begin{equation}\label{cond_T_neigh}
\max_{j\in \oJ} {|A_{j\cdot}^\T (\mle - T_*)|\over A_{j\cdot}^\T T_*} = \cO_\PP(1).
\end{equation}
Further, we show in Lemma \ref{lem_I_deviation} in Appendix \ref{app_tech_lemma}  that the Hessian matrix of (\ref{def_MLE}) at  $T_*$ concentrates around its population-level counterpart, with $X_j$ replaced by $\Pi_{*j}$. 
A sufficient condition under which (\ref{cond_T_neigh}) holds can be derived as follows. 
First note that
\begin{equation} \label{ineq:explain}
\max_{j\in \oJ}{|A_{j\cdot}^\T (\mle - T_*)| \over A_{j\cdot}^\T T_*} \le \max_{j\in \oJ}{\|A_{j\cdot}\|_\i  \over \Pi_{*j}}\|\mle - T_*\|_1 \overset{(\ref{def_rho})}{=} \rho \|\mle - T_*\|_1.
\end{equation}
We have bounded $\rho$  by $\rho \le (1\vee \xi)/ T_{\min}$ in  (\ref{bd_rho}),  and have provided an initial bound on $\|\mle - T_*\|_1$ in Theorem \ref{thm_mle}. Therefore, (\ref{cond_T_neigh}) holds if these two bounds combine to show $\rho \|\mle - T_*\|_1$  is of order $\cO_\PP(1)$. This is summarized in the following theorem. Let $\kappa(A_{\oJ}, K)$ be defined in (\ref{def_kappa_A}) with $s = K$ and $A_{\oJ}$ in place of $A_{\uJ}$. Recall that $\xi$ is defined in (\ref{def_xi}). In addition, we define
\begin{align}\label{M1,M2}
M_1 &\coloneqq  {\log K \over \kappa^4(A_{\uJ},s) ~T_{\min}^3}(1\vee \xi^3),\\
M_2 &\coloneqq {\log K \over \kappa^2(A_{\oJ}, K) T_{\min}^2}(1\vee \xi)(1+ \xi\sqrt{K-s}). \nonumber
\end{align}

\begin{thm}\label{thm_mle_fast}
For any $T_*\in \cT(s)$ with $1\le s \le K$, assume there exists some sufficiently large constant $C>0$ such that 
\begin{equation}\label{cond_N_sparse}
N\ge C \max \{M_1, M_2\}.
\end{equation}
Then, with probability $1-2p^{-1}-4K^{-1}-2e^{-K}$, we have
\[
\|\mle - T_*\|_1 \lesssim \kappa^{-1}(A_{\oJ}, s)\sqrt{K \over N}.
\]
\end{thm}

Condition (\ref{cond_N_sparse}) requires the sample size $N$ to be  sufficiently large  relative to $\Tm$, $\xi$ and the $\ell_1\to \ell_1$ condition number of $A$. 
If
$N\ge C M_1$,
then the argument  in (\ref{ineq:explain}) above implies (\ref{cond_T_neigh}), while we use 
$N\ge C M_2$ 
to prove in Appendix \ref{app_tech_lemma}   that the  Hessian matrix of (\ref{def_MLE}) at $T_*$ concentrates around its population-level counterpart. \\

Combining the bounds in Theorem \ref{thm_mle} and Theorem \ref{thm_mle_fast}, we immediately have the following faster rate of the MLE under (\ref{cond_N_sparse}),
\begin{equation}\label{fast_rates}
\|\mle - T_*\|_1 = \cO_{\PP}\left( \min\left\{\kappa^{-2}(A_{\uJ},s)\sqrt{\rho \log K \over  N}, ~ \kappa^{-1}(A_{\oJ}, s)\sqrt{K \over N}\right\}\right).
\end{equation}
We remark that, when $T_*$ is sparse, the first term in the minimum on the right of (\ref{fast_rates}) could be smaller than the second one (see one instance under item {\it (a)} of Corollary \ref{cor_mle_final}).

However, for dense $T_*\in \cT(K)$ such that $|S_*| = K$, the newly derived rate in Theorem \ref{thm_mle_fast} (the second term in (\ref{fast_rates})) is always faster than that in Theorem \ref{thm_mle} (the first term in (\ref{fast_rates})), as summarized in the following corollary. Its proof follows immediately from Theorem \ref{thm_mle_fast} by replacing $s$ by $K$ and noting that in that case $\xi = 0$,  by (\ref{def_xi}).

\begin{cor}[Dense $T_*$]\label{cor_mle_fast_dense}
For any $T_*\in \cT(K)$, assume there exists some sufficiently large constant $C>0$ such that 
\begin{equation}\label{cond_N_dense}
N \ge C {\log K \over \kappa^4(A_{\uJ},K) ~T_{\min}^3}.
\end{equation}
Then, we have  
\[
\PP\left\{\|\mle - T_*\|_1 \lesssim \kappa^{-1}(A_{\oJ}, K)\sqrt{K \over N}\right\} \ge 1-2p^{-1}-4K^{-1}-2e^{-K}.
\]
\end{cor}

Although in our current application we expect $T_*$ to be exactly sparse,  there are many other applications where $T_*$ can only be approximately sparse. For instance, in a standard Latent Dirichlet Allocation model \citep{BleiLDA}, $T_*$ follows a Dirichlet distribution and is never exactly sparse. The theoretical results derived above are directly applicable to these situations.

Theorem \ref{thm_mle_fast} and  Corollary \ref{cor_mle_fast_dense} allow us to pin-point  the  difficulty in establishing rate adaptation to sparsity of the $\mle$ of  a potentially sparse $T_*$, when its 
sparsity pattern is neither known, nor recovered. To this end, notice that although  the bound in Theorem \ref{thm_mle_fast} is derived for sparse $T_*$, the rate is essentially the same as that of Corollary \ref{cor_mle_fast_dense},  that pertains to a dense $T_*$ and, moreover, is established under the  stronger condition (\ref{cond_N_sparse}). 
This condition involves the quantity
$\xi$ defined in (\ref{def_xi}),  which  balances  entries of $A_{\oJ S_*^c}$ and $A_{\oJ S_*}$.
We thus view (\ref{cond_N_sparse}) as the price to pay, compared to (\ref{cond_N_dense}), for not knowing the support $S_*$ of $T_*$. 
We recall that all  prior existing literature on this problem, either classical
\citep{Agresti2012, BFH1975}
or more recent \cite{arora2016}
assumes that $S_*$ is known. 

The next section establishes the remarkable fact that the MLE of $T_*$ in topic models  can be exactly sparse, under conditions that we establish in this section. This property allows us to relax  (\ref{cond_N_sparse}) and prove that the rate of  $ \|\mle - T_*\|_1$  can adapt to the unknown sparsity of $T_*$, when the support of $\mle$ is included in the support of $T_*$.

\subsubsection{The sparsity  of the MLE in topic models }\label{sec_est_T_sparse_known_A}

We will shortly investigate and discuss conditions under which $\mle$ in topic models is sparse, a remarkable feature of the MLE  since there is no explicit regularization in (\ref{def_MLE}). To that end, we will show that    
\begin{equation}\label{event_supp_mle}
\E_{\supp} \coloneqq \left\{\supp(\mle)\subseteq \supp(T_*)\right\}
\end{equation}
holds with high probability.  Therefore, when $T_*$ has zero entries,   $\mle$ will  also be sparse, and have at least as many zeroes. 
Before stating these results more formally, we give  a first implication, in Corollary \ref{cor_mle_sparse_supp},  of the sparsity of the MLE on its  $\ell_1$-norm rate. 


\begin{cor}\label{cor_mle_sparse_supp}
For any $T_*\in \cT(s)$ with $1\le s < K$, assume there exists some sufficiently large constant $C>0$ such that 
\begin{equation}\label{cond_N_sparse_supp}
N \ge C
{\log (s \vee n) \over \kappa^4(A_{\uJ},s) ~T_{\min}^3}.
\end{equation}
Then, for any $\epsilon \ge 0$, 
\[
\PP\left[ \E_{\supp} \cap    \left\{      \|\mle - T_*\|_1 \lesssim \kappa^{-1}(A_{\oJ}, s)\sqrt{s\log (1/\epsilon ) \over N}\right\} \right] \ge 1-\frac{2}{p} -\frac{4}{ s \vee n} - 2\epsilon^s.
\]
\end{cor}
To compare the rates with Theorem \ref{thm_mle}, suppose Assumption \ref{ass_sep} in Section \ref{sec_est_A}  holds and we have $\rho \ge \rho_{S_*} = {1/\Tm}$ from Remark \ref{rem_rho}. Since $\Tm \le 1/s$ and $1\ge \kappa(A_{\oJ},s) \ge \kappa(A_{\uJ}, s)$, we conclude that the rate in Corollary \ref{cor_mle_sparse_supp} is no slower than that in Theorem \ref{thm_mle}.

Compared to Theorem \ref{thm_mle_fast} and condition (\ref{cond_N_sparse}), on the even $\E_{\supp}$, the faster rate in Corollary \ref{cor_mle_sparse_supp} is obtained under a weaker condition (\ref{cond_N_sparse_supp}). This reflects the benefit of (one-sided) support recovery, $\supp(\mle) \subseteq \supp(T_*)$.

In the following theorem, we show that $\E_{\supp}$ indeed holds with high probability under an incoherence condition on $A$. 

\begin{thm}\label{thm_supp}
For any $T_*\in \cT(s)$ with any $1\le s< K$, assume (\ref{cond_N_sparse_supp}). 
Further assume
there exists some sufficiently small constant $c>0$ such that 
\begin{equation}\label{cond_supp}
\left(\kappa^{-1}(A_{\uJ},s)\sqrt{\xi  s\over \Tm}
+ 1\right)\sqrt{\xi\log(K) \over \Tm N}
~ \le  ~ c \min_{k\in S_*^c} \sum_{j\in \oJ^c}A_{jk}.
\end{equation}
Then, one has 
$
\PP(\E_{\supp}) \ge 1-2p^{-1}-4(s\vee n)^{-1} -4K^{-1}.
$
\end{thm}
\begin{proof}[Sketch of the proof]
We defer the detailed proof to Appendix \ref{app_proof_supp}, but offer a sketch here. For any $T_*$ with $\supp(T_*) \subseteq [K]$, our proof of $\supp(\mle) \subseteq \supp(T_*)$  consists in two steps. We show that
\begin{enumerate}\setlength\itemsep{0mm}
\item[(i)]   there exists an optimal solution $\wt T$ to (\ref{def_MLE}) such that $\supp(\wt T) \subseteq \supp(T_*)$;
\item[(ii)]   if there exists 
any other optimal solution $\bar T$ to (\ref{def_MLE}) that is different from $\wt T$, we also have 
$\supp(\bar T) \subseteq \supp(T_*)$. 
\end{enumerate}
Since $\mle$ itself is an optimal solution to (\ref{def_MLE}), combining (i) and (ii) yields the desired result. 

To prove (i), we use the primal-dual witness approach based on the KKT condition of (\ref{def_MLE}). Specifically, we construct the (oracle) optimal solution $\wt T$ as
\begin{equation}\label{est_T_tilde}
\wt T_{S_*} =  \argmax_{\beta\in \Delta_s} N\sum_{j\in J}X_j \log\left(A_{jS_*}^\T \beta\right),\qquad \wt T_{S_*^c} = \b0.
\end{equation}
Here $S_* = \supp(T_*)$ and $s = |S_*|$. On the   random event
\begin{equation}\label{event_supp}
\max_{k\in S_*^c}\sum_{j\in J}X_j {A_{jk} \over A_{jS_*}^\T \wt T_{S_*}} < 1,
\end{equation}
we prove step (i) by showing that $\wt T$ is an optimal solution to (\ref{def_MLE}) via its KKT condition. Also on the event (\ref{event_supp}), we prove step (ii) by using the concavity of the loss function in (\ref{def_MLE}) together with some intermediate results from proving (i). Finally, we show that the random event (\ref{event_supp_mle}) holds with the specified probability in Theorem \ref{thm_supp} under condition (\ref{cond_supp}).
\end{proof}

For completeness, in Appendix \ref{app_supp_recovery},  we show that for a certain class of  topic models  $\mle$ is not only sparse, but can also consistently estimate the zero entries in $T_*$.  Other examples are possible, but we restrict our attention to  topic  models (\ref{topic}) with anchor words, satisfying Assumption \ref{ass_sep} stated in Section \ref{sec_est_A}, for which we  show that we also have  $\supp(T_*) \subseteq \supp(\mle)$ with high probability. Combination with Theorem \ref{thm_supp} proves consistent support recovery of $\mle$, in this class of topic models, a fact also  confirmed by our simulations in Appendix \ref{sec:sims}.

\begin{example}
We argued above  that, when $A$ has a certain configuration, if $T_*$ has zero entries, so will  $\mle$.  We provide below a simple but illuminating  example of this fact.
Assume that all words are anchor words: each topic uses its own dedicated words, and there is no overlap between words per topic. We collect the respective  word indices, per topic, in the set $\{I_1,\ldots, I_K\}$ which forms a partition of $[p]$. In this case, the columns of $A$ have disjoint supports, and by inspecting the displays (\ref{eq_kkt_oracle_1}) -- (\ref{eq_kkt_oracle_2}) in the proof of Theorem \ref{thm_mle}, one can deduce that  $\mle$ has the following closed-form expression 
\[
[\mle]_{k} = \sum_{i\in I_k}X_i,\qquad \forall\  k\in [K].
\]
Indeed, the above expression can be understood by noting that 
$
Z \sim \textrm{Multinomial}_K(N; T_*) 
$
where  $Z_k = N\sum_{i\in I_k}X_i$  \ for each $k\in [K]$. Therefore, when $T_{*k} = 0$ \ for some $k\in [K]$, we immediately have $Z_k = 0$, w.p. 1. Thus,  $[\mle]_k = 0$, and  $\supp(\mle) \subseteq \supp(T_*)$. 
For a more general $A$, the phenomenon $\supp(\mle) \subseteq \supp(T_*)$ still remains under the incoherence condition (\ref{cond_supp}) that we explain in detail in the following remark.     
\end{example} 

\begin{remark}\label{rem_cond_supp}
Condition (\ref{cond_supp}) can be interpreted as an incoherence condition on the submatrices $A_{\cdot S_*}$ and $A_{\cdot S_*^c}$. To see this, recall from Remark \ref{rem_rho} that $\xi$ controls the largest ratio of $\|A_{jS_*^c}\|_\i$ to $\|A_{jS_*}\|_1$ over all $j\in \oJ$. Since
\[
\oJ = \{j\in [p]: A_{jS_*}\ne \b0\} \quad \textrm{ and }\quad \oJ^c  = \{j\in [p]: A_{jS_*} = \b0\},
\]
the left hand side of (\ref{cond_supp}) controls from above the magnitude of the entries $A_{jS_*^c}$ for the rows with $A_{jS_*} \ne \b0$, whereas the right hand side bounds from below $A_{\cdot S_*^c}$ on the rows with $A_{jS_*} = \b0$. To aid intuition, the following figure illustrates the restriction on $A$ where the submatrix $A_{\oJ S_*^c}$ is required to have relatively small entries, while the submatrix $A_{\oJ^cS_*^c}$ needs to have relatively large entries.
\begin{figure}[H]
\centering
\includegraphics[width = .3\textwidth]{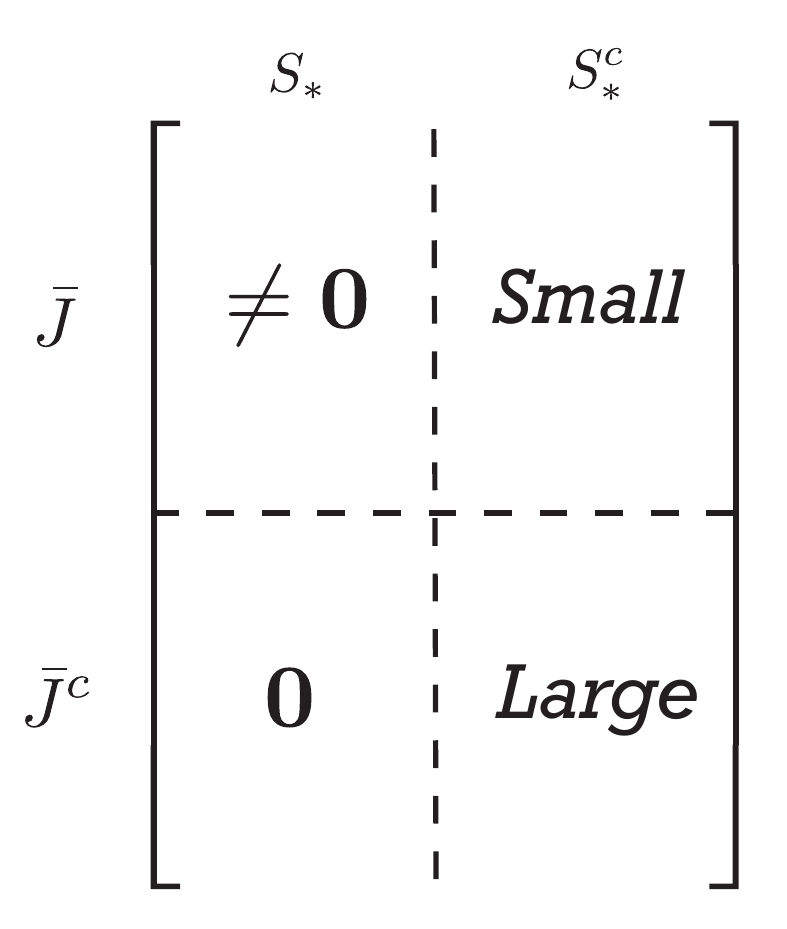}
\end{figure}
\noindent Generally speaking, the more incoherent $A_{\cdot S_*}$ and $A_{\cdot S_*^c}$ are, the more likely condition (\ref{cond_supp}) holds.

In particular,  condition (\ref{cond_supp}) always holds if $A_{\cdot S_*}$ and $A_{\cdot S_*^c}$ have disjoint support. Another favorable situation for (\ref{cond_supp}) is when there exist anchor words in the dictionary  (see, Assumption \ref{ass_sep} in Section \ref{sec_est_A}). Specifically, when there exist at least $m$ anchor words for each of the topics indexed by  $S_*^c$, and their non-zero entries in the corresponding rows of $A$ are lower bounded by $\delta\in (0, 1/m]$ (recall that columns of $A$ sum up to one), the right hand side of (\ref{cond_supp}) is no smaller than $c( m\delta)$. 
In general, suppose $\kappa^{-1}(A_{\oJ},s) = \cO(1)$. If $\xi = \cO(1)$, condition (\ref{cond_supp}) is implied by
\[
\min_{k\in S_*^c} \sum_{j\in \oJ^c}A_{jk} \ge C\sqrt{s\log(K) \over T_{\min}^2 N}.
\]
In case $\xi = \cO(1/s)$, condition (\ref{cond_supp}) requires 
\[
\min_{k\in S_*^c} \sum_{j\in \oJ^c}A_{jk} \ge C \sqrt{\log(K) \over sT^2_{\min} N}.
\]
\end{remark}

To conclude our discussion of the fast rates of the MLE, we remark that the rate in Theorem \ref{thm_mle} per se could be as fast as $\sqrt{s\log K / N}$ under additional conditions and if we restrict ourselves to the following subspace of $\cT(s)$:
\begin{equation}\label{par_space_restrict}
\cT'(s) \coloneqq \cT'(s, c_\star) = \{T\in \cT(s): T_{\min} \ge c_\star/s\}.
\end{equation}
Here $c_\star\in(0,1]$ is some absolute constant.
The following corollary summarizes the conditions that we need to simplify the rates in Theorem \ref{thm_mle} and combines them with the conditions in Corollary \ref{cor_mle_sparse_supp} and Theorem \ref{thm_supp} to yield a faster rate of the MLE when $T_*$ is sparse.


\begin{cor}\label{cor_mle_final}
For any $T_*\in\cT(s)$ with $1\le s < K$, 
\begin{enumerate}
\item[(a)] if  $T_*\in \cT'(s)$,
$\xi = \cO(1)$, $\kappa^{-1}(A_{\uJ}, s) = \cO(1)$ and $s\log (K) = \cO(N)$, then 
\[
\|\wh T_{\rm mle} - T_*\|_1 ~ =  ~ \cO_{\PP}\left(\sqrt{s\log (K) \over N}\right);
\]
\item[(b)] if conditions (\ref{cond_N_sparse_supp}) -- (\ref{cond_supp}) and $\kappa^{-1}(A_{\oJ}, s) = \cO(1)$ hold, then 
\[
\|\wh T_{\rm mle} - T_*\|_1 ~ =  ~ \cO_{\PP}\left(\sqrt{s \over N}\right).
\]
\end{enumerate}
\end{cor}

We note that the bound in case (a) from Theorem \ref{thm_mle} is slower by a factor $\sqrt{\log(K)}$, which is the price to pay for not recovering the support of $T_*$.
In Section \ref{sec_minimax} we benchmark the fast rate $\sqrt{s/N}$ in Corollary \ref{cor_mle_final} and show that it is minimax rate optimal, by establishing the minimax lower bounds of estimating $T_*\in \cT(s)$ for any $1\le s\le K$.

\begin{remark}[Comparison with existing work]
For known $A$, and when $S_*$ is also known, \cite{arora2016} analyzes the estimator $\wt T$ as defined in (\ref{est_T_tilde}). Note that $\wt T$ is not the MLE in general and $\|\wt T -T_*\|_1 = \|(\wt T -T_*)_{S_*}\|_1$ holds by definition. 
Under $\kappa^{-1}(A_{\oJ}, s) = \cO(1)$ and a condition similar to (\ref{cond_N_sparse_supp}), 
\citet[Theorem 5.3]{arora2016} proves $
\|(\wt T - T_*)_{S_*}\|_1 = \cO_{\PP}(\sqrt{s/ N})
$ {\em only} for $T_*\in \cT'(s)$. Therefore, the result of \cite{arora2016} is only comparable to ours when $T_*$ is dense with $S_* = [K]$. Even in this case, our result (see, for instance, Corollary \ref{cor_mle_fast_dense}) is more general in the sense that we do not require $\Tm \ge c/K$ to obtain the same rate. 
More generally, when $S_*$ is unknown, our result in Corollary \ref{cor_mle_final} shows that the MLE can still have fast rates in many scenarios. Moreover, we prove that the MLE is actually sparse and consistently estimate the zero entries of $T_*$ under the incoherence condition (\ref{cond_supp}).
\end{remark}

\subsubsection{Minimax lower bounds and the optimality of the MLE}\label{sec_minimax}

To benchmark the rate of $\mle$ in Corollary \ref{cor_mle_final}, we now establish the minimax lower bound of estimating $T_*$ over $\cT'(s)$ for any $1 < s\le K$. Notice that such a lower bound is also a minimax lower bound over $T_*\in\cT(s)$, a larger parameter space. 

The following theorem states the $\ell_1$-norm  minimax lower bound of estimating $T_*$ in (\ref{model_single}), from data $X$. 
\begin{thm}\label{thm_minimax}
Under (\ref{model_single}), assume $1<s \le c N$ for some small constant $c>0$. Then there exists some absolute constants $c_0 > 0$ and $c_1 \in (0, 1]$, depending on $c$ only, such that 
\[
\inf_{\wh T}\sup_{T_* \in \cT'(s)}\PP\left\{
\|\wh T - T_*\|_1 \ge 
c_0\sqrt{s\over N}
\right\} \ge c_1.
\]
The infimum is taken over all estimators $\wh T \in \Delta_K$.
\end{thm}
Different from the standard $\ell_1$-norm minimax rate, $s/\sqrt{N}$, of estimating an $s$-dimensional unconstrained vector from $N$ i.i.d. observations, for instance the regression coefficient vector  in linear regression, 	Theorem \ref{thm_minimax} shows that the $\ell_1$-minimax rate  of estimating  the  probability vector $T_* \in \cT'(s)$ is of order  $\sqrt{s/N}$.

In view of Theorem \ref{thm_minimax}, under the conditions of Corollary \ref{cor_mle_final}, the MLE is minimax optimal for $T_*\in \cT'(s)$. In fact, Corollary \ref{cor_mle_final} also shows that under conditions therein, the optimal rate can be still achieved by the MLE on a larger space $\cT(s)$. Furthermore, the derived rates in the minimax lower bounds in Theorem \ref{thm_minimax} are  sharp.

It is also worth mentioning that in contrast to the sparse linear regression setting where the minimax optimal rates of estimating a $p$-dimensional vector with at most $s$ non-zero entries contain a $\log(ep/s)$ term, the minimax optimal rates in our context do not contain an additional $\log(eK/s)$ term, an  advantage of MLE support recovery.

\begin{remark}[Method of moments and least squares estimators]\label{rem_MOM}\mbox{}
The method of moments is a natural alternative to MLE-based estimation. It would correspond to estimating $T_*$ by the solution $X = AT_*$. Since this solution may not lie in the probability simplex $\D_K$, one can consider instead   the restricted least squares estimator (RLS) that regresses $X$ onto $A$ over the probability simplex $\D_K$. However, this method is not optimal,  as it does not take into account the heteroscedasticity of the data $X$. We confirmed this via our simulation study in Appendix \ref{sec:sims}.
An iterative weighted RLS could be used to improve the performance of the RLS.  It is well known that in  
the classical setting with $K$ and $p$ fixed, this technique is asymptotically (as $N\to\i$) equivalent with the MLE (in fact, both are efficient), see, for instance,  \cite{BFH1975} and \cite{Agresti2012}. 
We confirmed this in our simulation studies in Appendix \ref{sec:sims}, but found that it never improved upon the MLE, and furthermore, had a significantly greater computational time than the MLE.

\begin{remark}
Since our target $T_*$ lies in a probability simplex, we view the $\ell_1$ norm as a natural metric for quantifying the estimation error. Nevertheless, our analysis readily gives the error bounds of estimating $T_*$ in $\ell_2$-norm, as stated in Appendix \ref{app_rate_ell_2}.
\end{remark}


\end{remark}

\subsection{Estimation of $T_*$ when $A$ is unknown}\label{sec_est_A_T_unknown_A}

When $A$ is unknown, we propose to estimate $A$ first. The estimation of $A$ has been well understood in the literature of topic models, as reviewed in Section \ref{sec_est_A}.
Our procedure of estimating $T_*$ for unknown $A$  is valid for  any estimator of $A$,  and is  stated and analyzed in Section \ref{sec_est_T_unknown_A}. In Section \ref{sec_est_A_T_unknown_A_eg}, we illustrate our general result by applying it to a particular estimator of $A$. 

\subsubsection{Estimation of $A$}\label{sec_est_A}

The estimation of $A$ under topic models has been  originally studied within a Bayesian framework \citep{BleiLDA,MCMC}, and  variational-Bayes  type  approaches were further proposed to accelerate the computation of fully Bayesian approaches. We refer to \cite{blei-intro} for an in-depth overview of this class of techniques.  

More recently   \citep{arora2012learning,arora2013practical,ding,Anandkumar,Tracy,TOP,bing2020optimal} studied provably fast algorithms for estimating $A$ from a frequentist point of view. The common thread of these works, both theoretically and computationally, is the usage of  the following separability condition.

\begin{ass}\label{ass_sep}
For each $k\in [K]$, there exists $j\in [p]$ such that $A_{jk} \ne 0$ and $A_{jk'} = 0$ for all $k' \in [K]\setminus \{k\}$.
\end{ass}

Assumption \ref{ass_sep} is also known as the anchor word assumption as it translates into assuming the existence of  words that are only related to a single topic. It has been empirically shown in \cite{Ding2015} that Assumption \ref{ass_sep} holds in most large corpora  for which the  topic models are reasonable modeling tools.  Assumption \ref{ass_sep}, coupled with a mild regularity condition on the topic matrix $\bT_*\in \R^{K\times n}$,  also serves as an identifiablity condition on model (\ref{topic}),  in that it can be shown that the matrix $A$ can be uniquely recovered from the expected frequency matrix $\bPi_*$. See, \cite{rechetNMF,arora2013practical} for the case when $K$ is known and more recently, \cite{TOP} for the case when $K$ is unknown. Since $K$ can be consistently estimated when it is unknown (see, for instance, \cite{TOP}), in the sequel we focus on estimators of $A$ that have $K$ columns and belong to the space
\begin{equation}\label{def_space_A}
\cA = \left\{A \in \R^{p\times K}: A_{\cdot k}\in \D_p, ~ \forall ~ k\in [K]\right\}.
\end{equation}
Our results of estimating $T_*$ in Section \ref{sec_est_T_unknown_A} below will apply to any estimator $\wh A \in \cA$ that 
is sufficiently close to $A$ in the matrix norms $\| \cdot \|_{\infty}$ and $\| \cdot \|_{1, \infty}$.

\subsubsection{Estimation of $T_*$}\label{sec_est_T_unknown_A}

Our theory for estimating $T_*$ in this section holds for any estimator $\wh A \in \cA$.   We therefore state them as such, and offer an example of the theory applied with a  particular estimator at the end of this section. 
Motivated by (\ref{def_MLE}), given any estimate $\wh A \in \cA$, we propose to estimate $T_*$ by 
\begin{equation}\label{def_T_hat}
\wh T = \argmax_{T\in \Delta_K} N\sum_{j\in J} X_j \log\left(\wh A_{j\cdot}^\T T\right).
\end{equation}
Note that, in contrast to $\mle$ in (\ref{def_MLE}) for known $A$, the above $\wh T$ depends on $\wh A$ and is not the MLE in general for unknown $A$. 

Since one can only identify and estimate  $A$ up to some permutation of columns, the following theorem provides the convergence rate of $\|\wh T -P^\T T_*\|_1$ with $P\in\H_K$ being some $K\times K$ permutation matrix.  	Its proof can be found in Appendix \ref{app_proof_thm_mle_unknown}.
Recall that the sets $\uJ$ and $\oJ$ are defined in  (\ref{def_oJ}) and (\ref{def_uJ}), and the quantity  $\rho$ is defined in (\ref{def_rho}).

\begin{thm}\label{thm_mle_unknown}
Suppose the events 
\begin{equation}
\label{ineq:max}
\bigcap_{j\in \bar J}
\left\{\|\wh A _{j\cdot} - (A  P)_{j\cdot}\|_\i \le \frac12  \Pi_{*j}\right\}
\end{equation}
and
\begin{equation}\label{ineq:ell1infty}
\left\{\|\wh A_{\uJ} - (A P)_{\uJ}\|_{1,\i} \le {1\over 2} \kappa(A_{\uJ},s)\right\} 
\end{equation}
hold with probability $1-\alpha$,	
for some permutation matrix $P\in \H_K$. Then,
we have, with probability $1-4p^{-1}-\alpha$,
\begin{align*}
\|\wh T - P^\T  T_*\|_1 &\le {3\over  \kappa^{2}(A_{\uJ},s)}\left\{
\sqrt{2\rho\log(p) \over N} + {2\rho \log(p)\over 3N}+16\rho \bigl\|\wh A_{\oJ}-  (AP)_{\oJ}\bigr\|_{1,\i}\right.\\
&\hspace{2.4cm}\left.   + {56\over 3}\rho \sum_{j\in \oJ\setminus \uJ} {\|\wh A_{j\cdot} - (AP)_{j\cdot}\|_\i \over \Pi_{*j}}{\log(p) \over N}
\right\}.
\end{align*}
\end{thm}

The restrictions (\ref{ineq:max}) and (\ref{ineq:ell1infty})  and the last two terms in the bound above reflect both the requirement and the effect of estimating $A$ on the overall $\ell_1$-convergence rate of $\wh T$. Note that by using condition  (\ref{ineq:max}) the last term in the bound can be simply bounded from above by 
$$
{\rho \over \kappa^{2}(A_{\uJ},s)}{|\oJ \setminus \uJ|\log(p)\over  N}.
$$ 
This term originates from words that have very small probability of occurrence, $\Pi_{*j}= \cO(\log(p)/N)$,  but have non-zero  observed frequencies,  $X_j>0$. For ease of presentation, we assume in the sequel that the number of such words is bounded, that is, $|\oJ\setminus \uJ| \le C$ for some finite constant $C>0$. Still, our analysis allows one to track their presence throughout the proof. 

To provide intuition of the first requirement (\ref{ineq:max}), suppose $P = \bI_K$ and note that this event guarantees that, for all $T_*\in \Delta_K$ and  for all $j\in \oJ$,
\begin{align}\label{ineq:hat AT-AT}
\wh A_{j\cdot}^\T T_* \in \left[
A_{j\cdot}^\T T_* \pm |\wh A_{j\cdot}^\T T_*-A_{j\cdot}^\T T_*| \right]
\subseteq \left[  \Pi_{*j} \pm \|\wh A_{j\cdot} -A_{j\cdot}\|_\i \right] 
\subseteq \left[  {1\over 2}\Pi_{*j}, ~
{3\over 2} \Pi_{*j} \right]
\end{align}
so that $\wh A_{j\cdot}^\T T_*$ and $\Pi_{*j}$ are the same up to a constant factor.
In particular,   $\Pi_{*j} > 0$ implies $\wh A_{j\cdot}^\T T_*>0$, ensuring that 
$T_*$ lies in the domain of the log-likelihood
function 
$N\sum_{j\in J} X_j \log (\wh A_j^\T T)$.

The second restriction (\ref{ineq:ell1infty})  allows us to replace the $\ell_1\to\ell_1$ condition number of the random matrix  $\wh A_{\uJ}$ by that of $A_{\uJ}$.
Since
\begin{align}\label{ineq:conditionnumber}
\kappa({\wh A_{\uJ}}, s) 
&= \min_{S\subseteq[K]: |S|\le s}\min_{v\in \cC(S)} \frac{ \| \wh A_{\uJ} v\|_1}{\| v\|_1} \nonumber\\
&\ge \kappa(A_{\uJ}, s) - \max_{S\subseteq[K]: |S|\le s}\max_{v\in \cC(S)}  \frac{ \| (\wh A_{\uJ} - A_{\uJ} )v\|_1}{\| v\|_1} \\
&\ge \kappa(A_{\uJ}, s) - \|\wh A_{\uJ}-A_{\uJ}\|_{1,\i},
\nonumber
\end{align}
the bound in (\ref{ineq:ell1infty}) immediately yields
\begin{equation}\label{conditionnumbers}
\kappa(\wh A_{\oJ} , s) \ge \frac12 \kappa(A_{\oJ} ,s).
\end{equation}


\medskip

Similar to the case of $A$ known, treated in Section \ref{sec_est_T_known_A_dense_T}, when $\wh T$ lies in the vicinity of $T_*$ in the sense of (\ref{cond_T_neigh}), the rate of $\|\wh T- P^\T T_*\|_1$ can be improved. The following result is an analogue of Theorem \ref{thm_mle_fast} for unknown $A$. Recall that $M_1$ and $M_2$ are defined in (\ref{M1,M2}). 

\begin{thm}\label{thm_mle_fast_unknown}
Assume there exists a sufficiently large constant $C>0$ such that 
\begin{equation}\label{cond_N_unknown}
N \ge C {\log(p)\over \log(K)}  \max \{M_1, M_2\}.
\end{equation}
Further assume $|\oJ\setminus \uJ| \le C'$ for some constant $C'>0$.
Suppose the events (\ref{ineq:max})
and
\begin{equation}\label{ineq:ell1infty2}
\left\{
\rho^2
{\|\wh A_{\uJ} - (A P)_{\uJ}\|_{1,\i} \le {1\over 2} \kappa^2(A_{\uJ},s)}   \right\}
\end{equation}
hold with probability $1-\alpha$,	
for some permutation matrix $P\in \H_K$. Then,
we have, with probability $1-8p^{-1}-\alpha$,
\[
\|\wh T - P^\T  T_*\|_1 ~ \lesssim ~ 
{1\over \kappa(A_{\oJ},s)}\sqrt{K\log(p) \over N} +{\rho \over  \kappa^{2}(A_{\oJ}, s)}\Bigl\|\wh A_{\oJ}-  (AP)_{\oJ}\Bigr\|_{1,\i}.
\]
\end{thm}
\begin{proof}
The proof of Theorem \ref{thm_mle_fast_unknown} can be found in Appendix \ref{app_proof_thm_mle_fast_unknown}.
\end{proof}

Condition (\ref{cond_N_unknown}) only differs from condition (\ref{cond_N_sparse}) for known $A$ by a $\log (p)$ term. Compared to the restrictions (\ref{ineq:max}) and (\ref{ineq:ell1infty}) in Theorem \ref{thm_mle_unknown}, Theorem \ref{thm_mle_fast_unknown}  
replaces (\ref{ineq:ell1infty}) by the 
stronger requirement (\ref{ineq:ell1infty2}) on $\|\wh A_{\oJ} - (AP)_{\oJ}\|_{1,\i}$ by a factor $ \rho^2/\kappa(A_{\oJ},s)$. \\

Regarding the support recovery of $\wh T$,  we also have an analogue of Theorem \ref{thm_supp} for unknown $A$.  The following theorem states the one-sided support recovery of $\wh T$ in (\ref{def_T_hat}) when $A$ is unknown and estimated by $\wh A \in \cA$.  Its proof can be found in Appendix \ref{app_proof_supp_unknown}.

\begin{thm}\label{thm_supp_unknown}
Assume there exists some positive constants $C,C',C''$ such that 
$
N \ge C
{\log (p) / T_{\min}^3},
$ $|\oJ\setminus \uJ| \le C'$ and $\kappa^{-1}(A_{\uJ}, s) \le C''$.
Suppose the intersection of events (\ref{ineq:max}), (\ref{ineq:ell1infty2}) and 
\begin{align}\label{cond_supp_unknown_A}
\Biggl\{\sqrt{\xi \log(p) \over  \Tm N}\left(1 + \sqrt{\xi s\over \Tm}\right)  &+ {\log(p)\over N}\\
&+\left(1 + {\xi \rho\over \Tm}\right) \|\wh A_{\oJ}- (AP)_{\oJ}\|_{1,\i} \le c \min_{k\in S_*^c} \sum_{j\in \oJ^c}A_{jk}\Biggr\},\nonumber
\end{align}
holds with probability at least $1-\alpha$,
for some permutation matrix $P\in \H_K$ and some sufficiently small constant $c>0$. Then $$
\PP\left\{\supp(\wh T) \subseteq \supp(T_*)\right\}\ge 1-10p^{-1}-\alpha.
$$
\end{thm}

Comparing to (\ref{cond_supp}) in Theorem \ref{thm_supp},  condition (\ref{cond_supp_unknown_A}) is stronger by the factor $\log(p)/N + \rho\|\wh A_{\oJ}- (AP)_{\oJ}\|_{1,\i}$ due to the error of estimating $A$. 
Theorem \ref{thm_supp_unknown} in conjunction with Theorem \ref{thm_mle_fast_unknown} immediately implies that, under the conditions therein, 
\[
\|\wh T - P^\T T_*\|_1 = \cO_{\PP}\left(
\sqrt{s\log p \over N} + \rho \|\wh A_{\oJ} - (AP)_{\oJ}\|_{1,\i}
\right).
\]

Theorem \ref{thm_supp_unknown} provides the one-sided support recovery of the estimator $\wh T$ based on an estimated $A$ that satisfies (\ref{ineq:max}), (\ref{ineq:ell1infty2}) and (\ref{cond_supp_unknown_A}). 
Similar to the results  we established for $\mle$ in Section \ref{sec_est_T_sparse_known_A},  the support of  $\wh T$  can also  consistently recover the support of $T_*$,  over a certain class of topic models, as discussed in Appendix \ref{app_supp_recovery_unknown_A}.

\begin{remark}
Our estimation of $T_*$ uses a plug-in estimator $\wh A$ of $A$ in (\ref{def_T_hat}). The estimation error naturally depends on how well $\wh A$ estimates $A$. Alternatively, if one is willing to assume additional structure on $\bT_*$,  then there exist approaches that directly estimate $\bT_*$ without estimating $A$ first. See, for instance, \cite{bansal2014provable} and \cite{klopp2021assigning}.
\end{remark}

\subsubsection{Application with the estimator proposed in \cite{TOP}}\label{sec_est_A_T_unknown_A_eg}

Our results in Section \ref{sec_est_T_unknown_A} hold for any estimator $\wh A\in\cA$ provided that the rate of $\wh A$ satisfies certain requirements. In this section, we illustrate these general results by taking $\wh A$ as the estimator proposed in \cite{TOP} and by providing concrete conditions for the aforementioned requirements on $\wh A$. 

Since \cite{TOP} studies the estimation of $A$ under Assumption \ref{ass_sep}, we denote by 
$I_k$ the index set of anchor words in topic $k$ for each $k\in [K]$. We write $\Im = \max_{k\in [K]}|I_k|$ and $I = \cup_{k=1}I_k$ with its complement set $I^c = [p]\setminus I$. Let $M := n\vee p\vee N.$ Under conditions stated in Appendix \ref{app_sec_A_cond},  \cite{TOP} establishes the following guarantees on $\wh A$,
\begin{equation}\label{up_bd_A_hat}
\min_{P\in \H_K}\|\wh A - AP\|_{1,\i} = \cO_\PP\left(
\sqrt{K(\Im + |I^c|)\log (M) \over n N}
\right).
\end{equation}
The above rate of convergence in $\|\cdot\|_{1,\i}$ norm is useful to apply Theorem \ref{thm_mle_fast_unknown} and is further shown to be minimax optimal, up to the factor $\log(M)$, in \cite{TOP} under Assumption \ref{ass_sep}. To validate condition (\ref{ineq:max}) in  Theorem \ref{thm_mle_fast_unknown}, one also needs a  control of $\|\wh A_{j\cdot} - (AP)_{j\cdot}\|_\i$ for $j\in \oJ$, which is not  studied in \cite{TOP}. We establish a new result on the rate of convergence of $\|\wh A_{j\cdot} - (AP)_{j\cdot}\|_\i$, that is,
\begin{equation}\label{up_bd_A_hat_rowwise}
\min_{P\in \H_K}\|\wh A_{j\cdot} - (AP)_{j\cdot}\|_\i ~ \lesssim ~ \sqrt{\|A_{j\cdot}\|_\i{K\log(M) \over nN}}\left(
1 \vee \sqrt{p\|A_{j\cdot}\|_\i}
\right)
\end{equation}
holds uniformly over $j\in \oJ$ with probability at least $1-\cO(M^{-1})$.
We defer its precise statement and proof to Theorem \ref{thm_A_sup_norm} of Appendix \ref{app_sec_A_cond}. Equipped with the guarantees on $\wh A$ in (\ref{up_bd_A_hat}) and (\ref{up_bd_A_hat_rowwise}),  for the estimator $\wh T$ of $T_*$ that uses $\wh A$ as the estimator of $A$, the following corollary provides the rate of convergence of $\|\wh T-T_*\|_1$ and its one-sided support recovery.  Set $\Pm \coloneqq \min_{j\in \oJ}\Pi_{*j}$.

\begin{cor}\label{cor_mle_unknown}
Assume that the quantities $\kappa^{-1}(A_{\uJ}, s)$, $\kappa^{-1}(A_{\oJ}, K)$, $\xi$ and $|\oJ\setminus \uJ|$ are bounded,  	     
\begin{equation}\label{cond_N_unknown_simp}
N \ge C~{\log(p)\over \Tm^2}\max\left\{
{1 \over \Tm}~, ~  1+ \sqrt{K-s}
\right\}
\end{equation}
and 
\begin{equation}\label{cond_event_A}
\Tm^2\gtrsim \sqrt{pK\log(M)\over nN}, \qquad \Pm \Tm \gtrsim {K\log (M)\over nN}.
\end{equation}
Then, the estimator  $\wh T$ from (\ref{def_T_hat}) based on $\wh A$
satisfies
\[
\min_{P\in \H_K}   \|\wh T - P^\T  T_*\|_1 
= \cO_\PP \left( 
\sqrt{K\log(p) \over N} + {1\over \Tm}\sqrt{K(\Im + |I^c|) \log (M)\over nN} ~ \right). 
\]
Furthermore, if 
\begin{equation}\label{cond_supp_unknown_A_simp}
{1\over \Tm}\sqrt{s\log(p)\over N} + {1\over \Tm^2}\sqrt{K(\Im + |I^c|) \log (M)\over nN}  \le c \min_{k\in S_*^c} \sum_{j\in \oJ^c}A_{jk},
\end{equation}
holds for some sufficiently small constant $c>0$, then, with probability tending to one as $p\to \i$, we have  $\supp(\wh T) \subseteq \supp(T_*)$ and 
\[
\min_{P\in \H_K}   \|\wh T - P^\T  T_*\|_1 ~ \lesssim ~ 
\sqrt{s\log(p) \over N} + {1\over \Tm}\sqrt{K(\Im + |I^c|) \log (M)\over nN}.
\]
\end{cor}
\begin{proof}
The result follows from Theorem \ref{thm_mle_fast_unknown} and Theorem \ref{thm_supp_unknown} after we  verify its conditions (\ref{ineq:max}), (\ref{cond_N_unknown}), (\ref{ineq:ell1infty2}) and (\ref{cond_supp_unknown_A}).
Condition (\ref{cond_N_unknown}) simplifies to (\ref{cond_N_unknown_simp}).
The rate on $\|\wh A_{j\cdot}-(AP)_{j\cdot}\|_{\i}$ in (\ref{up_bd_A_hat_rowwise}), condition (\ref{cond_event_A}) and the inequality $\max_{j\in \oJ}\|A_{j\cdot}\|_\i / \Pi_{*j} \le \rho\le (1\vee \xi)/\Tm \lesssim 1/\Tm$
imply (\ref{ineq:max}).
The 
rate on $\|\wh A-AP\|_{1,\i}$ in (\ref{up_bd_A_hat}) and the bounds
$\kappa^{-1} (A_{\oJ}, K)=\cO(1)$ and $\xi=\cO(1)$ together with conditions (\ref{cond_event_A}) and (\ref{cond_supp_unknown_A_simp}) imply (\ref{ineq:ell1infty2}) and (\ref{cond_supp_unknown_A}). 
\end{proof}

The result of Corollary \ref{cor_mle_unknown} requires that
\begin{enumerate}\setlength\itemsep{0mm}
\item[(a)] $A$ is well-behaved in that the quantities
$\kappa^{-1}(A_{\uJ}, s)$, $\kappa^{-1}(A_{\oJ}, K)$ and $\xi$ are bounded,
\item[(b)] there are only finitely many very small probability words ($|\oJ\setminus \uJ|$ stays bounded),
\item[(c)] the sample size $N$ is large enough to guarantee (\ref{cond_N_unknown_simp}),
\item[(d)] the corpus size $n$ and sample size $N$ are large enough and both topic probabilities and word probabilities need to satisfy mild signal strength conditions to guarantee  (\ref{cond_event_A}), and
\item[(e)]  $A$ is incoherent,  to satisfy (\ref{cond_supp_unknown_A_simp}) for one-sided support recovery. 
\end{enumerate}

The final bound for $\min_{P} \|\wh T-P^\T T_*\|_1$ involves two terms. Provided 
\begin{equation}\label{cond_n_p}
n ~  \gtrsim  ~ K(\Im + |I^c|) \log (M) \big / (s\Tm^2),
\end{equation}
the  rate $\sqrt {s\log(p) /N}$  dominates and
compared to Corollary \ref{cor_mle_final} and Theorem \ref{thm_minimax},
Corollary \ref{cor_mle_unknown} implies that the estimator $\wh T$ that uses $\wh A$ in \cite{TOP} has the same optimal convergence rate as $\mle$ that uses the true $A$, up to a $\log(p)$ factor.  By using $\Im + |I^c| < p$, one set of sufficient conditions for (\ref{cond_n_p}) is  $n \gg p\log (M) $ and both $K$ and $\Tm^{-1}$ are bounded. In many topic model applications, the number of documents $n$ is typically much larger than the vocabulary size $p$ and the number of topics remains small. For instance, in the IMDB movie reviews in Section \ref{sec:imdb}, we have $p\approx 500$ while $n\approx 20,000$ with the estimated $K$ being $6$.




\subsection{Estimation of $\Pi_*$ in topic models}\label{sec:est pi}

We compare the model-based estimator of $\Pi_*$ with the empirical estimator in two aspects: the $\ell_1$ convergence rate and the estimation of probabilities corresponding to zero observed frequencies. 

\subsubsection*{Improved convergence rate}
We begin   our discussion for known $A$. Let $\wt \Pi_A = A\mle$ be the model-based estimator of $\Pi_*$ with $\mle$ obtained in (\ref{def_MLE}) of Section \ref{sec_est_T_known_A}. Recall that $\wh\Pi = X$ is the empirical estimator of $\Pi_*$. Further recall $\oJ = \{j: \Pi_{*j} > 0\}$ from (\ref{def_oJ}) and write $\olp= |\oJ|$. 
Consider $s = K$, for simplicity. 

For $\wh \Pi$, it is easy to see, using the fact that each component of $N\wh \Pi$ has a Binomial distribution and the Cauchy-Schwarz inequality (twice), that
\begin{equation}\label{bd_Pi_n}
\EE \|\wh \Pi - \Pi_*\|_1  \le 
\sum_{i\in \bar J} \sqrt{ \Pi_{*i} \over N} \le 
\sqrt{\olp\over N}
%
\end{equation}
holds. 
Furthermore, the bound (\ref{bd_Pi_n}) is also sharp (one instance is when $\Pi_{*i} \asymp 1/\olp$). On the other hand, Corollary \ref{cor_mle_fast_dense} together with $\|A\|_{1,\i} = 1$ implies 
$$
\EE\|\wt \Pi_A - \Pi\|_1 \le \EE\|\mle  - T\|_1 \lesssim \sqrt{K\over N},
$$
provided that $\kappa^{-1}(A, K)$ is bounded. 
This rate is   faster than the rate (\ref{bd_Pi_n}) for $\wh\Pi$ by a factor $\sqrt{K/\olp}$. 
In the high-dimensional setting where $p\ge \olp \gg N$,
the bound in (\ref{bd_Pi_n}) does not converge to zero unless the summability condition $\sum_{i\in\oJ}\sqrt{\Pi_{*i}} = \cO(1)$ holds. In contrast, consistency of $\wt \Pi_A$ is  guaranteed 
as long as  $K = {o}(N)$.

When $A$ is unknown, the rate of the empirical estimator can still be improved by the model-based estimator $\wt \Pi_{\wh A} = \wh A\wh T$ with $\wh T$ obtained from (\ref{def_T_hat}) by using an accurate estimator $\wh A\in \cA$ of $A$. Specifically, provided that $\kappa^{-1}(A_{\oJ},s)$ and $\xi$ are bounded, the error due to estimating $A$ plays the following role in estimating $\Pi_*$,
\begin{align*}
\EE \| \wt \Pi_{\wh A} - \Pi_*\|_1 &\le \min_P\left\{\EE\|\wh A_{\oJ} - (AP)_{\oJ}\|_{1,\i} + \EE \|\wh T-P^\T T_*\|_1\right\}\\
&\lesssim  \min_P {1\over \Tm} \EE\|\wh A_{\oJ} - (AP)_{\oJ}\|_{1,\i} + \sqrt{K\log p \over N}
\end{align*}
where we used $\|A\|_{1,\i} = 1$ and $\|\wh T\|_1 = 1$ in the first line and invoked Theorem \ref{thm_mle_fast_unknown} to derive the second line. For the estimator $\wh A$ studied in Section \ref{sec_est_A_T_unknown_A_eg}, we have 
\[ 
\EE \| \wt \Pi_{\wh A} - \Pi_*\|_1 \lesssim {1\over \Tm}\sqrt{K(\Im + |I^c|)\log(M)\over nN}+\sqrt{K\log (p)\over N}.
\]
If $(\Im + |I^c|)\log(M)  \le n \Tm^2$, the above rate simplifies to $\sqrt{K\log(p)/ N}$. Moreover, 
as long as 
\[
n ~ \ge ~  {K(\Im + |I^c|) \over \olp}{\log(M) \over \Tm^2}
\]
and $K \log p\le \olp$, the estimate $\wt \Pi_{\wh A}$ improves upon $\wh \Pi$ (in the $\ell_1$ norm). 

Our model-based estimation of $\bPi_*$ uses the topic model assumption (\ref{topic}) and is to some extent related to other works, such as \cite{cao2020multisample,zhu2021learning}, where the estimation of $\bPi_*$ is studied under a low-rank structure of $\bPi_*$.

\subsubsection*{Estimating word probabilities corresponding to zero observed frequencies}    
One distinct aspect of the model-based mixture estimator compared to the empirical estimator lies in the estimation of the cell probabilities $\Pi_{*j}$ with $j\in J^c = \{j: X_j = 0\}$.

We distinguish between two situations: (i) $j \in J^c$ and $\Pi_{*j} > 0$ and (ii)  $j \in J^c$ and $\Pi_{*j} = 0$. We  discuss them separately.  For ease of reference to the results of the previous sections, recall that  $\oJ = \{j: \Pi_{*j}>0\}$.

In case (i), the empirical estimator always estimates $\Pi_{*j}$ by $\wh \Pi_j = X_j = 0$, while the mixture estimator $\wt \Pi_A $ may produce non-zero estimates, as it is designed to combine the strength of  the mixture components.  For instance, if condition (\ref{cond_T_neigh}) holds, then $[1-o_{\PP}(1)]\Pi_{*j} \le \wt \Pi_{A,j} \le [1+o_{\PP}(1)]\Pi_{*j}$, for all $j\in \oJ$, 
that is,
\[
|\wt \Pi_{A,j} - \Pi_{*j}| =  o_{\PP}\left(
\Pi_{*j} \right)
= o_{\PP}\left(
|\wh \Pi_j - \Pi_{*j}|
\right) \qquad \forall j\in \oJ\cap J^c,
\]
showing that, indeed, $\wt \Pi_{A,j}$ is a non-zero estimator of a non-zero  $\Pi_{*j}$, and has smaller estimation error than $\wh\Pi_j$.

In case (ii), for any $j$  such that $\Pi_{*j} = X_j = 0$, the empirical estimator makes no mistake while the model-based estimator $\wt \Pi_{A,j} = A_{j\cdot}^\T \mle$ could be non-zero. However, we remark that the total error of estimating $j\in \oJ^c$ made by $\wt \Pi_A$ is at most $\|(\mle - T_*)_{S_*^c}\|_1$ which converges to zero no slower than $\sqrt{K\log(K)/N}$ as shown in Section \ref{sec_est_T_known_A}. Indeed, by the fact that $A_{jS_*} = \b0$ for $j\in \oJ^c$, 
\[
\sum_{j\in \oJ^c}|\wt \Pi_{A,j} - \Pi_{*j}| = \sum_{j\in \oJ^c} A_{jS_*^c}^\T (\mle)_{S_*^c} \le \max_{k\in S_*^c}\sum_{j\in \oJ^c}A_{jk}\|(\mle - T_*)_{S_*^c}\|_1\le \|(\mle - T_*)_{S_*^c}\|_1.
\]
In particular, if $\supp(\mle)\subseteq \supp(T_*)$ holds, $ \|(\mle - T_*)_{S_*^c}\|_1 = 0$ and $\wt \Pi_A$ makes no mistake of estimating $\Pi_{*j}$ for $j\in \oJ^c$.

Summarizing, on the one hand, we expect the model-based estimator to outperform the empirical estimator for estimating the cell probabilities in (i). On the other hand, the model-based estimator is no worse than the empirical estimator for estimating the cell probabilities in (ii) when $A$ satisfies an incoherence condition (for instance, condition (\ref{cond_supp})). We verify these two points in our simulation studies in 
Appendix \ref{sec:sims}.

\section{The 1-Wasserstein distance  between documents in topic models}\label{sec:wass}
We now turn to the  main application of the  results of Section \ref{sec_est_T}. By abuse of terminology, we refer to the 1-Wasserstein distance between probabilistic  representations of documents as the distance between documents.  This section is devoted to the theoretical evaluation of the Wasserstein distance between appropriate discrete distributions, in topic models,  and to the illustration of our  proposed methods and theory to the analysis of  a real data set.

Consider  two  discrete  distributions $\gamma, \rho$ on $ \mathcal{X} \coloneqq \{x_1, \ldots,x_{\ell}, \ldots,  x_L\}$, with $x_{\ell} \in E$, where $E$ is a general, abstract, space,  and for some $L\ge 1$.  Let $D$ be a metric on $\mathcal{X}$ and denote by ${\bm D} \coloneqq \left(D(x_a, x_b)\right)_{ 1 \leq a, b\leq L}$ the  $L\times L$  matrix that collects pairwise distances between the elements in  $\mathcal{X}$. Then,  the $W_1$ distance between $\gamma$ and $\rho$ with respect to the metric $D$ is defined as
\begin{equation}\label{w1 def}
W_1(\gamma,\rho~; D) \coloneqq \inf_{w\in \Gamma(\gamma,\rho)} \tr(w{\bm D}).
\end{equation}
where $\Gamma(\gamma,\rho)$ is the set of couplings of $\gamma$ and $\rho$, namely, discrete distributions $w$ on $ \mathcal{X} \times \mathcal{X}$ with marginals $\gamma$ and $\rho$ respectively. In the above notation, $w$ is a doubly-stochastic $L \times L$ matrix.

\subsection{The  1-Wasserstein distance between probabilistic representations of documents at the word  and topic level} \label{sec:wass theory}

We consider two alternative probabilistic representations of a document $i$: (1) as a probability vector on $p$ words, $\Pi^{(i)}_*$, or (2) as a probability vector on $K$ topics, $T^{(i)}_*$.

In view of our data example in Section \ref{sec:imdb}, we regard  words as vectors in $\R^d$, for some $d$. Pre-trained embeddings of words \citep{mikolov2013efficient}, sentences \citep{reimers2019sentencebert}, and documents \citep{le2014distributed}, have become a popular general approach in natural language processing \citep{qiu2020pretrained}, and in particular allow one to define metrics  between words as metrics between their Euclidean vector representations. Specifically, let $\mathcal{X}_{word} \coloneqq \{ x_1, \ldots, x_a, \ldots, x_p\}$, so $x_a\in \R^d$ is a vector representing word $a$ in the dictionary via an embedding in $\R^d$. Then, with $\| \cdot  \|_2$ denoting the Euclidean distance on $\R^d$, we define 
\begin{equation}\label{def_D_word}
\Dw(a,b) \coloneqq \|x_a - x_b\|_2
\end{equation}
as the distance between words $a$ and $b$ for $a,b\in[p]$. The 1-Wasserstein distance between two discrete distributions $\Pi^{(i)}_*$ and $\Pi^{(j)}_*$ supported on these words, for any $i, j \in \{1, \ldots, n\} $,  is 
\begin{equation}\label{wmd pop}
W_1(\Pi^{(i)}_*,\Pi^{(j)}_*;\Dw) \coloneqq \inf_{w\in \Gamma(\Pi^{(i)}_{*}, \Pi^{(j)}_*)} \tr\left( w\Dwb \right).
\end{equation}

Alternatively, viewing the corpus as an ensemble, and under model (\ref{topic}), document differences can be explained in terms  of  $1-$Wasserstein distances between what can be regarded as sketches of the documents, the topic distributions $\bT_*$ in (\ref{topic}). For each document $i \in [n]$, the topic proportion  $T^{(i)}_*$ is  a discrete distribution supported on $K$ topics. Analogous to (\ref{wmd pop}), we define a population-level distance between topic distributions in document $i$ and $j$, based on the $1-$Wasserstein distance, by
\begin{equation}\label{top dist pop}
W_1(T^{(i)}_*, T^{(j)}_*; \Dt)   = \inf_{\alpha\in \Gamma(T^{(i)}_*, T^{(j)}_*)}\tr \left(\alpha \Dtb\right),
\end{equation}
where $\Dtb \in \R_+^{K\times K}$ is a metric matrix on $K$ topics. 

To define $\Dt$, we view a topic as being itself a distribution, on words. Specifically, for every $k \in [K]$, topic $k$ is a distribution on the $p$ words of the dictionary, with mass corresponding to $A_{\cdot k} \in \Delta_p$. We recall that the topic model specifies $A_{jk}$ as the probability of word $j$ given topic $k$. We therefore let $\mathcal{X}_{topic} = \{ A_{\cdot 1}, \ldots, A_{\cdot k}, \ldots, A_{\cdot K}: \  A_{\cdot k} \in \Delta_p\, \ \text{for} \ k \in [K] \}$.
With this view, metrics  between two topics $k$ and $l$ are distances  between discrete distributions $A_{\cdot k}$ and $A_{\cdot \ell}$ in $\Delta_p$,  with supports in  $\mathcal{X}_{word}$.

In this work we focus on two closely related such metrics.  The first one is itself a  1-Wasserstein distance: 
\begin{equation}\label{d top w}
\Dt_{W}(k,\ell) \coloneqq  W_1( A_{\cdot k}, A_{\cdot \ell}, \Dw),\quad \forall k,\ell \in [K],
\end{equation}
the calculation of which  requires optimization in $p$ dimensions and employs  input $\Dwb$ which, in the context of text analysis, is obtained from domain knowledge, as explained above, and further discussed in Section \ref{sec:imdb}.   
The second  metric  is the total variation, TV,  distance: 
\begin{equation}\label{def_D_topic_TV}
\Dt_{TV}(k,\ell) \coloneqq  {1\over 2}\|A_{\cdot k} - A_{\cdot \ell}\|_1,\quad \forall k,\ell \in [K],
\end{equation}
which is optimization free, and independent of the domain knowledge required by  (\ref{d top w}).  


We note that the space $\mathcal{X}_{topic}$ is bounded with respect to both metrics (\ref{d top w}) and (\ref{def_D_topic_TV}). In particular, the total variation distance is always bounded by $1$, and hence, $\|\Dtb_{TV}\|_{\i}\le 1$. Furthermore, by Lemma \ref{thm:w props} in Appendix \ref{proof:allbounds}, for any $k,\ell\in[K]$,
\[\Dt_{W}(k,\ell) =  W_1( A_{\cdot k}, A_{\cdot \ell}, \Dw)\le \|\Dwb\|_\i {1\over 2}\|A_{\cdot k} - A_{\cdot \ell}\|_1\le \|\Dwb\|_\i,\]
and thus $\|\Dtb_{W}\|_\i\le \|\Dwb\|_\i$. As noted in Remark \ref{rem:allbounds} below, $\|\Dwb\|_\i$ is typically bounded; in practice, word embeddings are often normalized to unit-length, in which case $\|\Dwb\|_\i\le 2$.

\subsection{Finite sample error bounds for   estimates of the 1-Wasserstein distance between documents} \label{sec:w bounds}

The theoretical analysis of estimates of the 1-Wasserstein distance $ W_1(\gamma,\rho;D)$ between discrete probability measures $\gamma$ and $\rho$ supported on a metric space $\mathcal{X}$ endowed with metric $D$  has been restricted, to the best of our knowledge,  to estimates  $W_1(\wh \rho ^{(i)},\wh\gamma^{(j)}; D)$ corresponding to  observed empirical frequencies   
$\wh \rho ^{(i)},\wh\gamma^{(j)}$, respectively observed on samples $i$ and $j$, of sizes $N_i$ and $N_j$.  

We drop the superscripts and subscripts in the next few paragraphs, for ease of presentation,  to give a brief overview of the  one-sample related results.

When $L$  is fixed and  $(\mathcal{X}, D)$ has bounded diameter, \cite{sommerfeld2017inference} showed that $\sqrt{N}W_1(\wh \rho,\rho; D)$ converges in distribution, while \cite{tameling2018empirical} showed that when $p = \infty$ and their summability condition (3) holds, $\sqrt{N}W_1(\wh \rho,\rho; D)$  converges weakly over the set of probability measures with finite first moment with respect to $D$, defined in their Section 2.1.

Finite sample rates of convergence for  $W_1(\wh \rho,\rho; D)$ when $L = L(N)$ are less studied, with the exception of \cite{weed2017sharp}, who showed that they are of the order $\sqrt{L/N}$, for $L < N$,  when $(\mathcal{X}, D)$ has bounded diameter, and obtained this result as a particular case of a general theory.

When  $(\mathcal{X}, D)$ has bounded diameter, the rate of $W_1(\wh \rho,\rho; D)$ can be obtained directly from a  bound on $\| \wh \rho - \rho\|_1$,  via the basic inequalities $c\| \wh \rho - \rho\|_1 \leq W_1(\wh \rho,\rho; D) \leq C \| \wh \rho - \rho\|_1 $ \citep{gibbs2002metrics}, where  
$c = \min_{x\neq y\in \mathcal{X}}D(x,y)$ and $C = \max_{x, y\in \mathcal{X}}D(x,y)$.
Therefore, when  $\rho, \wh \rho \in \Delta_L$, and $\wh \rho$  are observed frequencies, the rate $W_1(\wh \rho,\rho; D) \lesssim \sqrt{L/N}$,  with high probability, is therefore immediate,  and is small when $L < N$. Furthermore, $W_1(\wh \rho,\rho; D) \lesssim \sqrt{1/N}$ when $\sum_{j = 1}^{L}\sqrt{\rho_j} < \infty$, for any $L$,  allowed to depend on $N$ and be larger than $N$, matching the rate established for  $L = \infty$ in \citep{tameling2018empirical}.

We complement this literature by constructing and analyzing alternate estimates of the 1-Wasserstein distance between discrete distributions generated according to a topic model (\ref{topic}).  After obtaining any estimate 
$\wh A \in \cA$  and the estimate $\wh T^{(\ell)}$ from (\ref{def_T_hat}) by using this $\wh A$ and $X^{(\ell)}$, for each $\ell\in \{i,j\}$,  we propose to estimate the word-level document distance (\ref{wmd pop}) by 
\begin{equation}\label{def_W_word_tilde}
W_1(\wt \Pi^{(i)}, \wt \Pi^{(j)}; \Dw),\qquad \textrm{ with }\quad \wt \Pi^{(\ell)} = \wh A \wh T^{(\ell)},\quad \forall \ell \in \{i,j\}.
\end{equation}
For the Wasserstein distance between topic distributions in (\ref{top dist pop}) with the two choices of $\Dt$ in (\ref{d top w}) and (\ref{def_D_topic_TV}), we propose to estimate  $W_1(T_*^{(i)}, T_*^{(j)}; \Dt_W)$ and $W_1(T_*^{(i)}, T_*^{(j)}; \Dt_{TV})$, respectively, by 
\begin{alignat}{2}\label{def_D_topic_W_hat}
&W_1(\wh T^{(i)}, \wh T^{(j)}; \hDt_W),\qquad \textrm{ with }\quad \hDt_W(k,\ell) = W_1(\wh A_{\cdot k},\wh A_{\cdot \ell}; \Dw),\quad && \forall k,\ell \in [K];\\
\label{def_D_topic_TV_hat}
&W_1(\wh T^{(i)}, \wh T^{(j)}; \hDt_{TV}),\qquad \textrm{ with }\quad \hDt_{TV}(k,\ell) = {1\over 2}\|\wh A_{\cdot k} - \wh A_{\cdot \ell}\|_1, && \forall k,\ell \in [K].
\end{alignat}
The following proposition shows how error rates of the various Wasserstein distance estimates depend on the estimation of $A$ and $T^{(\ell)}_*$. 
Its proof can be found in Appendix \ref{proof:allbounds}. Recall that $\|M\|_{1,\i} = \max_j \|M_{\cdot j}\|_1$ for any matrix $M$. 
Define 
\[
R(\wh A, \wh T^{(i)}, \wh T^{(j)}) \coloneqq \min_{P\in \H_K}\left\{\|\wh A-AP\|_{1,\i} + \frac{1}{2}\sum_{\ell\in \{i,j\}}\|\wh T^{(\ell)}- P^\T T^{(\ell)}_*\|_1\right\}.
\]

\begin{prop}\label{allbounds}  
For any estimator  $\wh A\in \cA$ and the estimators $\wh T^{(i)}, \wh T^{(j)}$ from (\ref{def_T_hat})  based on this $\wh A$, we have:   
\begin{align}\label{word}
&\left|W_1(\wt \Pi^{(i)},\wt \Pi^{(j)}; \Dw) - W_1( \Pi_*^{(i)}, \Pi_*^{(j)}; \Dw)\right|  ~ \le ~ \|\Dwb\|_\i  R(\wh A, \wh T^{(i)}, \wh T^{(j)});\\\label{topicw}
&\left|W_1(\wh T^{(i)},\wh T^{(j)}; \hDt_W) - W_1(T^{(i)}_*,T^{(j)}_*; \Dt_W)\right| ~ \le ~  \|\Dwb\|_\i  R(\wh A, \wh T^{(i)}, \wh T^{(j)});\\\label{topictv}
& \left|W_1(\wh T^{(i)},\wh T^{(j)}; \hDt_{TV}) - W_1(T^{(i)},T^{(j)}; \Dt_{TV})\right| ~ \le  ~ R(\wh A, \wh T^{(i)}, \wh T^{(j)}). 
\end{align} 
\end{prop}

We provide supporting simulations in Appendix \ref{sec:semi syn sim} to study the rate of estimation of document distances, focusing on the estimator (\ref{def_D_topic_TV_hat}) as an illustrative example.

\begin{cor}\label{cor_allbounds}
Under conditions of Corollary \ref{cor_mle_unknown}, 
for the estimator $\wh A$ proposed in \cite{TOP}, and estimators $\wh T^{(i)}, \wh T^{(j)}$ from (\ref{def_T_hat})  based on this $\wh A$, with probability tending to one, the bounds given in Proposition \ref{allbounds} hold with 
\[
R(\wh A, \wh T^{(i)}, \wh T^{(j)}) ~\lesssim ~  \sqrt{\max\{\|T^{(i)}_*\|_0, \|T^{(j)}_*\|_0\}\log(p) \over N} + {1\over \Tm}\sqrt{K(\Im + |I^c|)\log(M)\over nN}.
\]
\end{cor}


\begin{remark}\label{rem:allbounds} We make the following remarks:
\begin{enumerate}
\item All error upper bounds given by Proposition \ref{allbounds} are of the same order, when $\|\Dwb\|_\i \leq C$, for some constant $C > 0$. In practice, word embedding vectors are often
normalized to unit length when used to define $\Dw$, in which case $\|\Dwb\|_\i \leq 2$.

\item The first two error bounds are the same, but in the first the estimation of both $A$ and $T_*$ play a role in the estimation of $\Pi_*$, whereas the second bound is influenced by the estimation of $A$ via the estimation of the distance metric. 

Although the error bounds are the same, computing the LHS of (\ref{word}) 
involves an optimization in dimension $p$, whereas the LHS of (\ref{topicw}) is  in the much lower  dimension $K$.  Although the distance metric  in (\ref{topicw}) does require the computation of $K(K-1)/2$ Wasserstein  distances in dimension $p$, as in (\ref{word}), all $n(n-1)/2$ pairwise distances between the documents in the corpus can be computed by only a $K$-dimensional Wasserstein distance; this results in a substantial computational gain for $K$ small and $n$ and $p$ large, the typical case in topic modelling (in our example in Section \ref{sec:imdb}, $n=20,605$ and $p=500$, whereas $K=6$). We note that approximations to the $W_1$ distance can be considered to reduce computational complexity at the cost of accuracy, as in \citep{kusner2015wmd}; we instead focus on exact calculation of the $W_1$ distance, but in a reduced dimension ($K$).

\item The LHS in (\ref{topictv}) is once again an optimization in dimension $K$, with input independent of $\|\Dwb\|_{\i}$, and therefore its bound is also independent of this quantity. Furthermore, $\hDt_{TV}$ is computed from simple $\ell_1$ norms of the columns of $\wh A$, so avoids the computational issues of the Wasserstein distance entirely.


\item We will shortly illustrate the advantage of our Wasserstein distance estimates in Section \ref{sec:imdb} below, where we analyze an IMBD movie review corpus.
To exploit the geometry of the word embeddings, \cite{kusner2015wmd} was the first to suggest using the $1$-Wasserstein distance (also known as the Earth Mover's Distance) between the word frequency vectors $\wh\Pi^{(i)}, \wh\Pi^{(j)}$. The benefit of using the Wasserstein distance, relative to the previously used $\ell_2$ or TV distances,  is that it takes into account the relative distance between words, as captured by $\Dw$, so documents with similar meaning can have a small distance even if there is little overlap in the exact words they use.

The analysis of a  corpus of movie reviews, presented in Section \ref{sec:imdb}, illustrates, on the same data set, 
that the three newly proposed document-distance estimates, $W_1(\wt \Pi^{(i)},\wt \Pi^{(j)};\Dw)$,  and  
$W_1(\wh T^{(i)},\wh T^{(j)}; \hDt)$, for estimates of  $\hDt$ of the two metrics defined in (\ref{d top w}) and (\ref{def_D_topic_TV}), are competitive. In particular,  $W_1(\wh T^{(i)},\wh T^{(j)}; \hDt_{TV})$ yields qualitatively similar results, relative to our other two proposed distances, while having the net benefit of  involving  optimization only in $K$ dimensions, and $K \ll p$, typically by several orders of magnitude. Furthermore, it obviates the need for pre-trained word embeddings. Our analysis further reveals that all our proposed distance estimates capture  well topical differences between the documents, while the standard $W_1(\wh\Pi^{(i)},\wh\Pi^{(j)};\Dw)$ between observed document frequencies is substantially less successful. 


\end{enumerate}

\end{remark}

\subsection{Application: IMDB movie reviews}\label{sec:imdb}
In this section we demonstrate our proposed approach of estimating topic proportions for use in document distance estimation.   Using a popular movie-review dataset \citep{maas-EtAl:2011:ACL-HLT2011}, we perform the following steps:

\begin{enumerate}
\item   Estimate the word-topic matrix $A$
using the method in  \cite{bing2020optimal}. For the reader's convenience, we restate the procedure in Appendix \ref{app_alg_A}. 
Use anchor words defined via $\wh A$ to give an initial interpretation of each topic. 
\item Estimate the  topic distributions  $\wh T^{(i)}$ from $\wh A$ and $X^{(i)}$,  for each document $i\in [n]$,  by solving (\ref{def_T_hat}). Use these estimates, in the context of the corpus, to adjust and refine the initial topic interpretation. 
\item Calculate document distances (\ref{def_W_word_tilde}) -- (\ref{def_D_topic_TV_hat}), along with other candidate distances, and compare their ability to capture similarity between the documents.
\end{enumerate}




\subsubsection*{Data and preprocessing} We use a collection of 50K IMDB movie reviews designed for unsupervised learning from the Large Movie Review Dataset \citep{maas-EtAl:2011:ACL-HLT2011}. We preprocess the data by removing stop words and words that have document frequency of less than 1\%. Among the remaining $1685$ words, we keep only the $500$ most common (by term frequency), for ease of interpretation of the topics (we found qualitatively similar results and reached the same conclusions when including all $1685$ words).  We also only keep documents with greater than $50$ words. After preprocessing, we end up with a $p\times n$ word-count matrix $\bX$, where $p=500$, $n= 20,605$.

\begin{remark}\mbox{}
\begin{enumerate}
	\item[(1)]  We recall from Section \ref{sec:top word est} that one motivation of our theoretical analysis of the estimation of $T_*^{(i)}$ is to address the case when $\Pi_{*j}^{(i)}=0$ for a document $i$ and word $j$. After preprocessing, the total number of distinct words in each review in this dataset is $63$ on average, much less than the vocabulary size $p=500$. Thus, for each document $i$ there are typically many words $j$ with $X^{(i)}_j =0$. For at least some of these words, it is possible that $\Pi^{(i)}_{*j}=0$. For example, we find reviews of films in genres such as horror and comedy that have no relation to `war', one of the $500$ words in the vocabulary: for these reviews, it is reasonable to expect the word `war' to have cell probability $\Pi^{(i)}_{*j}=0$. These observations provide a real-data example further motivating the need for a theoretical analysis allowing for this case. 
	
	
	\item[(2)] We also emphasize that our discrete mixture probability estimates allow us to construct non-zero estimates of non-zero $\Pi^{(i)}_{*j}$, even when $X_j^{(i)}=0$. 
	In fact, we find that the average number of non-zero entries in the estimator $\wt \Pi^{(i)}$, over all documents $i\in[n]$, is $490$, much larger than the average number of non-zero entries of $X^{(i)}$ (which we recall was $68$). In most cases, we found zero entries of $\wt \Pi^{(i)}$ correspond to anchor words for topics that are not present in document $i$. This demonstrates that $\wt \Pi^{(i)}$ is able to produce zero estimates for words that we expect to have no chance of occurring in document $i$, while still producing non-zero estimates corresponding to words that could occur in that document, but were not observed in that particular sample. 
	
\end{enumerate}
\end{remark}

\subsubsection{Estimating topic distributions for a refined understanding of the topics covered by a document corpus}

We run the method in \cite{bing2020optimal} on $\bX$ to estimate $A$ for this dataset, with tuning parameter $C_1 = 4$, and denote the output $\wh A$. The number of topics is estimated to be $\wh K = 6$. In Table \ref{table:IMDB top words} in Appendix \ref{sec:imdb supp}, we show the anchor words for each of the 6 topics, from which we can give an initial interpretation to the topics (shown in the third column of the table). In particular, the only anchor words for Topics 3 and 5 are `game' and `episode' respectively, despite this dataset nominally being composed of reviews of full-length movies.

To further interpret the topics (in particular Topics 3 and 5), we compute the estimated topic proportions $\wh T^{(i)}$ from $\wh A$ and $X^{(i)}$ for each document $i\in [n]$  using (\ref{def_T_hat}). Table \ref{table:IMDB excerpts} in Appendix \ref{sec:imdb supp} shows, for each $k\in[\wh K]$, examples of documents such that $\wh T_k = 1$; namely, documents that are generated entirely from topic $k$. This table demonstrates the usefulness of estimating the topic proportions $\wh T^{(i)}$: inspecting these topic-specific documents provides detailed information on what each topic captures. For space limitations, we only give an excerpt of Table \ref{table:IMDB excerpts} here in Table \ref{table:IMDB excerpts short}, featuring Topic 3 and 5.  We find that the documents displayed for Topics 3 and 5 are in fact not movie reviews, but reviews of video games and TV shows, respectively.

\begin{table}
\caption{Excerpts from documents that are estimated to be exclusively generated from Topics 3 and 5 (formally, documents with $\wh T^{(i)}_k = 1$ for each topic $k$). The third column gives the ID number in the original dataset \citep{maas-EtAl:2011:ACL-HLT2011}. See Table \ref{table:IMDB excerpts} in Appendix \ref{sec:imdb supp} for further excerpts for all 6 topics.}
\label{table:IMDB excerpts short}
\begin{tabular}{|l|l|l|p{9cm}|}
	\hline \textbf{Topic} &\textbf{Interpretation}&\textbf{Movie ID}&\textbf{Document excerpt}\\
	\hline 
	\multirow{2}{*}{Topic 3} & \multirow{2}{*}{Video Games} & 23,753&\textit{This game really is worth the ridiculous prices out there\ldots}
	\\ \cline{3-4}
	&  & 12,261&\textit{I remember playing this game at a friend\ldots} \\
	\hline
	\multirow{2}{*}{Topic 5} & \multirow{2}{*}{TV Shows} & 32,315& \textit{I used to watch this show when I was a little girl\ldots }
	\\ \cline{3-4}
	&  & 10,454&\textit{I've watched the TV show Hex twice over and I still can not get enough of it. The show is excellent\ldots } \\
	\hline
\end{tabular}
\end{table}


Besides these non-movie reviews, we confirm that the examples from Topics 1, 4, and 6 are indeed book adaptations, horror films, and films related to war and history, respectively. We see that Topic 2 is indeed related to sentiment in these examples, with both reviews being very negative. All details can be found in Table \ref{table:IMDB excerpts} in Appendix \ref{sec:imdb supp}.

In summary, we have demonstrated that  the estimated topic proportions $\wh\bT$ are useful tools for topic interpretation, and a needed companion to the estimation of $A$, on the basis of which one gives the initial definition of the topics.

\subsubsection{Estimating the 1-Wasserstein distance between documents}

We recall that, by abuse of terminology,  but for clarity of exposition, we refer to distances between probabilistic representations of documents as distances between documents. 

We now compare a set of candidate document distance measures, including our proposed methods. We select several representative documents among the documents kept from the IMDB dataset after preprocessing, and compute the distance between them. We recall that in order to compute the 1-Wasserstein distance between two documents represented via their respective topic-distributions, we need to first calculate the distance between elements on their supports, the topics, which in turn are probability distributions on words, estimated by the columns of $\wh A$. Therefore, with $\wh K = 6$, we first compute (\ref{def_D_topic_W_hat}) and (\ref{def_D_topic_TV_hat}), which we repeat here for clarity: 
\begin{equation}\label{doc dists imdb}
\hDt_W(k,l) = W_1(\wh A_{\cdot k},  \wh A_{\cdot l}; \Dw), \qquad 
\hDt_{TV}(k,l) = \frac{1}{2}\|\wh A_{\cdot k} - \wh A_{\cdot l}\|_1,
\end{equation}
for all $k,l\in \{1, \ldots, 6\}$.
To compute $\Dw$, we use open-source word embeddings from Google\footnote{\url{https://code.google.com/archive/p/word2vec/}} that come pre-trained using the word2vec model \citep{mikolov2013efficient} on a Google News corpus of around 100 billion words. These word embeddings contain a word vector $x_i$ for each the 500 words in our dictionary, except one item in the vocabulary (the number `10', common in movie ratings out of 10), for which we remove the corresponding row from $\wh A$ (then re-normalize to have unit column sums) when computing $\hDt_W$. We follow standard practice of normalizing all word-embeddings to unit length. The distance $\Dw(i,j)$ between words $i$ and $j$ is then computed as $\Dw(i,j) = \|x_i - x_j\|_2/\max_{i,j} \|x_i - x_j\|_2$. We divide by the normalizing factor $\max_{i,j} \|x_i - x_j\|_2$ so that the elements of $\Dw$ are in the range $[0,1]$. This results in $\hDt_W$ also being in the range $[0,1]$, and so on the same scale as $\hDt_{TV}$.

See Table \ref{table: doc details} for details on each document, including the estimated topic proportions, and Table \ref{table:doc dist} for the computed distances.
We make several remarks based on the results in Table \ref{table:doc dist}. 
\begin{enumerate}
\item Consider the distance between documents $D_1$ and $D_2$, which have $\wh T^{(1)}=\wh T^{(2)}$ and are both entirely generated from the Horror topic. Since $\wh T^{(1)}=\wh T^{(2)}$, all distances between the topic proportions (panels (a), (b), and (e)) are equal to zero. Since $\wh T^{(1)}=\wh T^{(2)}$ implies $\wt \Pi^{(1)}=\wt \Pi^{(2)}$,  the distance based on the latter estimators is also zero (Table \ref{table:doc dist}, panel (d)). The only distance that does not capture this underlying topical similarity is the Word Mover's Distance (WMD), which has a value of $0.56$ between $D_1$ and $D_2$.

\item Compare the distances between $D_4$ (a video game review) and each other document (all movie reviews) to the distances between pairs of movie reviews. For the two Wasserstein distances between the topic proportions, as well as the distance based on $\wt \Pi$ (Table \ref{table:doc dist}, panels (a), (b), and (d), respectively), the distance between the video game review $D_4$ and any other review is much greater than the distance between any two movie  reviews. (The one exception is that the distance between $D_4$ and $D_6$ is not large, since $D_6$ has substantial weight on the Video game topic). Thus, these methods are able to detect the difference between video game and movie reviews. We similarly find that these methods detect documents from the TV show topic as outliers from full-length movie reviews, but don't include this in Table \ref{table:doc dist} for simplicity of presentation.

In contrast, the WMD in panel (c) computes the distances between all pairs of distinct documents to be all relatively close together. In fact, based on the WMD, the Horror film review $D_1$ is the same distance to the Video game review $D_4$ as the War \& History film review $D_3$; we note that this is perhaps unsurprising, given that the WMD is not designed to capture similarity based on topics. On the other hand, the TV distance in panel (e) computes the distance between any two documents with disjoint topics to be the maximum value of $1$, not distinguishing between topics that are more or less similar. 


\item The Wasserstein distance based on $\hDt_{TV}$ (panel (a) of Table \ref{table:doc dist}) gives qualitatively similar results to the other two model-based Wasserstein distances (panels (b) and (d)), while obviating the need for the pre-trained word embeddings used to compute $\Dw$ and $\hDt_{W}$, and the calculation of any $p$-dimensional Wasserstein distances, which are computationally expensive.
\end{enumerate}

In summary, the three Wasserstein-based distances defined with the estimated parameters of the topic model (panels (a), (b), (d) in Table \ref{table:doc dist}) are the most successful in capturing topic-based document similarity, and the distance based on $\hDt_{TV}$ (panel (a)) has the further benefit of not requiring the use of pre-trained word embeddings or $p$-dimensional optimization.


\begin{table}[h]
	\centering
	\caption{For each document in Table \ref{table:doc dist}, we give the document ID from the original IMDB dataset, the topic proportions (estimated using (\ref{def_T_hat})), and the interpretations of each topic in the document.}
	\label{table: doc details}
	\begin{tabular}{c|c|c|c}
		\textbf{Document}&\textbf{ID}&\textbf{Topic proportions} & \textbf{Topic interpretations}  \\
		\hline
		$D_1$ & 29,114  & $\wh T = (0,0,0,1,0,0)$ & Horror\\
		$D_2$ & 3,448 & $\wh T = (0,0,0,1,0,0)$ & Horror\\
		$D_3$ & 26,918 & $\wh T = (0,0,0,0,0,1)$ & War \& History\\
		$D_4$ & 23,753 & $\wh T = (0,0,1,0,0,0)$ & Video games\\
		$D_5$ & 4,058 & $\wh T = (0,0,0,0.5,0,0.5)$ & Horror + War \& History\\
		$D_6$ & 5,977 & $\wh T = (0,0,0.5,0.5,0,0)$ & Horror + Video games\\
	\end{tabular}
\end{table}

\begin{table}[h!]
	\caption{Distances between documents using various metrics.}
	\label{table:doc dist}
	\begin{center}
		\subfloat[$W_1(\wh T^{(i)}, \wh T^{(j)},\hDt_{TV})$]{
			\begin{tabular}{ |c|c|c|c|c|c|c|}
				\hline
				&$D_1$ & $D_2$ & $D_3$ & $D_4$ & $D_5$ & $D_6$\\
				\hline 
				$D_1$ &0&	0&	0.14&	0.21&	0.07&	0.10 \\
				\hline
				$D_2$ &$\cdot$ &	0&	0.14&	0.21&	0.07&	0.10\\
				\hline 
				$D_3$ & $\cdot$&$\cdot$&0&	0.23&	0.07&	0.18\\
				\hline
				$D_4$& $\cdot$ &$\cdot$&$\cdot$&0&	0.22&	0.10\\
				\hline
				$D_5$ &$\cdot$ &$\cdot$&$\cdot$&$\cdot$& 0&	0.11\\
				\hline
				$D_6$ &$\cdot$ &$\cdot$&$\cdot$&$\cdot$&$\cdot$&0\\
				\hline
		\end{tabular}}
		\quad
		\subfloat[$W_1(\wh T^{(i)}, \wh T^{(j)},\hDt_{W})$]{
			\begin{tabular}{ |c|c|c|c|c|c|c|}
				\hline
				&$D_1$ & $D_2$ & $D_3$ & $D_4$ & $D_5$ & $D_6$\\
				\hline 
				$D_1$ &0&	0&	0.10&	0.16&	0.05&	0.08 \\
				\hline
				$D_2$ &$\cdot$ &	0&	0.10&	0.16&	0.05&	0.08\\
				\hline 
				$D_3$ & $\cdot$&$\cdot$&0&	0.17&	0.05&	0.13\\
				\hline
				$D_4$& $\cdot$ &$\cdot$&$\cdot$&0&	0.16&	0.08\\
				\hline
				$D_5$ &$\cdot$ &$\cdot$&$\cdot$&$\cdot$& 0&	0.09\\
				\hline
				$D_6$ &$\cdot$ &$\cdot$&$\cdot$&$\cdot$&$\cdot$&0\\
				\hline
		\end{tabular}}
		
		\subfloat[$W_1(\wh \Pi^{(i)}, \wh \Pi^{(j)},\Dw)$ (WMD)]{
			\begin{tabular}{ |c|c|c|c|c|c|c|}
				\hline
				&$D_1$ & $D_2$ & $D_3$ & $D_4$ & $D_5$ & $D_6$\\
				\hline 
				$D_1$ &0&	0.56&	0.62&	0.62&	0.56&	0.60 \\
				\hline
				$D_2$ &$\cdot$ &	0&	0.66 & 0.68 &	0.63&	0.64\\
				\hline 
				$D_3$ & $\cdot$&$\cdot$&0&	0.71&	0.61 &	0.67\\
				\hline
				$D_4$& $\cdot$ &$\cdot$&$\cdot$&0&	0.65&	0.63\\
				\hline
				$D_5$ &$\cdot$ &$\cdot$&$\cdot$&$\cdot$& 0&	0.59\\
				\hline
				$D_6$ &$\cdot$ &$\cdot$&$\cdot$&$\cdot$&$\cdot$&0\\
				\hline
		\end{tabular}}
		\quad
		\subfloat[$W_1(\wt \Pi^{(i)}, \wt \Pi^{(j)},\Dw)$]{
			\begin{tabular}{ |c|c|c|c|c|c|c|}
				\hline
				&$D_1$ & $D_2$ & $D_3$ & $D_4$ & $D_5$ & $D_6$\\
				\hline 
				$D_1$ &0&	0&	0.10&	0.16&	0.05&	0.08 \\
				\hline
				$D_2$ &$\cdot$ &	0&	0.10&	0.16&	0.05&	0.08\\
				\hline 
				$D_3$ & $\cdot$&$\cdot$&0&	0.17&	0.05&	0.12\\
				\hline
				$D_4$& $\cdot$ &$\cdot$&$\cdot$&0&	0.16&	0.08\\
				\hline
				$D_5$ &$\cdot$ &$\cdot$&$\cdot$&$\cdot$& 0&	0.09\\
				\hline
				$D_6$ &$\cdot$ &$\cdot$&$\cdot$&$\cdot$&$\cdot$&0\\
				\hline
		\end{tabular}}
		
		\subfloat[$TV(\wh T^{(i)}, \wh T^{(j)})$]{
			\begin{tabular}{ |c|c|c|c|c|c|c|}
				\hline
				&$D_1$ & $D_2$ & $D_3$ & $D_4$ & $D_5$ & $D_6$\\
				\hline 
				$D_1$ &0&	0&	1&	1&	0.50&	0.50  \\
				\hline
				$D_2$ &$\cdot$ &0&	1&	1 &0.50&	0.50\\
				\hline 
				$D_3$ & $\cdot$&$\cdot$&0&	1&	0.50&	1\\
				\hline
				$D_4$& $\cdot$ &$\cdot$&$\cdot$&0&	1&	0.50\\
				\hline
				$D_5$ &$\cdot$ &$\cdot$&$\cdot$&$\cdot$& 0&	0.50\\
				\hline
				$D_6$ &$\cdot$ &$\cdot$&$\cdot$&$\cdot$&$\cdot$&0\\
				\hline
		\end{tabular}}
	\end{center}
\end{table}

	%
	%

\section*{Acknowledgements}
	Bunea and Wegkamp are supported in part by NSF grant DMS-2015195.



\newpage

\appendix

Appendix \ref{app_supp_recovery} contains results on the support recovery of $T_*$ for both known $A$ and unknown $A$.  Appendix \ref{sec:imdb supp} contains supplementary results on the IMDB data set.  Simulation results on estimation of $T_*$ and $\Pi_*$ are presented in Appendix \ref{sec:sims} while semi-synthetic simulations to compare document-distance estimation rates are stated in Appendix \ref{sec:semi syn sim}. All the proofs are collected in Appendices \ref{app_proof_T_known_A} -- \ref{app_tech_lemma}. Appendix \ref{app_alg_A} contains the algorithm used for estimating the word-topic matrix $A$. Appendix \ref{app_upperbounds_A} states guarantees on estimation of $A$ based on some existing results. 
Finally, discussion on the $\ell_2$ convergence rate of estimating $T_*$ is stated in Appendix \ref{app_rate_ell_2}.

\section{Recovery of the support of $T_*$}\label{app_supp_recovery}

\subsection{Support recovery when $A$ is known}\label{app_supp_recovery_known_A}

We discuss the consistent support recovery of the estimator $\mle$, and introduce another simple consistent estimator of $S_* = \supp(T_*)$ in the presence of anchor words.

In light of Theorem \ref{thm_supp}, establishing consistent support recovery for $\mle$ also requires the other direction, $\supp(T_*)\subseteq \supp(\mle)$, for which we provide a simple sufficient condition below in the presence of anchor words. 

\begin{prop}[Consistent support recovery of $\mle$]\label{prop_supp}
	Suppose there exists at least one anchor word $j_k$ for each topic $k\in S_*$ such that $\Pi_{*j_k} \ge 2\eps_{j_k}$ with $\eps_{j_k}$ defined in (\ref{def_eps_j}). Then, with probability $1-2p^{-1}$,
	\[
	\supp(T_*) \subseteq \supp(\mle).
	\]
	Furthermore, if additionally (\ref{cond_supp}) holds, then, with probability $1-2p^{-1}-6s^{-1} - 2K^{-1}$,
	\[
	\supp(T_*) = \supp(\mle).
	\]
\end{prop}
Proposition \ref{prop_supp} imposes a signal condition on the frequency of the anchor words corresponding to the non-zero topics. Recall  $\eps_{j_k}$ from (\ref{def_eps_j}) that the signal condition simply requires 
\[
\Pi_{*j_k} \gtrsim {\log(p) \over N}, \quad \textrm{for one anchor word $j_k$ of topic $k\in S_*$.}
\]

In addition to the above signal condition, if Assumption \ref{ass_sep} holds (or equivalently, there exists at least one anchor word for each of the zero topics, that is, the topic $k\in S_*^c$), then the following simple estimator 
\begin{equation}\label{def_S_hat}
	\wh S\coloneqq \{k\in [K]:\ \exists X_j>0 \text{  corresponding to anchor word $j$ of topic $k$}\}
\end{equation}
consistently estimates $S_*$, as stated in the following proposition.
\begin{prop}\label{prop_S_hat}
	Under Assumption \ref{ass_sep}, 
	we have $\wh S \subseteq S_*$ with probability one.
	Furthermore, if additionally there exists at least one anchor word $j_k$ for each topic $k\in S_*$ such that $\Pi_{*j_k} \ge 2\eps_{j_k}$ with $\eps_{j_k}$ defined in (\ref{def_eps_j}). Then,
	\[
	\PP\{\wh S = S_*\} \ge 1 - 2p^{-1}.
	\]
\end{prop}
\begin{proof}
	To show $\wh S\subseteq S_*$,
	if $k\in \wh S$, then we must have $k\in S_*$. This is because if $k\not\in S_*$, with probability one, we couldn't have observed any anchor word $X_j>0$ of topic $k$ as $\Pi_{*j} = A_{jk}T_{*k} = 0$. Conversely, to show $S_* \subseteq \wh S$, if $k\in S_*$ and there exists a $\Pi_{*j_k} > 2 \eps_{j_k}$, then on the event $\E$, $X_{j_k} \ge \Pi_{*j_k} - |X_{j_k}-\Pi_{*j_k}| > \eps_{j_k}>0$, that is, $k\in \wh S$. This completes the proof. 
\end{proof}

The estimator $\wh S$ simply collects the topics for which we have observed anchor words. 
Proposition \ref{prop_supp} ensures that we always have $\wh S \subseteq S_*$ under Assumption \ref{ass_sep}. In practice, this property is helpful to check whether $\E_{\supp}$ holds. Specifically, if Assumption \ref{ass_sep} holds and 
we find $\supp(\mle) \subseteq \wh S$, then we necessarily have $\supp(\mle) \subseteq \supp(T_*)$.

\subsection{Support recovery when $A$ is unknown}\label{app_supp_recovery_unknown_A}

Regarding the consistent support recovery of $\wh T$, we remark that the results in Section \ref{app_supp_recovery_known_A} continue to hold provided that the anchor words can be consistently estimated. Consistent estimation of the anchor words has been fully established in \cite{TOP}. Also, see, \cite{rechetNMF,arora2013practical} for other procedures of estimating anchor words.

\section{Supplementary results on IMDB data}\label{sec:imdb supp}

In this section we present further details of our analysis in Section \ref{sec:imdb} of the IMDB movie review dataset. We first give Table \ref{table:IMDB top words}, which gives an initial interpretation to each estimated topic based on its anchor words.

\begin{table}[H]
	\centering
	\caption{Anchor words for each topic in the IMDB dataset, along with an initial interpretation of each topic.}
	\label{table:IMDB top words}
	\begin{tabular}{ |c|p{8cm}||l|} 
		\hline \textbf{Topic} &\textbf{Anchor words}& \textbf{Initial interpretation}\\
		\hline 
		1& book, read, version & Book adaptations
		\\
		\hline
		2&  crap, talent & Sentiment\\
		\hline
		3& game & Game-related\\
		\hline
		4& blood, dark, dead, evil, fans, flick, genre, gore, horror, house, killer, sequel, strange & Horror films\\
		\hline
		5&episode & TV Shows\\
		\hline 
		6& history, war & History \& war films\\
		\hline
	\end{tabular}
\end{table}

Table \ref{table:d_top} gives the computed values of the two topic-distance matrices (\ref{doc dists imdb}) and shows that these  distances qualitatively capture the same similarity relationships between the topics. This is despite the fact that $\hDt_W$ incorporates word similarity from pre-trained word embeddings, whereas $\hDt_{TV}$ depends only parameters estimated directly from the IMDB corpus. 

\begin{table}[h!]
	\caption{The two $\wh K\times \wh K$ matrices of distances between topics. }
	\label{table:d_top}
	\begin{center}
		\subfloat[$\hDtb_{TV}$]{
			\begin{tabular}{ |c|c|c|c|c|c|c|}
				\hline
				&\textbf{Topic 1} & \textbf{Topic 2}&\textbf{Topic 3}&\textbf{Topic 4}&\textbf{Topic 5}&\textbf{Topic 6}\\
				
				\hline
				\textbf{Topic 1} &0  &0.14 & 0.22 &0.13 & 0.20 & 0.14 \\
				\hline
				\textbf{Topic 2} & 0.14 & 0 & 0.22 & 0.14 & 0.22 & 0.17\\
				\hline 
				\textbf{Topic 3} & 0.21 & 0.22 & 0 & 0.21 & 0.22 &0.23\\
				\hline
				\textbf{Topic 4} &  0.13 & 0.14 & 0.21 & 0 & 0.21 & 0.14\\
				\hline
				\textbf{Topic 5} & 0.20 & 0.22 & 0.22 & 0.21 & 0 & 0.22\\
				\hline
				\textbf{Topic 6} & 0.14 & 0.17 & 0.23 & 0.14 & 0.22 &0\\
				\hline
		\end{tabular}}
		\vspace{3mm}
		\subfloat[$\hDtb_W$]{
			\begin{tabular}{ |c|c|c|c|c|c|c|}
				\hline
				&\textbf{Topic 1} & \textbf{Topic 2}&\textbf{Topic 3}&\textbf{Topic 4}&\textbf{Topic 5}&\textbf{Topic 6}\\
				
				\hline
				\textbf{Topic 1} &0  &0.10 & 0.16 &0.10 & 0.15 & 0.10 \\
				\hline
				\textbf{Topic 2} & 0.10 & 0 & 0.17 & 0.10 & 0.16 & 0.12\\
				\hline 
				\textbf{Topic 3} & 0.16 & 0.17 & 0 & 0.16 & 0.17 &0.17\\
				\hline
				\textbf{Topic 4} &  0.10 & 0.10 & 0.16 & 0 & 0.16 & 0.10\\
				\hline
				\textbf{Topic 5} & 0.15 & 0.16 & 0.17 & 0.16 & 0 & 0.17\\
				\hline
				\textbf{Topic 6} & 0.10 & 0.12 & 0.17 & 0.10 & 0.17 &0\\
				\hline
		\end{tabular}}
	\end{center}
\end{table}

Finally, we give Table \ref{table:IMDB excerpts}, which gives excerpts of two documents for each topic that are estimated to be exclusively generated from that topic. These excerpts allow for further interpretation of the topics.


\begin{table}
	\caption{Excerpts from documents that are estimated to be exclusively generated from Topics 3 and 5 (formally, documents with $\wh T^{(i)}_k = 1$ for each topic $k$). Excerpts from two separate documents with this property are given for each topic.  The third column gives the ID number in the original dataset \citep{maas-EtAl:2011:ACL-HLT2011} for reference. We find that the ``movie reviews" corresponding to Topics 3 and 5 are in fact reviews of video games and TV shows, respectively.}
	\label{table:IMDB excerpts}
	\begin{tabular}{|l|l|l|p{9cm}|}
		\hline \textbf{Topic} &\textbf{Interpretation}&\textbf{Movie ID}&\textbf{Document excerpt}\\
		\hline 
		\multirow{2}{*}{Topic 1} & Book Adaptations & 37,123 &\textit{This was an OK movie, at best, outside the context of the book. But having read and enjoyed the book quite a bit it was a real disappointment in comparison\ldots } 
		\\ \cline{3-4}
		&  & 15,709 &\textit{This has always been one of my favourite books. I was thrilled when I saw that the book had been made into a movie, for the first time since it was written, over 50 years before\ldots } \\
		\hline
		\multirow{2}{*}{Topic 2} & Negative reviews & 12,445&\textit{This was the most disappointing films I have ever seen recently. And I really hardly believe that people say goods things about this very bottom film!\ldots  }
		\\ \cline{3-4}
		&  & 32,442&\textit{Acting was awful. Photography was awful. Dialogue was awful. Plot was awful. (I'm not being mean here...It really was this bad.)\ldots } \\
		\hline
		\multirow{2}{*}{Topic 3} & Video Games & 23,753&\textit{This game really is worth the ridiculous prices out there. The graphics really are great for the SNES, though the magic spells don't look particularly great\ldots}
		\\ \cline{3-4}
		&  & 12,261&\textit{I remember playing this game at a friend. Watched him play a bit solo until we decided to try play 2 and 2, which we found out how to do\ldots} \\
		\hline
		\multirow{2}{*}{Topic 4} & Horror & 29,114&\textit{After watching such teen horror movies as Cherry Falls and I know what you did last summer, I expected this to be similar\ldots  }
		\\ \cline{3-4}
		&  & 3,448& \textit{Being a HUGE fan of the horror genre, I have come to expect and appreciate cheesey acted, plot-holes galore, bad scripts\ldots} \\
		\hline
		\multirow{2}{*}{Topic 5} & TV Shows & 32,315& \textit{I used to watch this show when I was a little girl\ldots }
		\\ \cline{3-4}
		&  & 10,454&\textit{I'v watched the TV show Hex twice over and I still can not get enough of it. The show is excellent\ldots } \\
		\hline
		\multirow{2}{*}{Topic 6} & War \& History & 6,709&\textit{Carlo Levi, an Italian who fought against the arrival of Fascism in his native Torino, was arrested for his activities\ldots }
		\\ \cline{3-4}
		&  & 26,918&\textit{As directed masterfully by Clint Eastwood, ``Flags of Our Fathers" plays both as a war film and a sensitive human drama\ldots } \\
		\hline
		
	\end{tabular}
\end{table}

\section{Simulations on the estimation of $T_*$ and $\Pi_*$}\label{sec:sims}
In this section we present a simulation study of the estimation of topic proportions $T_*$ and word-distribution $\Pi_*$ to accompany our theoretical analysis in Section \ref{sec_est_T}. 

We perform simulations to study the performance of the MLE in (\ref{def_MLE}) (known $A$) and the estimator in (\ref{def_T_hat}) (unknown $A$) for estimating $T_*$, and compare to the Restricted Least Squares (RLS) estimator, 
\begin{equation}\label{rls}
	\wh T_{\text{rls}} \coloneqq  \min_{T\in \Delta_K}\|X - \wh A T\|_2^2,
\end{equation}
as well as the iterative weighted restricted least squares (IWRLS) estimator, both mentioned in Remark \ref{rem_MOM}. To compute the IWRLS estimator for \textit{known} $A$, we use the following steps (for unknown $A$, we just replace $A$ by $\wh A$ estimated using the Sparse-TOP method of \citet{bing2020optimal}). We use the parameter $\eps = 10^{-8}$ to avoid division by zero, $\delta = 10^{-4}$ as a stopping criterion, and $m_{it}=1000$ as the maximum number of iterations.
\begin{enumerate}
	\item Compute $\wh T_{\text{rls}}$ from (\ref{rls}) and set $\wh T = \wh T_{\text{rls}}$.
	\item Let
	\begin{equation*}
		\wh D = \diag(d_1, \ldots, d_p),\qquad \textrm{with}\quad  d_j = \frac{1}{\sqrt{(A\wh T)_j\vee \eps}}.
	\end{equation*}
	\item Update $\wh T$ as
	\begin{equation*}
		\wh T \leftarrow \argmin_{T \in \Delta_K}\|\wh D(X - A T)\|_2^2.
	\end{equation*}
	\item Repeat Steps 2 and 3 until either the $\ell_1$ distance between $\wh T$ from the current step and the previous step is less than $\delta$, or a maximum of $m_{it}$ iterations have been completed, then take $\wh T_{\rm iwrls} = \wh T$ as the final estimator.
\end{enumerate}

We then compare model-based estimators of $\Pi_*$ based on estimates of $T_*$ to the empirical estimate of $\Pi_*$. In terms of notation, recall that $N$ is the number of words in a document, $n$ is the number of documents in the corpus, $K$ is the number of topics, and $p$ is the dictionary size.

\subsubsection*{Data generating mechanism} For fixed anchor word sets $I_1,\ldots, I_K \subset [p]$, we generate the $p\times K$ matrix $A$ as follows. We set $A_{ik} = K/p$ for all $i\in I_k$ and $k\in [K]$. Draw all entries of non-anchor words from Uniform$(0,1)$, then normalize each sub-column $A_{I^c k}$ to have sum $1 - \sum_{i\in{I_k}}A_{ik}$ where $I^c = [p]\setminus (\cup_k I_k)$. We choose balanced anchor word sets such that $|I_k|=m_{anc}$ for all $k\in[K]$. We choose $m_{anc} = 5$, $K=20$, and $p=1500$ for all experiments in this section.

For fixed support size $s$, we generate $T_*^{(1)},\ldots, T_*^{(n)}$ identically and independently as follows. For each $T_*^{(i)}$, select a subset $S\subset [K]$ with $|S|=s$ by drawing elements from $[K]$ uniformly at random without replacement. Set $[T_*^{(i)}]_{S^c} = 0$ and generate each entry of $[T_*^{(i)}]_{S}$ independently from Uniform$(0,1)$. Finally, normalize $T_*^{(i)}$ so its entries sum to $1$. The result is that $T_*^{(i)}$ has support $s$ for all $i\in [n]$.

For each choice of $A$ and $T_*^{(1)},\ldots,T_*^{(n)}$, we then set $\Pi^{(i)}_* = AT_*^{(i)}$ for $1\le i\le n$ and generate $NX^{(i)}\sim \text{Multinomial}_p(N, \Pi_*^{(i)})$. We report the $\ell_1$ error of the estimation of $T_*$ by each method averaged over all $100$ repetitions of the simulation.

\subsubsection*{Estimation of $T_*$ with known $A$}
We first compare the MLE in (\ref{def_MLE}), RLS, and IWRLS when $A$ is known. In this case we can take $n=1$ and drop the superscript on $T_*$, $X$, and $\Pi_*$. 
To see the impact of sparsity in these methods, we also include a baseline estimator for each method that corresponds to the support of $T_*$ being exactly known. To be precise, let $S_*$ be the support of $T_*$, and $A_{\cdot S_*}$ the $p\times |S_*|$ submatrix of $A$ with columns restricted to the support of $T_*$. The baseline estimators are computed by estimating $T_{S_*}$ using the MLE, RLS, or IWRLS with $A_{\cdot S_*}$ and $X$ as input, and estimating $T_{S_*^c}$ by zeroes.
We plot the $\ell_1$ error of estimation of $T_*$ as a function of $N$ and $s=|S_*|$ in Figure \ref{fig:error T A true}.

\paragraph*{Results} In the left panel Figure \ref{fig:error T A true} we see how the $\ell_1$ error of all estimators decays as the document length increases. On the other hand, we see in the right panel that the error increases as the support size of $T_*$ increases. We observe in both panels  the remarkable feature that the MLE with unknown support has nearly identical risk to the MLE with exactly known support, empirically illustrating Theorem \ref{thm_supp}.  In contrast, the RLS with unknown support performs substantially worse than with known support, illustrating that the RLS does not enjoy the same support recovery properties as the MLE. Comparing the MLE and RLS, both with unknown support, we also observe that the risk of the MLE is uniformly lower than that of the RLS. This gives support to the MLE being a clearly superior estimator of $T_*$.

For the simulation parameters in Figure \ref{fig:error T A true}, IWRLS has approximately equal risk to that of the MLE. This in line with the fact the IWRLS is asymptotically (as $N\to\i$) equivalent to the MLE; see, for instance,  \cite{BFH1975} and \cite{Agresti2012}. However, in Figure \ref{fig:error T A true small N}, we plot the error of the MLE and IWRLS for smaller values of $N$, where the asymptotic equivalence of these two methods breaks down. We see that the MLE has lower error for small $N$, a regime of practical interest corresponding to short documents. Furthermore, from Table \ref{tab:time}, we observe that in the worst case, the IWRLS has a computation time of around two orders of magnitude more than that of the MLE (while the RLS has the lowest computation time of all). Lastly, from Table \ref{tab:max its} we see that for small $N$, the IWLS did not converge within 1000 iterations for a large proportion of runs (we found similar results, with even longer run times, when increasing the maximum allowed iterations for the IWRLS). The increased computation time of the IWRLS relative to the MLE, along with our observation that the error is either equal (for large $N$) or greater (for small $N$), give strong support for the MLE being preferred as an estimator of $T_*$.

\begin{figure}[h!]
	\centering
	\begin{tabular}{cc}
		\hspace{-4mm}
		\includegraphics[width=.50\textwidth]{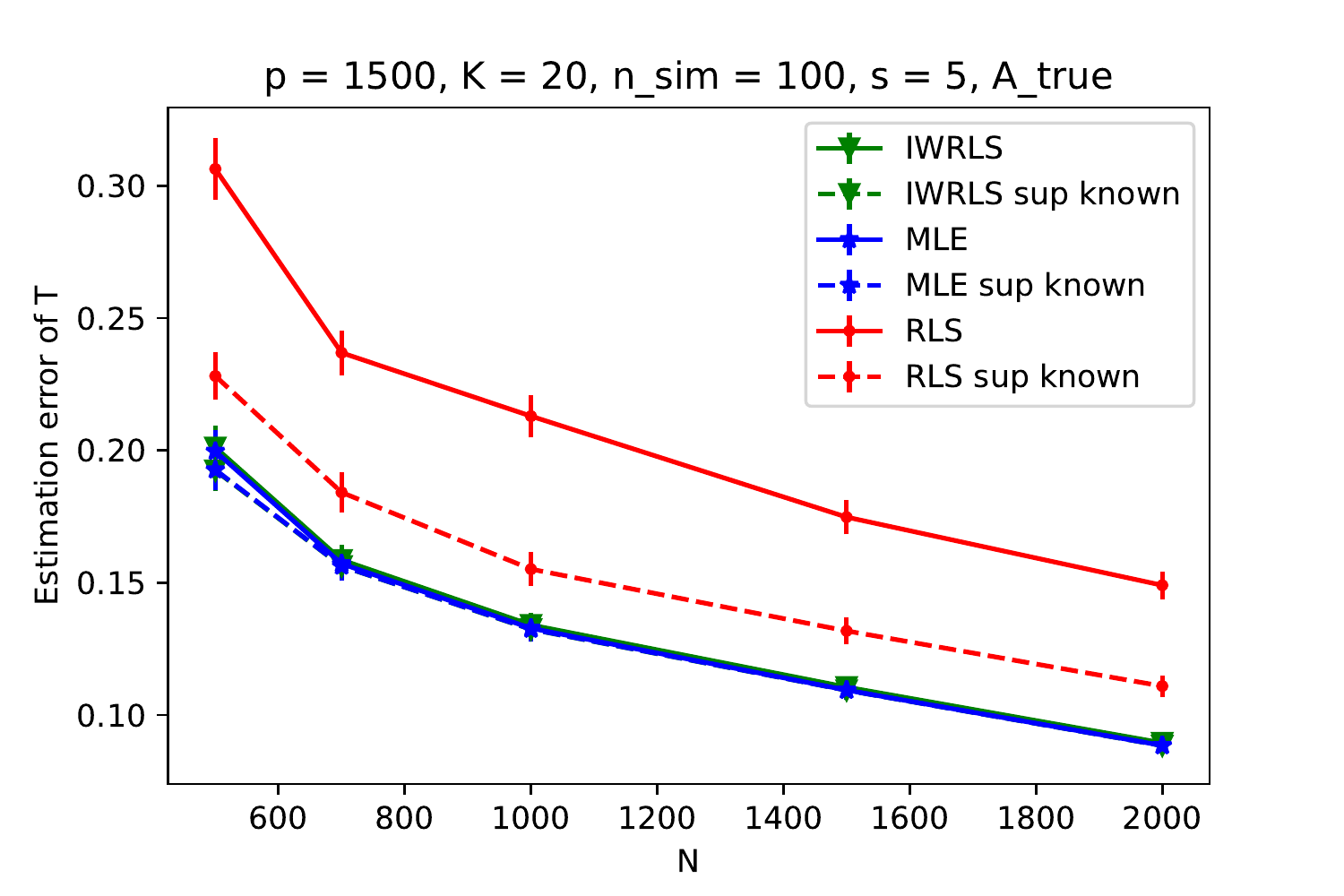}& \includegraphics[width=.50\textwidth]{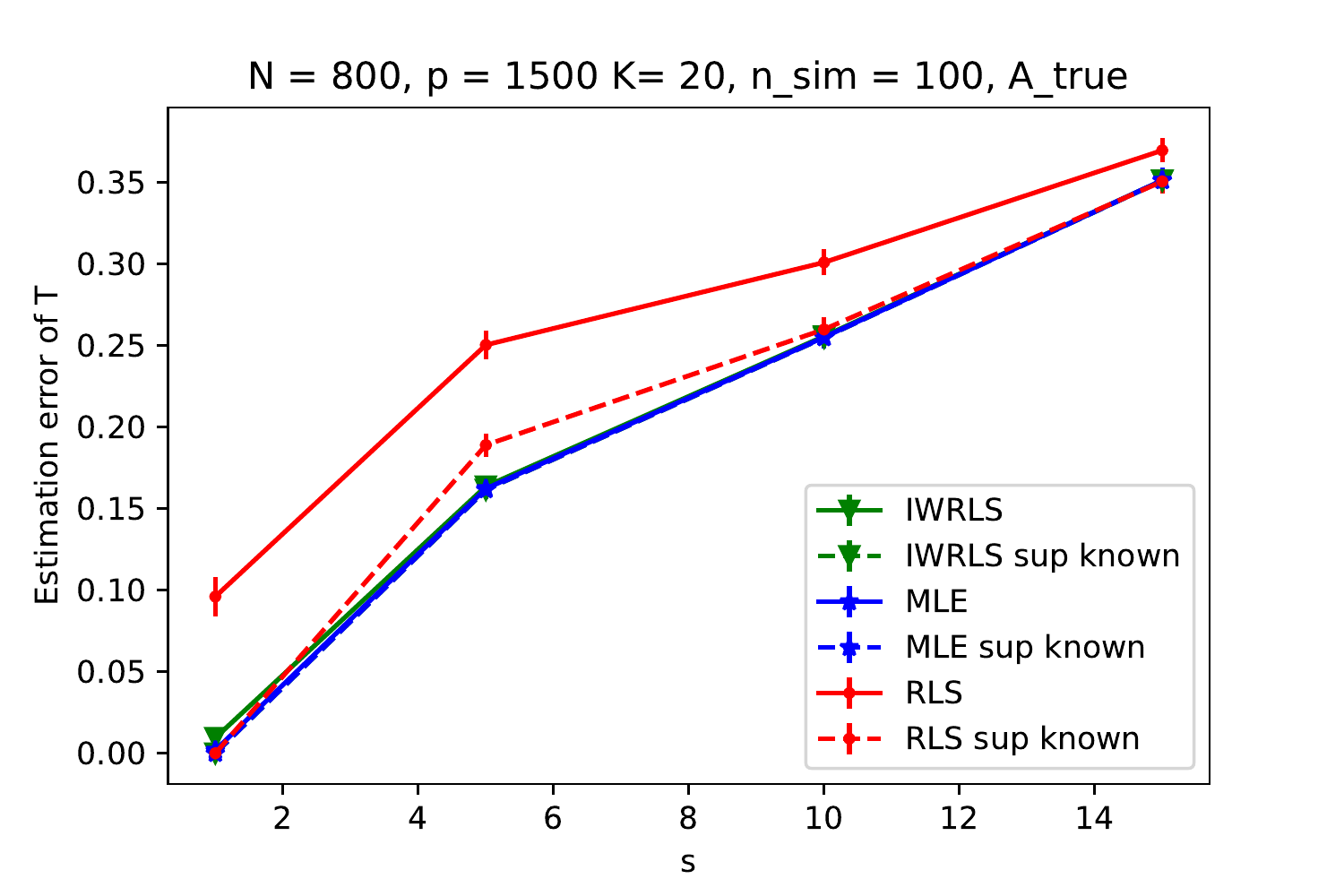}
	\end{tabular}
	\caption{$\ell_1$ error of the estimation of $T_*$ for the MLE, RLS, and IWRLS, as a function of document length $N$ (left) and support size $s$ (right), when $A$ is known. Dashed lines correspond to the predictors when the true support of $T_*$ is known. The error for the MLE and IWRLS are approximately equal for these simulation settings, for both known and unknonwn support.}
	\label{fig:error T A true}
	\vspace{-5mm}
\end{figure}

\begin{figure}[h!]
	\centering
	\begin{tabular}{cc}
		\hspace{-4mm}
		\includegraphics[width=.50\textwidth]{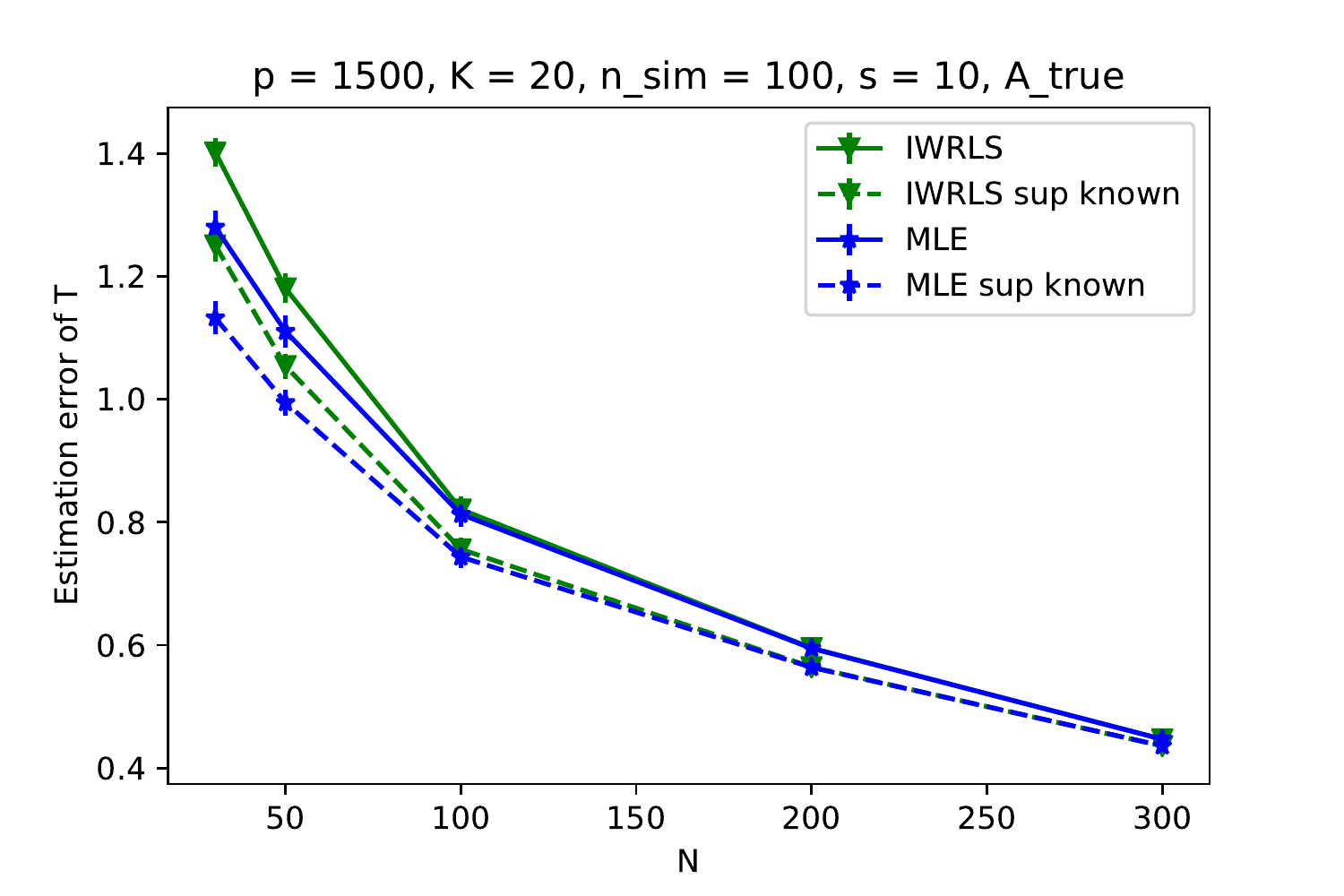}
	\end{tabular}
	\caption{$\ell_1$ error of the estimation of $T_*$ for the MLE and IWRLS, for small values of $N$, when $A$ is known. Dashed lines correspond to the predictors when the true support of $T_*$ is known.}
	\label{fig:error T A true small N}
	\vspace{-5mm}
\end{figure}

\begin{table}
	\centering
	\begin{tabular}{c|c|c|c|c|c}
		Method & $N=30$ & $N=50$ & $N=100$& $N=200$ & $N=300$\\
		\hline
		MLE & 0.68 &0.64& 0.60& 0.57 & 0.54\\
		RLS & 0.08& 0.08& 0.08& 0.08& 0.08\\
		IWRLS & 61.87& 43.31& 31.19&  9.63&  3.44\\
	\end{tabular}
	\caption{Average computation time for each method (in seconds) from the simulation in Figure \ref{fig:error T A true small N}.}
	\label{tab:time}
\end{table}

\begin{table}
	\centering
	\begin{tabular}{c|c|c|c|c}
		$N=30$ & $N=50$ & $N=100$& $N=200$ & $N=300$\\
		\hline
		40\% & 31\%& 23\%& 5\%& 0\%  \\
	\end{tabular}
	\caption{Percentage of IWRLS runs from the simulation in Figure \ref{fig:error T A true small N} that reached the maximum number of iterations (1000) and did not converge.}
	\label{tab:max its}
\end{table}

\subsubsection*{Estimation of $T_*$ with unknown $A$}
We next compare the estimator in (\ref{def_T_hat}), RLS, and IWRLS when $A$ is unknown and estimated by $\wh A$ from the Sparse-TOP method \cite{bing2020optimal}. For all values of $n$, we choose the first document $T_*^{(1)}$ to estimate (this choice is arbitrary, as $T_*^{(i)}$ is drawn from an identical distribution for all $i\in [n]$). In Figure \ref{fig:error T A hat}, we plot the average $\ell_1$ error of estimating $T_*^{(1)}$ as a function of $N$ and $s$. We also include the MLE in (\ref{def_MLE}) with known $A$ for comparison. We refer to the estimator (\ref{def_T_hat}) as MLE-A-hat in the plot.

\paragraph*{Results} Similarly to the case with known $A$, we see in Figure \ref{fig:error T A hat} that the estimator (\ref{def_T_hat}) uniformly outperforms the RLS, and for the (large) values of $N$ in the plot, the IWLS has approximately equal risk to estimator (\ref{def_T_hat}). Comparing the MLE with known and unknown $A$, we observe in the left panel that the impact of not knowing $A$ decreases as the document length increases, which is expected as longer documents improve the estimate $\wh A$. We also observe in the right panel that not knowing $A$ has less impact for $T_*$ that are more sparse. 
\begin{figure}[h!]
	\centering
	\begin{tabular}{ccc}
		\hspace{-4mm}
		\includegraphics[width=.50\textwidth]{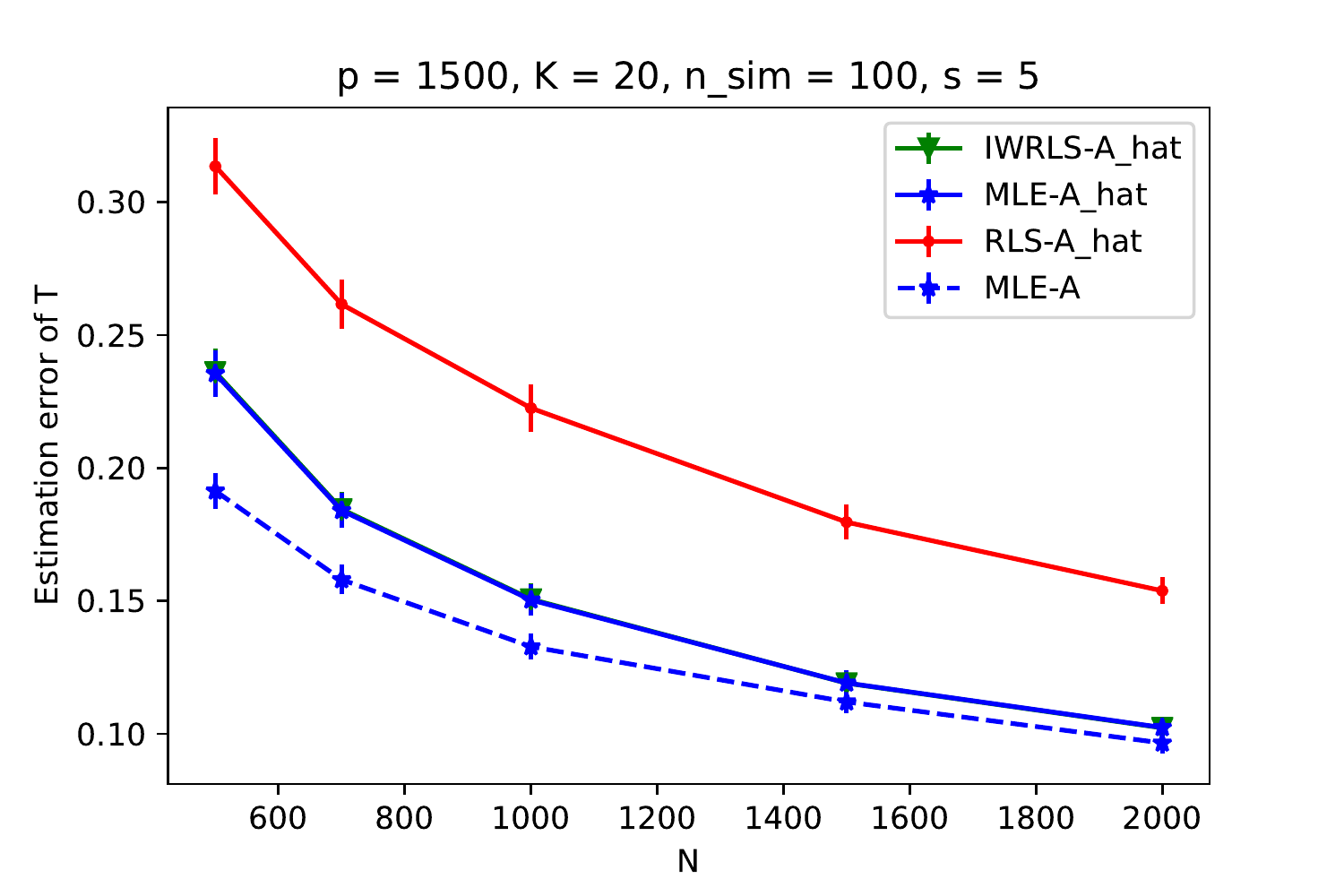} & \includegraphics[width=.50\textwidth]{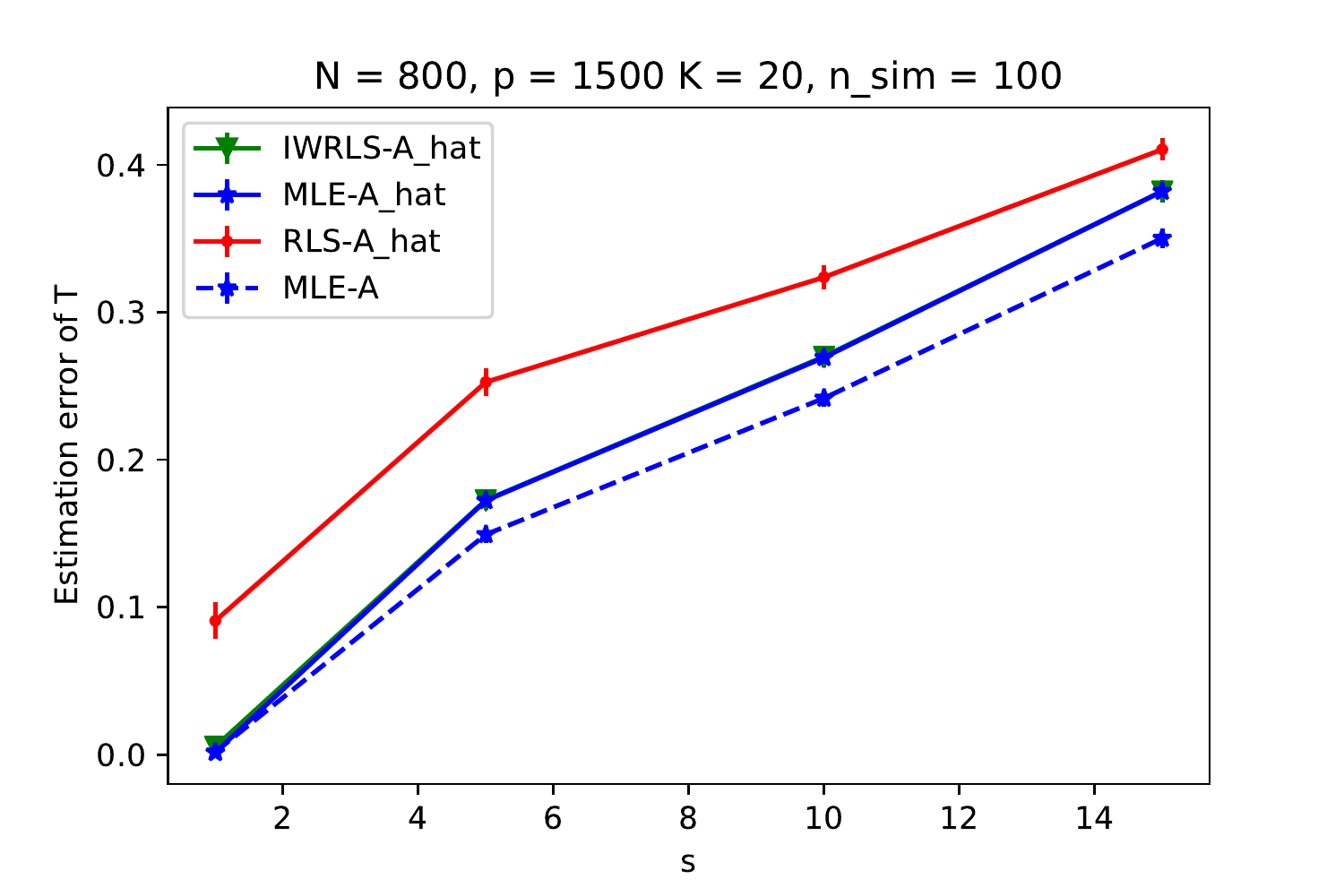}
		\\
	\end{tabular}
	\caption{$\ell_1$ error of the estimation of $T_*$ for the MLE, RLS, and IWRLS, as a function of document length $N$ (left) and support size $s$ (right). Solid lines correspond to estimators using $\wh A$, and the dashed line corresponds to the MLE with known $A$. The error for the MLE and IWRLS (both for unknown $A$) are approximately equal in these plots.}
	\label{fig:error T A hat}
\end{figure}

\subsubsection*{Estimation of $\Pi_*$}

We compare the three estimators of $\Pi_*$ presented in Section \ref{sec:est pi}: the empirical estimate $\wh\Pi$, and the model-based estimators $\wt \Pi_{A} = A \mle$ and $\wh \Pi_{\wh A} = \wh A\wh T$.
Setting $n=1000$, we repeat the simulation $100$ times and plot the average $\ell_1$ error in estimating the first document $\Pi^{(1)}_*$ as a function of document length $N$ in the top left panel of Figure \ref{fig:error pi}. Note that $\wh\Pi^{(1)}$ 
is a function only of the first document vector $X^{(1)}$, and ignores the other $n-1$ documents in the corpus. In contrast, the model-based estimator with unknown $A$, $\wt \Pi^{(1)}_{\wh A}$, uses all $n$ documents in the corpus via the estimation of $\wh A$. We drop the superscript $1$ in the remainder of this discussion for ease of notation.

Recall the definitions $\bar J \coloneqq \{j: \Pi_{*j}>0\}$ and $J \coloneqq \{j:X_j>0\}$ from Section \ref{sec:est pi}. As discussed in that section, a particular advantage of the model-based estimators $\wt \Pi_A$ and $\wt \Pi_{\wh A}$ over the empirical estimate $\wh \Pi$ is their ability to non-trivially estimate non-zero cell probabilities with zero counts ($\Pi_{*j}$ with $j\in \bar J\setminus J$), while still estimating the zero cell probabilities $j\in \bar J^c$ nearly as well as $\wh\Pi$. 
We here conduct a simulation to empirically study the ability of each method to estimate these two classes of cell probabilities.  In each simulation run, we compute for each estimator $\Pi^{\text{est}}$ among $\wh \Pi$, $\wt\Pi_A$, and $\wt \Pi_{\wh A}$ the quantities $\sum_{j\in \bar J\setminus J} | \Pi^{\text{est}}_j - \Pi_{*j}|$ and $\sum_{j\in \bar J^c} | \Pi^{\text{est}}_j - \Pi_{*j}|$; we plot their average values over 100 runs as a function of $N$ in the top right and bottom panels of Figure \ref{fig:error pi}, respectively.

\begin{figure}[h!]
	\centering
	\begin{tabular}{ccc}
		\hspace{-4mm}
		\includegraphics[width=.50\textwidth]{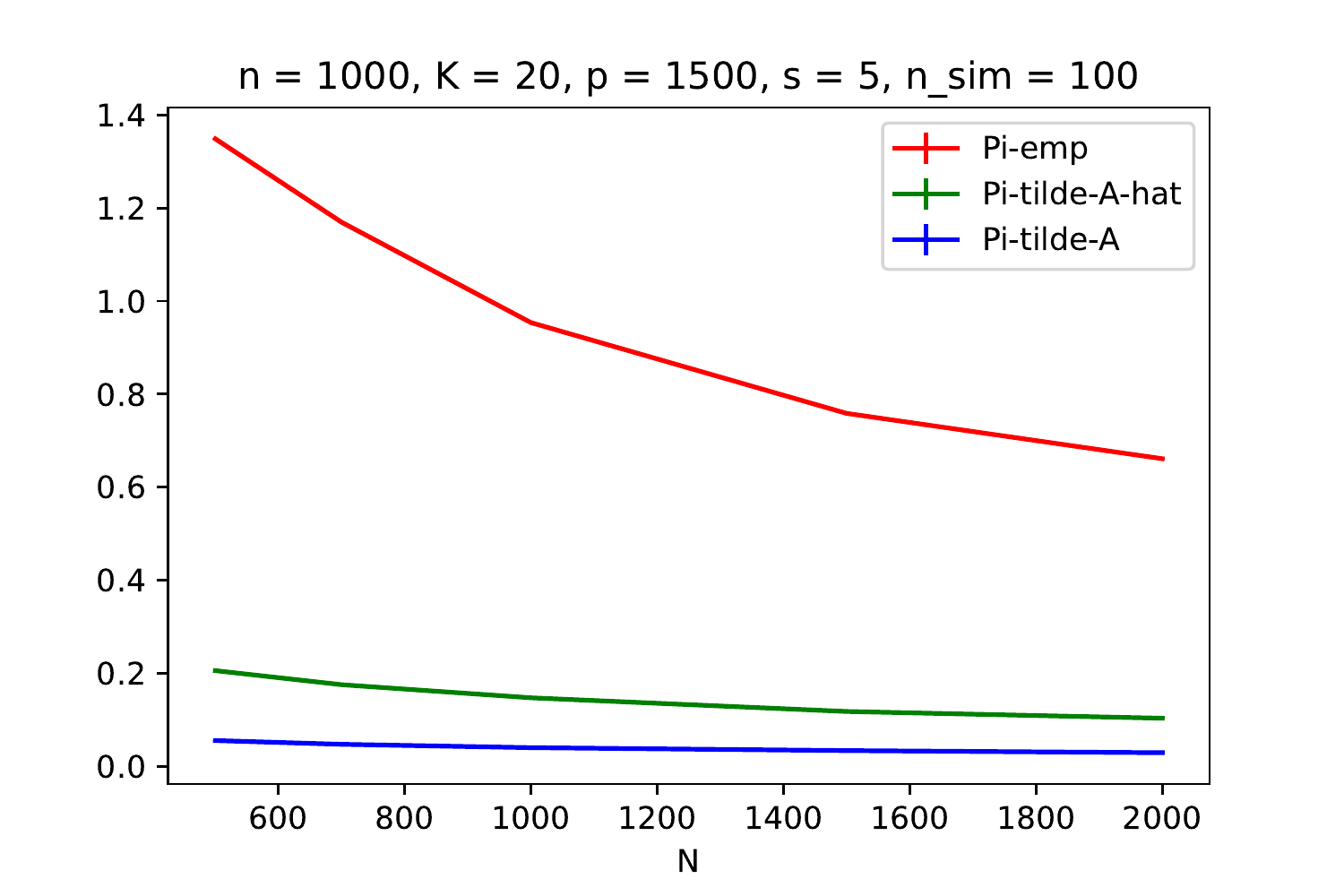} & \includegraphics[width=.50\textwidth]{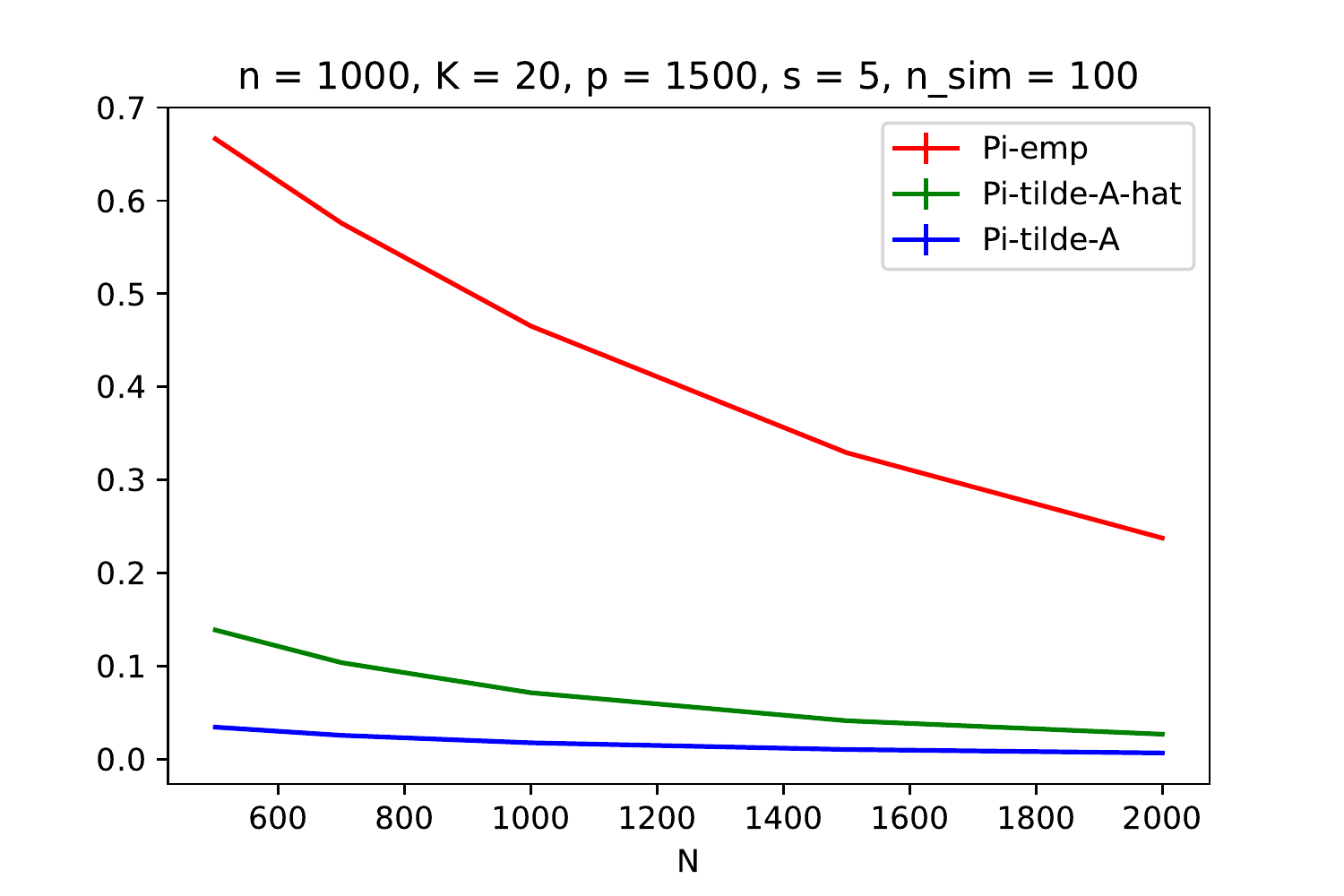}
	\end{tabular}
	\includegraphics[width=.50\textwidth]{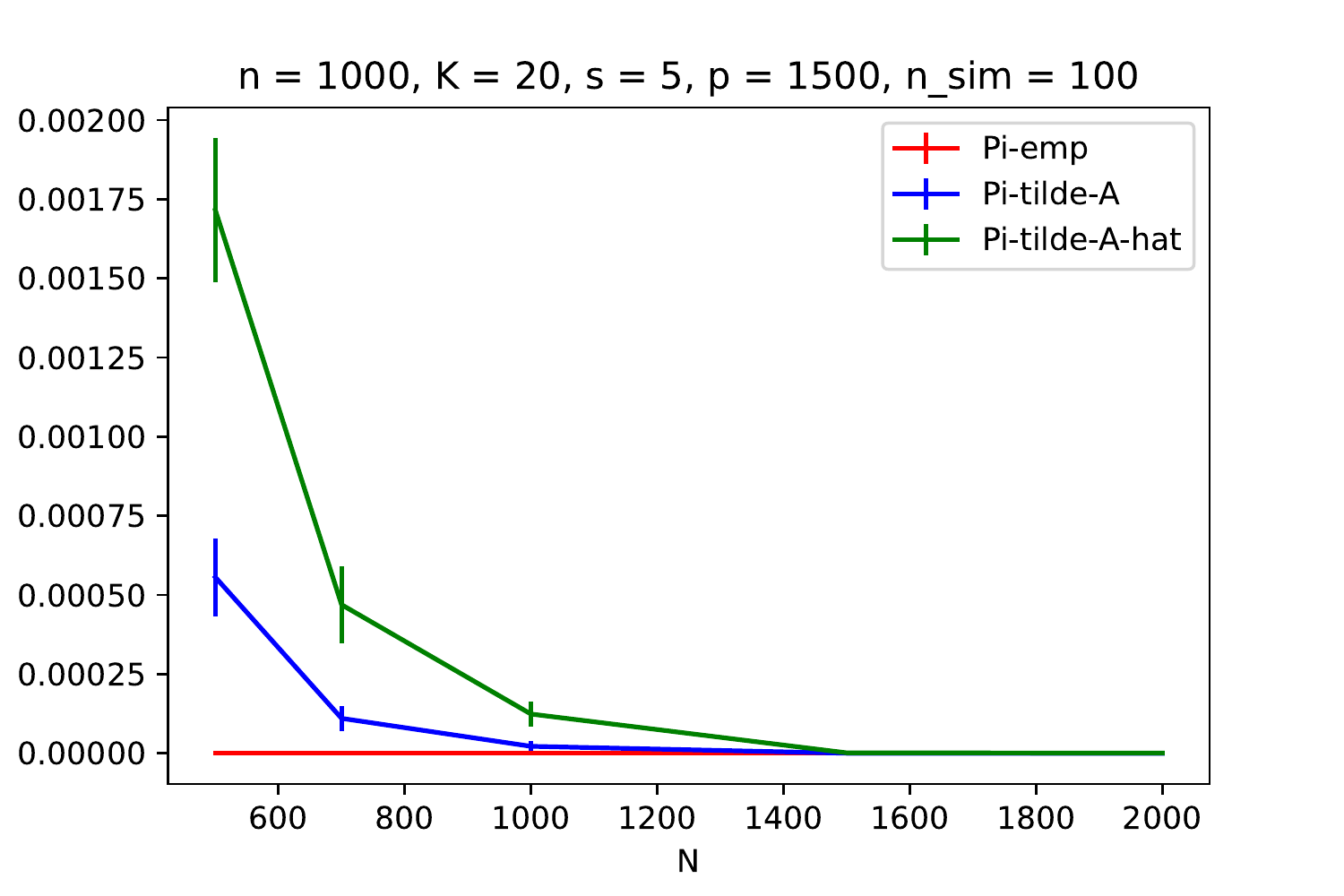} 
	\caption{Top left: $\ell_1$ error of the estimation of $\Pi_*$ as a function of document length $N$. Top right: Error in estimating the cell probabilities $\Pi_{*j}$ with $j\in \bar J\setminus J$. Bottom: Error in estimating $\Pi_{*j}$ with $j\in \bar J^c$. Error bars are present in top plots but too small to observe.}
	\label{fig:error pi}
\end{figure}

\paragraph*{Results} We observe from the left panel of Figure \ref{fig:error pi} that while the error of all three estimators decays with $N$, the error  of the empirical estimator is substantially larger than that of the model-based estimator $\wt \Pi_{\wh A}$, which is in turn larger than the error of model-based estimator with known $A$, $\wt \Pi_A$, while being very close to it.  This demonstrates the basic motivation of the model-based estimation approach: by borrowing statistical strength from across the full corpus, $\wt \Pi_{ A}$ and $\wt \Pi_{\wh A}$ provide a far superior estimate of the frequencies for an individual document. The difference between $\wt \Pi_{ A}$ and $\wt \Pi_{\wh A}$ on the other hand reflects the effect of estimating $A$. 

In the top right panel of Figure \ref{fig:error pi}, we verify that the two model-based estimators are able to estimate the non-zero cell probabilities with zero counts ($\Pi_{*j}$ with $j\in \bar J\setminus J$) much better than the trivial estimate of $\wh \Pi_{j}=0$. We note for clarity that for $\Pi^{\text{est}} = \wh\Pi$, $\sum_{j\in \bar J\setminus J}| \Pi^{\text{est}}_j - \Pi_{*j}|$ reduces to $ \sum_{j:\ \Pi_{*j}>0,\ X_j = 0} \Pi_{*j}$. While we see that this quantity decreases with $N$ from the red line in the top right panel of Figure \ref{fig:error pi}, this is simply due to the fact that $|\{j:X_j=0\}|$ decreases with $N$.

Lastly, in the bottom panel of Figure \ref{fig:error pi}, we see that while for small $N$ ($N\le 1500$ for $\wt\Pi_{\wh A}$) the error in estimating the zero cell probabilities by the model-based estimators is non-zero, it is several orders of magnitude smaller than the overall $\ell_1$ error in the top left panel, and in any case quickly decays to zero as $N$ increases. As expected, $\wh \Pi_{j} =0$ for all $j\in \bar J$, so the error for $\wh \Pi$ is exactly zero in this bottom panel.

In summary, by borrowing statistical strength across the corpus of $n$ documents, the two model-based estimators $\wt\Pi_{A}$ and $\wt \Pi_{\wh A}$ perform substantially better than the empirical estimator at estimating $\Pi_*$ in $\ell_1$ error and estimating non-zero cell probabilities with zero counts, while still having nearly the same performance as the empirical estimator at estimating the zero cell probabilities.


\section{Semi-synthetic simulations to compare document-distance estimation rates}\label{sec:semi syn sim}

We perform semi-synthetic simulations to empirically study the rate of estimation of the topic-based document distance (\ref{top dist pop}) for the choice (\ref{def_D_topic_TV}) of $\Dt$, by the estimator (\ref{def_D_topic_TV_hat}). We also ran the same simulations for the choice  (\ref{d top w}) of $\Dt$ with the estimator (\ref{def_D_topic_W_hat}) and found similar results, which we do not report here due to space limitations.

\subsubsection*{Data and preprocessing} We work with the NIPS bag-of-words dataset \citep{Dua:2019}. We preprocess the data by removing stop words, removing documents with less than $150$ words, and removing words that appear in less than $150$ documents. We are left with $1490$ documents and dictionary size $p = 1270$.
From this data we estimate a loading matrix $A_0$ using the Sparse-TOP algorithm \citep{bing2020optimal} and find $K=21$ topics. We then treat this estimated $A_0$ as our ground truth for semi-synthetic experiments.


\subsubsection*{Semi-synthetic data generation} We generate topic distributions $T_*^{(1)}, \ldots, T_*^{(n)}$ with $K=21$ exactly following the procedure in Section \ref{sec:sims}. In particular, $T_*^{(1)}, \ldots, T_*^{(n)}$ all have the same support size, which we denote as $s$. We choose $n= 2000$ for all simulations. For each $i\in [n]$, we set $\Pi^{(i)}_* = A_0 T^{(i)}_*$ and draw $X^{(i)}\sim \text{Mutlinomial}_p(N,\Pi^{(i)})$.\\

For each simulation, we form the estimate $\wh A$ using Sparse-TOP \cite{bing2020optimal}, and form estimates $\wh T^{(1)}$, $\wh T^{(2)}$ of the topic distributions of the first two documents using the MLE (\ref{def_T_hat}). From $\wh A$, we compute the estimated topic-distance metric 
\[\hDt_{TV}(k,l) = \frac{1}{2}\|\wh A_{\cdot k} - \wh A_{\cdot l}\|_1\quad \forall k,l\in \{1,2,\ldots,21\},\]
and its population counterpart with $\wh A$ replaced by $A_0$. We then compute the error
\begin{equation}\label{w err to plot}
	|W_1(\wh T^{(1)}, \wh T^{(2)}; \hDt_{TV}) - W_1( T_*^{(1)}, T_*^{(2)}; {\Dt_{TV}})|.
\end{equation}
We repeat this simulation $n_{sim} = 50$ times for different values of $N$ and $s$, and plot the average error in Figure \ref{fig:nips}.

\paragraph*{Results}  We see from Figure \ref{fig:nips} that the error (\ref{w err to plot}) grows significantly as the support size $s$ of $T_*^{(1)}$ and $T_*^{(2)}$ increases. This is can be understood by the fact that the error in estimating $T_*^{(1)}$ and $T_*^{(2)}$ also increases with $s$; recall Figure \ref{fig:error T A hat} for an empirical demonstration of this. For all values of $s$, we observe the error decaying as $N$ increases.

\begin{figure}[h!]
	\centering
	\begin{tabular}{cc}
		\hspace{-4mm}
		\includegraphics[width=.50\textwidth]{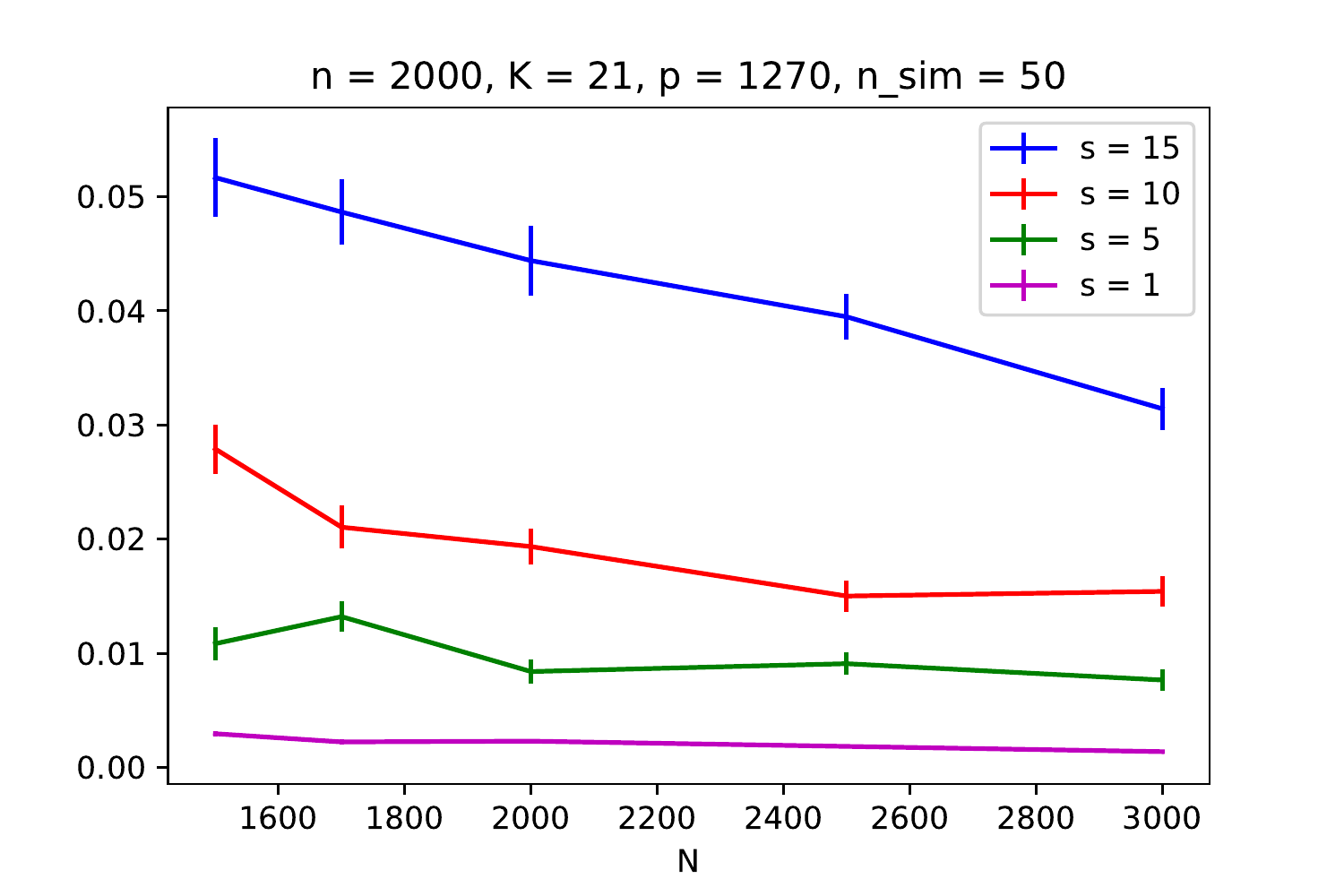}
	\end{tabular}
	\caption{Error (\ref{w err to plot}) as a function of $N$, for different values of support size $s$ of the synthetically generated topic distributions $T_*^{(1)}$ and $T_*^{(2)}$.}
	\label{fig:nips}
\end{figure}

\section{Proofs for Section \ref{sec_est_T_known_A}: Estimation with known $A$}\label{app_proof_T_known_A}

Throughout the proofs, we will suppress the subscript $*$ for notational simplicity. Correspondingly, we write $S_T = S_*$ to denote its dependency on $T$.

\subsection{Proof of Theorem \ref{thm_mle}: The general finite sample bounds of the $\ell_1$ norm convergence rate of the MLE}\label{app_proof_thm_mle}

Recall $\eps_j$ is defined in (\ref{def_eps_j}).
Define the event
\begin{equation}\label{def_event}
	\E \coloneqq \bigcap_{j=1}^p\left\{
	|X_j - \Pi_j| \le \eps_j
	\right\}
\end{equation}
which, according to Lemma \ref{lem_basic} in Appendix \ref{app_tech_lemma}, holds with probability at least $1 - 2p^{-1}.$  On the event $\E$, 
we have 
\[
\uJ \subseteq J \subseteq \oJ.
\]
Indeed, $\uJ \subseteq J$ follows by noting that, for any $j\in \uJ$, $X_j \ge \Pi_j - |X_j-\Pi_j|  > \eps_j$. The other direction $J \subseteq \oJ$ holds trivially since $\Pi_j = 0$ implies $X_j = 0$ for all $j\in [p]$. We work on the event $\E$ for the remainder of the proof.\\

For notational simplicity, we write $\wh T = \mle$. Recall that 
\begin{align*}
	\wh T &\coloneqq   \argmax_{T\in \D_K} N\sum_{j\in J}X_j \log\left(A_{j\cdot}^\T T\right).
\end{align*}
From the KKT conditions of this optimization problem we have 
\begin{align}\label{eq_kkt_1}
	&N\sum_{j\in J} X_j {A_{j\cdot} \over A_{j\cdot}^\T \wh T} + \lambda  + \mu \1_K = 0,\\\label{eq_kkt_2}
	& \lambda_k \ge 0, \quad \lambda_k \wh T_k = 0, ~ \forall k\in [K],\quad  \1_K^\T\wh T = 1.
\end{align}
After taking the inner-product with $\wh T$ on both sides of (\ref{eq_kkt_1}), we get 
\[
\mu = -N\sum_{j\in J} X_j = -N.
\]
Plugging this into (\ref{eq_kkt_1}) gives the expression
\[
N\sum_{j\in J} X_j {A_{j\cdot} \over A_{j\cdot}^\T \wh T} + \lambda  =  N~\1_K.
\]
Next, we take the inner-product on both sides with $\D \coloneqq \wh T - T$ and use the fact that $\1_K^\T \D = 0$ to obtain 
\begin{align*}
	N\sum_{j\in J} X_j {A_{j\cdot}^\T \D\over A_{j\cdot}^\T \wh T}  + \lambda^\T \D = 0.
\end{align*}
By adding and subtracting terms, we have 
\begin{align*}
	0 & = N\sum_{j\in J} X_j  \left({A_{j\cdot}^\T \D\over A_{j\cdot}^\T \wh T}   -{A_{j\cdot}^\T \D\over A_{j\cdot}^\T T}  \right) +N\sum_{j\in J} {X_j \over A_{j\cdot}^\T T} A_{j\cdot}^\T \D + \lambda^\T \D\\
	&= N\sum_{j\in J} X_j  \left({A_{j\cdot}^\T \D\over A_{j\cdot}^\T \wh T}   -{A_{j\cdot}^\T \D\over A_{j\cdot}^\T T}  \right) +N\sum_{j\in \oJ} {X_j \over A_{j\cdot}^\T T} A_{j\cdot}^\T \D + \lambda^\T \D\\
	& = N\sum_{j\in J} X_j  \left({A_{j\cdot}^\T \D\over A_{j\cdot}^\T \wh T}   -{A_{j\cdot}^\T \D\over A_{j\cdot}^\T T}  \right)+N\sum_{j\in \oJ} \left(X_j - A_{j\cdot}^\T T\right) {A_{j\cdot}^\T \D\over A_{j\cdot}^\T T} + N\sum_{j\in\oJ}  A_{j\cdot}^\T \D + \lambda^\T \D.
\end{align*}
In the second equality, we used $\Pi_j = A_{j\cdot}^\T T > 0$ for $j\in \oJ$ and $X_j = 0$ for $j\in \oJ\setminus J$. 
Since 
\[
{A_{j\cdot}^\T \D\over A_{j\cdot}^\T \wh T}   -{A_{j\cdot}^\T \D\over A_{j\cdot}^\T T}  = -{(A_{j\cdot}^\T \D)^2\over A_{j\cdot}^\T \wh T \cdot  A_{j\cdot}^\T T},
\]
we conclude 
\begin{align}\label{display_key}\nonumber
	N\sum_{j\in J} {X_j \over   A_{j\cdot}^\T T}{(A_{j\cdot}^\T \D)^2\over A_{j\cdot}^\T \wh T} &= N\sum_{j\in\oJ} \left(X_j - A_{j\cdot}^\T T\right) {A_{j\cdot}^\T \D\over A_{j\cdot}^\T T} + N\sum_{j\in\oJ}  A_{j\cdot}^\T \D + \lambda^\T \D\\
	&\le  N\sum_{j\in\oJ} \left(X_j - A_{j\cdot}^\T T\right) {A_{j\cdot}^\T \D\over A_{j\cdot}^\T T} +N\sum_{j\in\oJ}  A_{j\cdot}^\T \D 
\end{align}
by using
$\lambda^\T \D= -\lambda^\T T\le  0$ from (\ref{eq_kkt_2}) in the last step.			
For the left hand side in (\ref{display_key}), use $\uJ \subseteq J$ to obtain
\begin{align*}
	\sum_{j\in J} {X_j \over   A_{j\cdot}^\T T}{(A_{j\cdot}^\T \D)^2\over A_{j\cdot}^\T \wh T} & \ge \sum_{j\in  \uJ} {X_j \over   A_{j\cdot}^\T T}{(A_{j\cdot}^\T \D)^2\over A_{j\cdot}^\T \wh T} \ge \min_{j\in  \uJ} {X_j \over A_{j\cdot}^\T T} \sum_{j\in  \uJ}{(A_{j\cdot}^\T \D)^2\over A_{j\cdot}^\T \wh T}.
\end{align*}
Since $\sum_{j=1}^p A_{j\cdot}^\T \wh T = 1$, we further observe that 
\begin{align*}
	\sum_{j\in  \uJ}{(A_{j\cdot}^\T \D)^2\over A_{j\cdot}^\T \wh T} &= \sum_{j\in  \uJ}{(A_{j\cdot}^\T \D)^2\over A_{j\cdot}^\T \wh T}\left(
	\sum_{j\in  \uJ} A_{j\cdot}^\T \wh T + \sum_{j\in \uJ^c} A_{j\cdot}^\T \wh T
	\right)\\
	&\ge \sum_{j\in  \uJ}{(A_{j\cdot}^\T \D)^2\over A_{j\cdot}^\T \wh T}
	\sum_{j\in  \uJ} A_{j\cdot}^\T \wh T \\
	&\ge 	 \left(\sum_{j\in  \uJ}|A_{j\cdot}^\T \D|\right)^2\\
	&\ge \kappa^2({A_{\uJ}}, s) \|\D\|_1^2.
\end{align*}
Here we use the Cauchy-Schwarz inequality in the third line and the definition  (\ref{def_kappa_A}) of the
$\ell_1\to\ell_1$ condition number $\kappa(A_{\uJ},s)$ together with $\D\in\cC(S_T)$ in the last line. From the inequality
\[
{X_j \over A_{j\cdot}^\T T} \ge  {\Pi_j - |X_j - \Pi_j|\over \Pi_j} \ge 1 - {\eps_j \over \Pi_j} \ge {1 \over 2} \qquad \forall j\in \uJ,
\]
we can now conclude 
\begin{align*}
	\sum_{j\in J} {X_j \over   A_{j\cdot}^\T T}{(A_{j\cdot}^\T \D)^2\over A_{j\cdot}^\T \wh T} 
	~ \ge  ~ 
	{1 \over2}\kappa^2(A_{\uJ},s) \|\D\|_1^2.
\end{align*}
It remains to bound from above the right-hand side 
\[ N\sum_{j\in\oJ} \left(X_j - A_{j\cdot}^\T T\right) {A_{j\cdot}^\T \D\over A_{j\cdot}^\T T} +N\sum_{j\in\oJ}  A_{j\cdot}^\T \D \]
of (\ref{display_key}). The identity $\sum_{j=1}^pA_{j\cdot}^\T  \Delta = \1_K^\T \D = 0$ implies
\begin{align}\label{display_ADelta}
	\sum_{j\in\oJ}  A_{j\cdot}^\T \D =
	-\sum_{j\not\in\oJ}  A_{j\cdot}^\T \D 
	\le \sum_{j\not\in\oJ}A_{j\cdot}^\T T 
	= \sum_{j\not\in\oJ}\Pi_j = 0.
\end{align} 
The  last equality uses the definition of $\oJ$. The inequality $u^\T v \le \|u\|_1 \|v\|_\i$ gives 
\[
\sum_{j\in\oJ} \left(X_j - A_{j\cdot}^\T T\right) {A_{j\cdot}^\T \D\over A_{j\cdot}^\T T} \le   \|\Delta\|_1 \max_{k\in [K]}\left|\sum_{j\in\oJ} \left(X_j - A_{j\cdot}^\T T\right) {A_{jk}\over A_{j\cdot}^\T T}\right|.
\]
By invoking Lemma \ref{lem_oracle_error} in Appendix \ref{app_tech_lemma} with a union bound over $k\in [K]$ to bound the above term, we conclude that, for any $t> 0$,
\begin{align*}
	{1\over 2} \kappa^2(A_{\uJ},s)
	\|\D\|_1 \le & 	\sqrt{2\rho\log (K/t) \over N} + {2\rho \log (K/t)\over 3N}
\end{align*}
with probability $1-2t$.
The proof is complete. \qed

\subsection{Proof of Theorem \ref{thm_mle_fast}: Fast rates of the MLE} \label{app:thm_mle_fast}
To prove Theorem \ref{thm_mle_fast}, we first work on the event under which Theorem \ref{thm_mle} holds, that is, 
\[
\|\wh T_{\rm mle} - T\|_1 ~ \le   ~  {2 \over \kappa^{2}(A_{\uJ}, s)}
\left\{\sqrt{2\rho\log (K/\epsilon) \over N} + {2\rho \log (K/\epsilon)\over N}\right\}.
\]
We write $\wh T = \mle$ for notational ease for the remainder of the proof.
Condition (\ref{cond_N_sparse}) together with $\rho \le (1\vee \xi) / \Tm$ then ensures that 
\[
\max_{j\in \oJ}{|A_{j\cdot}^\T (\wh T - T)|\over A_{j\cdot}^\T T} \le \rho \|\mle - T\|_1 \le c
\]
for some sufficiently small constant $c>0$.
As a result, we can deduce  
\begin{equation}\label{bd_A_T_hat}
	(1-c)A_{j\cdot}^\T T ~ \le~  
	A_{j\cdot}^\T \wh T 
	~ \le ~  (1 + c)A_{j\cdot}^\T T,\quad \forall j\in \oJ.
\end{equation}
Recall from (\ref{display_key}) that
\begin{align*}
	\sum_{j\in J} {X_j \over  A_{j\cdot}^\T T}{(A_{j\cdot}^\T \D)^2\over  A_{j\cdot}^\T \wh T} 
	&\le  \sum_{j\in\oJ} \left(X_j - A_{j\cdot}^\T T\right) {A_{j\cdot}^\T \D\over A_{j\cdot}^\T T} +\sum_{j\in\oJ} A_{j\cdot}^\T \D.
\end{align*}
By (\ref{bd_A_T_hat}), $A_{j\cdot}^\T \wh T \ge (1-c)\Pi_j > 0$ for all $j\in \oJ\setminus J$. Together with $X_j = 0$ for all $j\in \oJ\setminus J$, the above display implies  
\[
\sum_{j\in \oJ} {X_j \over  A_{j\cdot}^\T T}{(A_{j\cdot}^\T \D)^2\over  A_{j\cdot}^\T \wh T} 
\le  \sum_{j\in\oJ} \left(X_j - A_{j\cdot}^\T T\right) {A_{j\cdot}^\T \D\over A_{j\cdot}^\T T} +\sum_{j\in\oJ} A_{j\cdot}^\T \D,
\]
which, by (\ref{bd_A_T_hat}) again, further implies 
\[
{1\over 1+c}\sum_{j\in\oJ} {X_j \over  (A_{j\cdot}^\T T)^2}(A_{j\cdot}^\T \D)^2 
\le  \sum_{j\in\oJ} \left(X_j - A_{j\cdot}^\T T\right) {A_{j\cdot}^\T \D\over A_{j\cdot}^\T T}.
\]
Define  
\begin{equation}\label{def_I_I_hat}
	H = \sum_{j\in \oJ}{1 \over \Pi_j}A_{j\cdot}A_{j\cdot}^\T,\qquad \wh H = \sum_{j\in \oJ}{X_j\over \Pi_j^2}A_{j\cdot}A_{j\cdot}^\T.
\end{equation}
Notice that, for any $v\in \R^K$,
\begin{align}\label{kappa_I}
	v^\T H v &=    \sum_{j\in \oJ}{(A_{j\cdot}^\T v)^2 \over \Pi_j} =  \sum_{j\in \oJ}{(A_{j\cdot}^\T v)^2 \over \Pi_j}\sum_{j\in \oJ}\Pi_j \ge \|A_{\oJ}v\|_1^2 \ge \kappa^2(A_{\oJ},K)\|v\|_1^2.
\end{align}
The first inequality uses the Cauchy-Schwarz inequality. Condition (\ref{cond_N_sparse}) implies $\kappa(A_{\oJ},K)>0$, hence $H$ is invertible. 
We thus have 
\[
{\Delta^\T \wh H \Delta \over 1+c}
\le  \left\|\sum_{j\in\oJ} \left(X_j - A_{j\cdot}^\T T\right) { H^{-1/2}A_{j\cdot}\over A_{j\cdot}^\T T}\right\|_2 \|H^{1/2}\D\|_2.
\]
By  invoking Lemma  \ref{lem_oracle_error_whitening} with $t = 4 K \log(5)$ and Lemma \ref{lem_I_deviation} in Appendix \ref{app_tech_lemma} concludes 
\begin{align}\label{bd_IDelta_ell_2}\nonumber
	\|H^{1/2}\D\|_2^2 &\lesssim   \|H^{1/2}\D\|_2 \left(\sqrt{K\over N} + {(1\vee \xi) \over 3\kappa(A_{\oJ},K) \Tm}\cdot {K\over N}\right)\\
	&\lesssim \|H^{1/2}\D\|_2 \sqrt{K\over N} & \textrm{by (\ref{cond_N_sparse})}
\end{align}
with probability $1 - 2K^{-1}-2e^{-K}$. Since (\ref{kappa_I}) also implies 
\begin{equation}\label{bd_Delta_1}
	\|H^{1/2}\Delta\|_2 \ge \kappa(A_{\oJ}, s)\|\Delta\|_1,
\end{equation}
by using $\Delta \in \cC(S_T)$, 
the result follows. The proof is complete.

\subsection{Proof of Corollary \ref{cor_mle_sparse_supp}: Fast rates of the MLE when it is sparse}

On the event $\E_{\supp}$, for any $T\in \cT(s)$ with $|S_T| = s$, we observe 
\[
[\mle]_{S_T} = \argmax_{T\in \Delta_s} N\sum_{j\in J}X_j \log\left(A_{jS_T}^\T T_{S_T}\right),\qquad [\mle]_{S_T^c} = 0.
\]
Since $\|\mle - T\|_1 = \|[\mle - T]_{S_T}\|_1$, the result follows immediately from Corollary \ref{cor_mle_fast_dense} with $K = s$. If we take $\log(s\vee n)$ instead of $\log(s)$ in (\ref{cond_N_dense}) and take $s\log(1/\epsilon)$ instead of $\log (s)$ in the bound, the resulting probability tail becomes $1-2p^{-1}-4(s\vee n)^{-1}-2\epsilon^{s}$.

\subsection{Proof of Theorem \ref{thm_supp}: One-sided sparsity recovery of the MLE}\label{app_proof_supp}

For any $T$ with $\supp(T) = S_T$, our proof of $\supp(\mle) \subseteq \supp(T)$  consists of two parts: 
\begin{enumerate}
	\item[(i)] we show that there exists an optimal solution $\wt T$ to (\ref{def_MLE}) such that $\supp(\wt T) \subseteq \supp(T)$;
	\item[(ii)] we show that if there exists any optimal solution $\bar T$ to (\ref{def_MLE}) that is different from $\wt T$, then $\supp(\bar T) \subseteq \supp(T)$. 
\end{enumerate}

\paragraph*{Proof of step (i)} Our proof of step (i) uses the primal-dual witness approach by first constructing an oracle estimator $\wt T$ with $\supp(\wt T)\subseteq \supp(T)$, and then proving that $\wt T$ is an optimal solution. 

Towards this end, we first notice that any pair $(\wh T, \lambda, \mu)$ is an optimal solution to (\ref{def_MLE}) if and only if it satisfies the KKT condition in (\ref{eq_kkt_1}) -- (\ref{eq_kkt_2}). Having this in mind, we define
$\wt T_{S_T^c} = 0$
and
\begin{equation}\label{def_MLE_oracle}
	\wt T_{S_T} = \argmax_{\beta\in \Delta_s} N \sum_{j\in J}  X_j\log\left(A_{jS_T}^\T \beta\right).
\end{equation}
The KKT condition corresponding to (\ref{def_MLE_oracle}) states
\begin{align}\label{eq_kkt_oracle_1}
	&N\sum_{j\in J} X_j {A_{jS_T} \over A_{jS_T}^\T \wt T_{S_T}} + \wt \lambda_{S_T}  + \wt \mu \1_s = 0;\\\label{eq_kkt_oracle_2}
	& \wt\lambda_k \ge 0, \quad \wt\lambda_k \wt T_k = 0, ~ \forall k\in S_T,\quad  \wt T_{S_T}^\T \1_s = 1.
\end{align}
Note that (\ref{eq_kkt_oracle_1}) and (\ref{eq_kkt_oracle_2}) together imply $\wt \mu = -N$ by multiplying both sides of (\ref{eq_kkt_oracle_1})  by $\wt T_{S_T}$. We thus define 
\[
\wt \mu = -N,\qquad \wt \lambda_k = N\left(
1 - \sum_{j\in J} X_j {A_{jk} \over A_{jS_T}^\T \wt T_{S_T}}
\right), \ k\in [K].
\]
Clearly, $\supp(\wt T) \subseteq \supp(T)$ by definition. It remains to verify $(\wt T, \wt \lambda, \wt \mu)$ satisfies (\ref{eq_kkt_1}) -- (\ref{eq_kkt_2}) in lieu of $(\wh T, \lambda, \mu)$. By construction, we only need to prove 
\begin{equation}\label{lambda_Sc_pos}
	\wt \lambda_{k} > 0,\footnote{We in fact only need a non-strict inequality for proving (i). The strict inequality is used to prove (ii).}\qquad \forall k\in S_T^c.
\end{equation}
Pick any $k\in S_T^c$. Adding and subtracting terms yields 
\begin{align*}
	\sum_{j\in J} X_j {A_{jk} \over A_{jS_T}^\T \wt T_{S_T}} &~=~  \sum_{j\in J} X_j \left({A_{jk} \over A_{jS_T}^\T \wt T_{S_T}} - {A_{jk} \over A_{jS_T}^\T T_{S_T}}\right) +\sum_{j\in J} X_j {A_{jk} \over A_{jS_T}^\T T_{S_T}}\\
	&~=~ 
	\sum_{j\in J} X_j {A_{jk} A_{jS_T}^\T (T_{S_T} - \wt T_{S_T})\over A_{jS_T}^\T \wt T_{S_T} A_{jS_T}^\T T_{S_T}} + \sum_{j\in \oJ} X_j {A_{jk} \over A_{jS_T}^\T T_{S_T}}\\
	&~=~\sum_{j\in J} X_j {A_{jk} A_{jS_T}^\T (T_{S_T} - \wt T_{S_T})\over A_{jS_T}^\T \wt T_{S_T} A_{jS_T}^\T T_{S_T}} + \sum_{j\in \oJ} \left(X_j - A_{jS_T}^\T T_{S_T}\right){A_{jk} \over A_{jS_T}^\T T_{S_T}} + \sum_{j\in \oJ} A_{jk}\\
	&~\coloneqq~ R_{1,k} + R_{2,k} + \sum_{j\in \oJ} A_{jk},
\end{align*}
where in the second step we used $\Pi_j = A_{jS_T}^\T T_{S_T} > 0$ for $j\in \oJ$ and $X_j = 0$ for $j\in \oJ\setminus J$. 
Since $\sum_{j=1}^p A_{jk} = 1$, it suffices to show 
\[
|R_{1,k} | + |R_{2,k}| \le \sum_{j\in \oJ^c}A_{jk},\qquad \forall k\in S_T^c.
\]
To bound $R_{1,k} $, by writing $\Delta = \wt T_{S_T} - T_{S_T}$ for simplicity, we observe 
\begin{align*}
	|R_{1,k}| &= \left|
	\sum_{j\in J} X_j {A_{jk} A_{jS_T}^\T \Delta\over A_{jS_T}^\T \wt T_{S_T}A_{jS_T}^\T T_{S_T}}
	\right|\\
	& = \left|
	\sum_{a\in S_T}\Delta_a
	\sum_{j\in J} X_j {A_{jk} A_{ja}\over A_{jS_T}^\T \wt T_{S_T}A_{jS_T}^\T T_{S_T}}
	\right|\\
	&\le \|\D\|_1\max_{a\in S_T}
	\sum_{j\in J} X_j {A_{jk} A_{ja}\over A_{jS_T}^\T \wt T_{S_T}A_{jS_T}^\T T_{S_T}}\\
	&\le \|\D\|_1 \max_{j\in \oJ}{A_{jk} \over A_{jS_T}^\T T_{S_T}} \max_{a\in S_T}
	\sum_{j\in J} X_j {A_{ja}\over A_{jS_T}^\T \wt T_{S_T}}.
\end{align*}
From the KKT conditions (\ref{eq_kkt_oracle_1}) -- (\ref{eq_kkt_oracle_2}), we deduce that 
\[
\max_{a\in S_T}
\sum_{j\in J} X_j {A_{ja}\over A_{jS_T}^\T \wt T_{S_T}}
\le 1.
\]
Also by $A_{jS_T}^\T T_{S_T} = \Pi_j$, we conclude 
\begin{equation*}
	\max_{k\in S_T^c}|R_{1,k}| \le \|\Delta\|_1 \max_{k\in S_T^c}\max_{j\in \oJ}{A_{jk}\over \Pi_j} \overset{(\ref{eqn_rho_Sc})}{=} \|\D\|_1\rho_{S_T^c}.
\end{equation*}	
Regarding $R_{2,k}$, invoking Lemma  \ref{lem_oracle_error} with an union bound over $k\in S_T^c$ yields 
\begin{equation*}
	\max_{k\in S_T^c}|R_{2,k}| \le \sqrt{2\rho_{S_T^c}\log ((K-s)/t) \over N} + {2\rho_{S_T^c}\log ((K-s)/t)\over 3N}
\end{equation*}
with probability $1-2p^{-1}-2t.$ 
The desired result follows, provided that 
\begin{align*}
	\min_{k\in S_T^c}\sum_{j\in \oJ^c} A_{jk} & > \rho_{S_T^c} \|\Delta\|_1+ \sqrt{2\rho_{S_T^c}\log ((K-s)/t) \over N} + {2\rho_{S_T^c}\log ((K-s)/t)\over 3N},
\end{align*}
which is ensured by (\ref{cond_supp}) in Theorem \ref{thm_supp} coupled with the rate of $\|\D\|_1$ in Corollary \ref{cor_mle_sparse_supp} and the bound $\rho_{S_T^c} \le \xi / \Tm$ from (\ref{eqn_rho_Sc}) of Remark \ref{rem_rho}.

\paragraph*{Proof of step (ii)} Suppose there exists $\bar T \ne \wt T$ such that $\bar T$ is also an optimal solution to (\ref{def_MLE}). Then, the fact that both $\bar T$ and $\wt T$ are optimal solutions implies 
\[
f(\wt T) = f(\bar T),\qquad \textrm{with}\qquad f(T) = N\sum_{j\in J} X_j \log (A_{j\cdot}^\T T).
\]
Let $\nabla f(\wt T)$ denote the gradient of $f(T)$ at $\wt T$. By adding and subtracting terms,  we obtain
\[
f(\wt T) - f(\bar T) + \langle \nabla f(\wt T), \bar T - \wt T\rangle =  \langle \nabla f(\wt T), \bar T - \wt T\rangle.
\]
The concavity of $f(T)$ ensures that the left hand side of the above equality is positive. We thus have 
\[
\langle \nabla f(\wt T), \wt T - \bar T\rangle \le 0.
\]
Since $\wt T$ satisfies the KKT condition in (\ref{eq_kkt_1}),
$\1_K^\T(\wt T - \bar T) = 0$ and $\wt \lambda^\T \wt T = 0$ from the restrictions (\ref{eq_kkt_2}),
we further deduce that 
\begin{eqnarray*}
	0&=&	\langle \nabla f(\wt T) +\wt \lambda +\wt \mu \1_K, \wt T-\bar T\rangle \\
	&=& \langle \nabla f(\wt T) , \wt T-\bar T\rangle + \langle \wt \lambda , \wt T-\bar T\rangle + \langle \wt \mu \1_K, \wt T-\bar T\rangle 
	\\
	&=& \langle \nabla f(\wt T) , \wt T-\bar T\rangle + \langle \wt \lambda , \wt T-\bar T\rangle\\
	&\le& \langle \wt \lambda , \wt T-\bar T\rangle\\
	&=& -\wt\lambda^\T \bar T\\
	&\le& 0,
\end{eqnarray*}
that is, $\wt\lambda^\T \bar T=0$. 	We conclude that since (\ref{lambda_Sc_pos}) holds, that is, $\wt\lambda_{S_T^c} \succ \mathbf{0}$, then we must have $\bar T_{S_T^c} = \mathbf{0}$.
This shows that $\supp(\bar T) \subseteq \supp(T)= S_T$ and completes our proof.  \qed


\subsection{Proof of Theorem \ref{thm_minimax}: Minimax lower bounds of estimating $T_*$ in $\ell_1$ norm}

We start by constructing the hypotheses. Pick any $1< s\le K$. We choose 
$$
T^{(0)} = {1\over s}(\1_s^\T, \b0^\T)^\T.
$$
For now, suppose $s$ is even. Let $\M = \{0,1\}^{s/2}$. Following \citet[Lemma 2.9]{tsybakov09}, there exists $w^{(j)} \in \M$ for $j = 1,\ldots, |\M|$ such that 
$w^{(0)} = \b0$,  $\log(|\M|) \ge s\log(2) / 16$ and 
\[
\|w^{(j)} - w^{(i)} \|_1 \ge {s\over 16},\quad \forall i\ne j. 
\]
For all  $1\le j\le |\M|$, let 
$
\wt w^{(j)}  = ([w^{(j)}]^\T, -[w^{(j)}]^\T, \b0^\T)^\T
$
and 
\[
T^{(j)} = T^{(0)} + \gamma~  \wt w^{(j)}
\]
with 
$
\gamma = \sqrt{c_0/(sN)}
$
for some constant $c_0>0$. It is easy to see that $T^{(j)} \in \cT'(s)$ for all $0\le j\le |\M|$, under $s\le cN$ for sufficiently small $c>0$. We aim to invoke \citet[Theorem 2.5]{tsybakov09} by proving the following:
\begin{enumerate}\setlength\itemsep{0mm}
	\item[(a)] $\textrm{KL}(\PP_{T^{(j)}}, \PP_{T^{(0)}}) \le \log (|\M|) / 16$, for all $1\le j\le |\M|$;
	\item[(b)] $\|T^{(j)} - T^{(i)}\|_1 \ge c'\sqrt{s/n}$ for all $1\le i\ne j\le |\M|$.
\end{enumerate}
Write $\Pi^{(j)} = AT^{(j)}$ for $0\le j\le |\M|$ and $\oJ^{(0)} = \{i: \Pi_i^{(0)} > 0\}$ for simplicity.
To prove (a), since 
\[
\max_{i\in \oJ^{(0)}} {|\Pi_i^{(j)} - \Pi_i^{(0)}| \over \Pi_i^{(0)}} = 
\gamma\max_{i\in \oJ^{(0)}} {|A_{i\cdot}^\T \wt w^{(j)}| \over A_{i\cdot}^\T T^{(0)}} \le s\gamma < 1, 
\]
and $\Pi_i^{(j)} > 0$ for all $i\in \oJ^{(0)}$, 
invoke \citet[Lemma 12]{bing2020optimal} with $n = 1$ to obtain 
\begin{align*}
	\textrm{KL}(\PP_{T^{(j)}}, \PP_{T^{(0)}}) &\le (1 + c'')\gamma^2 ~ N\sum_{i\in \oJ^{(0)}} {[A_{i\cdot}^\T \wt w^{(j)}]^2 \over A_{i\cdot}^\T T^{(0)}}\\
	&\le (1 + c'')\gamma^2  N  \sigma_1(G_0) \|\wt w^{(j)}\|_2^2\\
	&\le (1+c'') c_0 ~ \sigma_1(G_0) & \textrm{by }\|\wt w^{(j)}\|_2^2\le s
\end{align*}
where $\sigma_1(G_0)$ denotes the largest eigenvalue of $G_0$ with
\[
G_0 = \sum_{i\in \oJ^{(0)}}{1\over A_{i\cdot}^\T T^{(0)}}A_{iS_0}A_{iS_0}^\T.
\]
Here we write $S_0 = \supp(T^{(0)})$. 
The result of part (a) then follows by showing $\sigma_1(G_0) \le s$. To this end, by using the inequality $\sigma_1(M) \le \|M\|_{\infty, 1}$ for any symmetric matrix $M$, we have
\begin{align*}
	\sigma_1(G_0) &\le \max_{k\in S_0} \sum_{a\in S_0}\sum_{i\in \oJ^{(0)}}{A_{ja} A_{jk} \over A_{j\cdot}^\T T^{(0)}}\\
	& = s\max_{k\in S_T} \sum_{i\in \oJ^{(0)}}{ A_{jk} } && \textrm{by } A_{j\cdot}^\T T^{(0)} = {\|A_{jS_0}\|_1 \over s}\\
	& \le s &&\textrm{by }\sum_{j=1}^pA_{jk} = 1.
\end{align*}
We proceed to prove (b) by noting that 
\[
\|T^{(j)} - T^{(i)}\|_1 \ge \gamma {s\over 16} = \sqrt{c_0 \over 16^2 }\sqrt{s\over n},\quad \forall i\ne j.
\]
This concludes the proof when $s$ is even. When $s$ is odd and $s\ge 3$, the same arguments hold by defining 
$\M' = \{0,1\}^{(s-1)/2}$.

\section{Proofs for Section \ref{sec_est_A_T_unknown_A}: Estimation with unknown $A$}\label{app_proof_T_unknown_A}

\subsection{Proof of Theorem  \ref{thm_mle_unknown}: The general bound of the $\ell_1$-norm convergence rate of $\wh T$}\label{app_proof_thm_mle_unknown}

We work on the intersection of events defined in (\ref{ineq:max}) and (\ref{ineq:ell1infty}), and defined in (\ref{def_event}). Without loss of generality, we assume (\ref{ineq:max}) and (\ref{ineq:ell1infty}) hold for the permutation $P = \bI_K$. 
First, we recall that
\begin{equation}\label{bd_A_hat_T}
	{1\over 2}\Pi_j \le \Pi_j - |(\wh A_{j\cdot}-A_{j\cdot})^\T T| \ \le \ \wh A_{j\cdot}^\T T \ \le \ \Pi_j +|(\wh A_{j\cdot}-A_{j\cdot})^\T T| \le {3\over 2}\Pi_j,
\end{equation}
for all $j\in \oJ$, see (\ref{ineq:hat AT-AT}).
The proof resembles the proof of Theorem \ref{thm_mle}.		
The KKT conditions are now
\begin{align}\label{eq_kkt_1_unknown}
	&N\sum_{j\in J} X_j {\wh A_{j\cdot} \over \wh A_{j\cdot}^\T \wh T} + \lambda  + \mu \1_K = 0;\\\label{eq_kkt_2_unknown}
	& \lambda_k \ge 0, \quad \lambda_k \wh T_k = 0, ~ \forall k\in [K],\quad  \wh T^\T \1_K = 1.
\end{align}
Using the same reasoning in the proof of Theorem \ref{thm_mle}, we arrive at
\begin{align}\label{display_key_unknown}
	N\sum_{j\in J} {X_j \over   \wh A_{j\cdot}^\T T}{(\wh A_{j\cdot}^\T \D)^2\over \wh A_{j\cdot}^\T \wh T} &= N\sum_{j\in J} X_j  {\wh A_{j\cdot}^\T \D\over \wh A_{j\cdot}^\T T}   + \lambda^\T \D\\\nonumber
	&\le N\sum_{j\in \oJ} (X_j - \wh A_{j\cdot}^\T T) {\wh A_{j\cdot}^\T \D\over \wh A_{j\cdot}^\T T} + N\sum_{j \in \oJ}\wh A_{j\cdot}^\T \D
\end{align}
with $\D := \wh T - T$.
For the left hand side of (\ref{display_key_unknown}), use $\uJ \subseteq J$ to obtain
\[
\sum_{j\in J} {X_j \over   \wh A_{j\cdot}^\T T}{(\wh A_{j\cdot}^\T \D)^2\over \wh A_{j\cdot}^\T \wh T}  \ge \sum_{j\in  \uJ} {X_j \over   \wh A_{j\cdot}^\T T}{(\wh A_{j\cdot}^\T \D)^2\over \wh A_{j\cdot}^\T \wh T} \ge \min_{j\in \uJ}{X_j \over   \wh A_{j\cdot}^\T T}\sum_{j\in  \uJ} {(\wh A_{j\cdot}^\T \D)^2\over \wh A_{j\cdot}^\T \wh T}.
\]
We argue as before in the proof of Theorem \ref{thm_mle} to obtain
\begin{align*}
	\sum_{j\in  \uJ} {(\wh A_{j\cdot}^\T \D)^2\over \wh A_{j\cdot}^\T \wh T}&		\ge \kappa^2({\wh A_{\uJ}}, s)
\end{align*}
and use (\ref{ineq:conditionnumber}) to prove
$\kappa({\wh A_{\uJ}}, s)
\ge{{1\over 2} \kappa(A_{\uJ},s)
}$, cf.~(\ref{conditionnumbers}).
Since the inequality
\[
{X_j \over \wh A_{j\cdot}^\T T}  \overset{(\ref{bd_A_hat_T})}{\ge}~ {2X_j \over 3\Pi_j} \ge {2\over 3}\cdot {\Pi_j - |X_j - \Pi_j|\over \Pi_j} \ge {2\over 3}\left(1 - {\eps_j \over \Pi_j}\right) \ge {1 \over 3}
\]
holds	for all $j\in \uJ$,
we conclude 
\begin{align*}
	\sum_{j\in \uJ} {X_j \over   \wh A_{j\cdot}^\T T}{(\wh A_{j\cdot}^\T \D)^2\over \wh A_{j\cdot}^\T \wh T} 
	~ \ge  ~ 
	{1 \over 3}\kappa^2(A_{\uJ},s) \|\D\|_1^2.
\end{align*}
Next, for the right hand side of (\ref{display_key_unknown}), using the same argument in the proof of Theorem \ref{thm_mle} gives $\sum_{j\in \oJ} \wh A_{j\cdot}^\T \D \le  0$. It remains to bound from above 
\begin{align*}
	\left|\sum_{j\in \oJ} (X_j - \wh A_{j\cdot}^\T T) {\wh A_{j\cdot}^\T \D\over \wh A_{j\cdot}^\T T} \right| &\le \left|
	\sum_{j\in\oJ} (A_{j\cdot}- \wh A_{j\cdot})^\T T {\wh A_{j\cdot}^\T \D\over \wh A_{j\cdot}^\T T} 
	\right| \\
	&\hspace{-2cm}+ \left|
	\sum_{j\in\oJ} (X_j -A_{j\cdot}^\T T)\left({\wh A_{j\cdot}^\T \D \over \wh A_{j\cdot}^\T T} - {A_{j\cdot}^\T \D \over A_{j\cdot}^\T T} \right)
	\right|   + \left|
	\sum_{j\in\oJ} \left(X_j -  A_{j\cdot}^\T T\right) { A_{j\cdot}^\T \D \over  A_{j\cdot}^\T T} 
	\right|.
\end{align*}
In the proof of Theorem \ref{thm_mle}, we have shown 
\begin{equation}\label{bd_term_3}
	\left|\sum_{j\in \oJ} (X_j-\Pi_j) \frac {A_{j\cdot}^\T \D}{A_{j\cdot}^\T T}\right| \le 	\|\D\|_1\left\{\sqrt{2\rho\log (K/t) \over N} + {2\rho \log (K/t)\over 3N}\right\}
\end{equation}
with probability $1-2t$, for any $t\ge 0$. For the first term, by using $\|T\|_1 =1$,
\begin{align*}
	\left|
	\sum_{j\in\oJ} (A_{j\cdot}- \wh A_{j\cdot})^\T T {\wh A_{j\cdot}^\T \D\over \wh A_{j\cdot}^\T T} 
	\right| &\le \|\D\|_1 \max_{1\le k\le K}\left|
	\sum_{j\in\oJ} (A_{j\cdot}- \wh A_{j\cdot})^\T T {\wh A_{jk}\over \wh A_{j\cdot}^\T T} 
	\right|\\ 
	&\le \|\D\|_1 \sum_{j\in\oJ} |(A_{j\cdot}- \wh A_{j\cdot})^\T T| {\|\wh A_{j\cdot}\|_\i\over \wh A_{j\cdot}^\T T}\\
	&\le  \|\D\|_1  \max_{j\in\oJ}{\|\wh A_{j\cdot}\|_\i\over \wh A_{j\cdot}^\T T}\left\|\wh A_{\oJ}-  A_{\oJ}\right\|_{1,\i}.
\end{align*}
Use (\ref{bd_A_hat_T}) and (\ref{ineq:max}) to find
\begin{equation}\label{bd_Ahat_jk_Ahat_T}
	\max_{j\in\oJ}{\|\wh A_{j\cdot}\|_\i\over \wh A_{j\cdot}^\T T} \le 2\max_{j\in\oJ}{\|A_{j\cdot}\|_\i + \|\wh A_{j\cdot}-A_{j\cdot}\|_\i\over \Pi_j}  \le  2\left(\rho +{1\over 2}\right) \le 3\rho,
\end{equation}
hence
\begin{align}\label{bd_term_1}
	\left|
	\sum_{j\in\oJ} (A_{j\cdot}- \wh A_{j\cdot})^\T T {\wh A_{j\cdot}^\T \D\over \wh A_{j\cdot}^\T T} 
	\right|\le    3 \|\D\|_1 \rho \left\|\wh A_{\oJ}-  A_{\oJ}\right\|_{1,\i}.
\end{align}
Furthermore, the second term can be bounded from above by 
\begin{align*}
	&\|\D\|_1 \max_{1\le k\le K}\left\{\left|
	\sum_{j\in\oJ} (X_j -A_{j\cdot}^\T T){\wh A_{jk}(\wh A_{j\cdot}- A_{j\cdot})^\T T \over \wh A_{j\cdot}^\T T\cdot  A_{j\cdot}^\T T}\right| + 	\left|
	\sum_{j\in\oJ} (X_j -A_{j\cdot}^\T T){\wh A_{jk} - A_{jk}\over A_{j\cdot}^\T T}
	\right| \right\}\\
	\le &\|\D\|_1 \max_{1\le k \le K}\left\{
	\sum_{j\in\oJ} {|X_j - \Pi_j| \over \Pi_j}{\|\wh A_{j\cdot}\|_\i \over\wh A_{j\cdot}^\T T} |\wh A_{jk}- A_{jk}| +
	\sum_{j\in\oJ} {|X_j - \Pi_j| \over \Pi_j}|\wh A_{jk}- A_{jk}|\right\}\\
	\le & (1+3\rho)\|\D\|_1\max_{1\le k \le K}
	\sum_{j\in\oJ} {|X_j - \Pi_j| \over \Pi_j}|\wh A_{jk}- A_{jk}|
\end{align*}
where the last step uses (\ref{bd_Ahat_jk_Ahat_T}). Invoke the event $\E$ defined in (\ref{def_event}) to find
\begin{align}\label{bd_term_2}
	\left|
	\sum_{j\in\oJ} (X_j -A_{j\cdot}^\T T)\left({\wh A_{j\cdot}^\T \D \over \wh A_{j\cdot}^\T T} - {A_{j\cdot}^\T\D\over A_{j\cdot}^\T T} \right)
	\right|   
	\le  	4\|\D\|_1\rho\max_{1\le k \le K}\sum_{j\in \oJ} {\eps_j \over \Pi_j} |\wh A_{jk}- A_{jk}|.
\end{align}
By collecting terms (\ref{bd_term_3}), (\ref{bd_term_1}) and (\ref{bd_term_2}) and using the expression of $\eps_j$ in (\ref{def_eps_j}), we conclude
\begin{align*}
	{1\over 3} \kappa^2(A_{\uJ},s)
	\|\D\|_1\le ~&~ 4 \rho \max_{1\le k \le K}\sum_{j\in \oJ}|\wh A_{jk}-  A_{jk}|   \left(1 + 2\sqrt{\log(p)\over \Pi_j N} +  {4\log(p) \over 3\Pi_j N} \right)\\
	&\quad + 	\sqrt{2\rho\log (K/t) \over N} + {2\rho \log (K/t)\over 3N}
\end{align*}
with probability $1-2t$, for any $t\ge 0$. Take $t = p^{-1}$. Recall that $\uJ\subseteq \oJ$ and 
$\Pi_j \ge 8\log(p) / (3N)$ for all $j \in \uJ$. We have
\begin{align*}
	&\sum_{j\in \oJ}|\wh A_{jk}-  A_{jk}|   \left(1 + 2\sqrt{\log(p)\over \Pi_j N} +  {4\log(p) \over 3\Pi_j N} \right)\\ 
	&\le 2\sum_{j\in \oJ}|\wh A_{jk}-  A_{jk}|   \left(1 +    {7\log(p) \over 3\Pi_j N} \right)\\
	&\le 4\|\wh A_{\oJ k} - A_{\oJ k}\|_1 + {14\over 3}\sum_{j\in \oJ\setminus \uJ}{|\wh A_{jk}-A_{jk}|\over \Pi_j}{\log(p)\over N},
\end{align*}
completing the proof.
\qed

\subsection{Proof of Theorem \ref{thm_mle_fast_unknown}: Fast rates of $\wh T$}\label{app_proof_thm_mle_fast_unknown}

We work on 
the intersection of the events, defined in (\ref{ineq:max}) and (\ref{ineq:ell1infty2}), and the event $\E$ in (\ref{def_event}), so we can assume that 
the conclusion of Theorem \ref{thm_mle_unknown} holds for  $P = \bI_K$ without loss of generality, that is,
\begin{align*}
	\|\wh T - T_*\|_1 &\le {3\over  \kappa^{2}(A_{\uJ},s)}\left\{
	\sqrt{2\rho\log(p) \over N} + {2\rho \log(p)\over 3N}+16\rho \bigl\|\wh A_{\oJ}-  A_{\oJ}\bigr\|_{1,\i}\right.\\
	&\hspace{2.4cm}\left.  + {56\over 3}\rho \sum_{j\in \oJ\setminus \uJ} {\|\wh A_{j\cdot} - A_{j\cdot}\|_\i \over \Pi_j}{\log(p) \over N}
	\right\}\\
	&\lesssim {1\over  \kappa^{2}(A_{\uJ},s)}\left\{
	\sqrt{\rho\log(p) \over N} + {\rho \log(p)\over N}+\rho \bigl\|\wh A_{\oJ}-  A_{\oJ}\bigr\|_{1,\i}
	\right\}.
\end{align*}
The last step also uses (\ref{ineq:max}) and $|\oJ\setminus \uJ|\le C'$ to collect terms.
Similar to the arguments of proving Theorem \ref{thm_mle_fast}, we notice that  (\ref{ineq:max}), (\ref{cond_N_unknown}) and (\ref{ineq:ell1infty2})   guarantee that
\begin{equation}\label{bd_A_T_hat_unknown}
	(1-c)\Pi_j \le  A_{j\cdot}^\T \wh T \le (1+c) \Pi_j,\qquad \forall~  j\in J\subseteq \oJ.
\end{equation}
From (\ref{display_key_unknown}) and the observation that \begin{align*}
	\sum_{j \in \oJ}\wh A_{j\cdot}^\T \D &= -\sum_{j \in \oJ^c}\wh A_{j\cdot}^\T \D\le \sum_{j \in \oJ^c}\wh A_{j\cdot}^\T T \overset{(\ref{bd_A_hat_T})}{\le} {3\over2}\sum_{j \in \oJ^c}\Pi_j = 0,
\end{align*}
we have 
\[
N\sum_{j\in J} {X_j \over   \wh A_{j\cdot}^\T T}{(\wh A_{j\cdot}^\T \D)^2\over \wh A_{j\cdot}^\T \wh T} 
\le N\sum_{j\in \oJ} (X_j - \wh A_{j\cdot}^\T T)  {\wh A_{j\cdot}^\T \D\over \wh A_{j\cdot}^\T T}.
\]
In the proof of Theorem \ref{thm_mle_unknown} we have shown 
\begin{align*}
	\left|\sum_{j\in \oJ} (X_j - \wh A_{j\cdot}^\T T) {\wh A_{j\cdot}^\T \D\over \wh A_{j\cdot}^\T T} \right| &\le \rI + \rII + \rIII 
\end{align*}
where 
\begin{align*}
	&\rI = \left|
	\sum_{j\in\oJ} (A_{j\cdot}- \wh A_{j\cdot})^\T T {\wh A_{j\cdot}^\T \D\over \wh A_{j\cdot}^\T T} 
	\right| \le  3 \|\D\|_1 \rho \left\|\wh A_{\oJ}-  A_{\oJ}\right\|_{1,\i},\\
	&\rII =  \left|
	\sum_{j\in\oJ} (X_j -\Pi_j)\left({\wh A_{j\cdot}^\T \D \over \wh A_{j\cdot}^\T T} - {A_{j\cdot}^\T \D \over A_{j\cdot}^\T T} \right)
	\right|\le 4\|\D\|_1\rho\max_{1\le k \le K}\sum_{j\in \oJ} {\eps_j \over \Pi_j} |\wh A_{jk}- A_{jk}|,\\
	&\rIII =  \left|
	\sum_{j\in\oJ} \left(X_j -  \Pi_j\right) { A_{j\cdot}^\T \D \over  \Pi_j} 
	\right| \le \left|
	\sum_{j\in \oJ} (X_j-\Pi_j) \frac {A_{j\cdot}^\T H^{-1/2} H^{1/2} \D}{ \Pi_j}
	\right|.
\end{align*}
Also by the arguments in the proof of Theorem \ref{thm_mle_fast}, one can deduce that 
\begin{align*}
	\left|\sum_{j\in \oJ} (X_j-\Pi_j) \frac {A_{j\cdot}^\T H^{-1/2} H^{1/2} \D}{ \Pi_j}\right| &\le \|H^{1/2}\D\|_2 \left\|\sum_{j\in \oJ}{X_j-\Pi_j \over \Pi_j} H^{-1/2}A_{j\cdot}\right\|_2\\ 
	&\lesssim \|H^{1/2}\D\|_2 \sqrt{K\log(p) \over N}
\end{align*}
with probability $1 - 2p^{-1}$. With the same probability, by using (\ref{bd_Delta_1}), we conclude 
\begin{align}\label{disp_intermediate}
	&\sum_{j\in J} {X_j \over   \wh A_{j\cdot}^\T T}{(\wh A_{j\cdot}^\T \D)^2\over \wh A_{j\cdot}^\T \wh T}\\\nonumber
	&  \lesssim \|H^{1/2}\D\|_2 \left[\sqrt{K\log(p) \over N} +{\rho\over  \kappa(A_{\oJ}, s)}
	\max_{1\le k \le K}\sum_{j \in \oJ} 
	|\wh A_{jk}-A_{jk}|\left(
	1 + {\log(p)\over \Pi_j N}
	\right) 
	\right]\\\nonumber
	& \lesssim \|H^{1/2}\D\|_2\sqrt{K\log(p) \over N} +{\rho\|H^{1/2}\D\|_2 \over  \kappa(A_{\oJ}, s)}\left(\|\wh A_{\oJ}-A_{\oJ}\|_{1,\i}+
	\sum_{j \in \oJ\setminus \uJ} 
	{\|\wh A_{j\cdot}-A_{j\cdot}\|_\i \over \Pi_j} {\log(p)\over  N} 
	\right).
\end{align}
We proceed to bound from below the left hand side. Since
(\ref{ineq:ell1infty2})
together with (\ref{bd_A_T_hat_unknown}) implies 
\begin{align}\label{bd_A_hat_pi_hat}\nonumber
	(1/2-c)\Pi_j  &\le  A_{j\cdot}^\T \wh T -  \|\wh A_{j\cdot}-A_{j\cdot}\|_\i \|\wh T\|_1\\\nonumber
	&\le  
	\wh A_{j\cdot}^\T \wh T\\
	&\le A_{j\cdot}^\T \wh T + \|\wh A_{j\cdot}-A_{j\cdot}\|_\i \|\wh T\|_1 \le (3/2+c) \Pi_j
\end{align}
and 
\begin{equation}\label{bd_sandwich_A_hat_T}
	\Pi_j/2 \le  
	\wh A_{j\cdot}^\T  T  \le  3\Pi_j/2,
\end{equation}
for all $j\in J\subseteq \oJ$. We have 
\[
\sum_{j\in J} {X_j \over   \wh A_{j\cdot}^\T T}{(\wh A_{j\cdot}^\T \D)^2\over \wh A_{j\cdot}^\T \wh T} = \sum_{j\in \oJ} {X_j \over   \wh A_{j\cdot}^\T T}{(\wh A_{j\cdot}^\T \D)^2\over \wh A_{j\cdot}^\T \wh T} \gtrsim \D^\T \wt H \D  
\]
where we write
\[
\wt H = \sum_{j\in \oJ} {X_j \over  \Pi_j^2}\wh A_{j\cdot}\wh A_{j\cdot}^\T.
\]
Recall the definition of $\wh H$ from (\ref{def_I_I_hat}). It follows that
\begin{align*}
	\D^\T \wt H \D  & \ge \D^\T \wh H \D - 
	\sum_{j\in \oJ}X_j {|\wh A_{j\cdot}^\T \D|\over \Pi_j} {|(\wh A_{j\cdot} - A_{j\cdot})^\T \D|\over \Pi_j} - \sum_{j\in \oJ}X_j {|A_{j\cdot}^\T \D|\over \Pi_j} {|(\wh A_{j\cdot} - A_{j\cdot})^\T \D|\over \Pi_j}.
\end{align*}
By using 
\begin{equation*}
	\max_{j\in \oJ}{\|\wh A_{j\cdot}\|_\i \over \Pi_j} \le \rho +  	\max_{j\in \oJ}{\|\wh A_{j\cdot} - A_{j\cdot}\|_\i \over \Pi_j} \le  \rho + {1\over 2} \le 2\rho,
\end{equation*}
we first have 
\begin{align*}
	&\sum_{j\in \oJ}X_j {|\wh A_{j\cdot}^\T \D|\over \Pi_j} {|(\wh A_{j\cdot} - A_{j\cdot})^\T \D|\over \Pi_j}\\
	&\le 2\rho \|\D\|_1^2 \max_{k\in [K]}\sum_{j\in \oJ}{X_j \over \Pi_j}|\wh A_{jk}-A_{jk}| \\
	& \le 2\rho \|\D\|_1^2 \max_{k\in [K]}\sum_{j\in \oJ}\left( 1 + {|X_j-\Pi_j| \over \Pi_j}\right)|\wh A_{jk}-A_{jk}| \\
	&\le 4\rho \|\D\|_1^2 \max_{k\in [K]}\sum_{j\in \oJ}\left( 1 +{7\log(p)\over 3\Pi_j N}\right)|\wh A_{jk}-A_{jk}| \\
	&\lesssim   \rho \|\D\|_1^2\left(\|\wh A_{\oJ}-  A_{\oJ}\|_{1,\i}  + \sum_{j \in \oJ\setminus \uJ} 
	{\|\wh A_{j\cdot}-A_{j\cdot}\|_\i \over \Pi_j} {\log(p)\over  N} \right)\\
	&\le  {\rho \over \kappa(A_{\oJ},s)}  \left(\|\wh A_{\oJ}-  A_{\oJ}\|_{1,\i}  + {\log(p)\over  N} \right)\|H^{1/2}\D\|_2^2
\end{align*}
where the last line uses (\ref{bd_Delta_1}) and $|\oJ\setminus \uJ|\le C'$. 
A similar argument also gives the same upper bound for 
\[
\sum_{j\in \oJ}X_j {| A_{j\cdot}^\T \D|\over \Pi_j} {|(\wh A_{j\cdot} - A_{j\cdot})^\T \D|\over \Pi_j}.
\]
Under condition (\ref{cond_N_unknown}) and (\ref{ineq:ell1infty2}), and also by invoking Lemma \ref{lem_I_deviation}, we readily have 
\[
\PP\left\{\D^\T \wt H \D \gtrsim \|H^{1/2}\D\|_2^2\right\} \ge 1-2K^{-1},
\]
which together with (\ref{disp_intermediate})
gives 
\[
\|H^{1/2}\D\|_2 \lesssim \sqrt{K\log(p) \over N} +
{\rho \over  \kappa(A_{\oJ}, s)}\left(\|\wh A_{\oJ}-A_{\oJ}\|_{1,\i}+
\sum_{j \in \oJ\setminus \uJ} 
{\|\wh A_{j\cdot}-A_{j\cdot}\|_\i \over \Pi_j} {\log(p)\over  N} 
\right).
\]
Invoke (\ref{bd_Delta_1}) and use (\ref{cond_N_unknown}) and $|\oJ\setminus \uJ|\le C'$ to simplify the expression to 
complete the proof. \qed

\subsection{Proof of Theorem \ref{thm_supp_unknown}: One-sided sparsity recovery of $\wh T$}\label{app_proof_supp_unknown}

The arguments resemble the proof of Theorem \ref{thm_supp}. The proof of step (ii) follows exactly from the same argument by replacing $A_{j\cdot}$ by $\wh A_{j\cdot}$. We therefore only prove step (i): there exists an optimal solution $\wt T$ to (\ref{def_T_hat}) such that $\supp(\wt T) \subseteq \supp(T)$. Similarly, we define
$\wt T_{S_T^c} = 0$
and
\begin{equation}\label{def_MLE_oracle_unknown}
	\wt T_{S_T} = \argmax_{\beta\in \Delta_s} N \sum_{j\in J}  X_j\log\left(\wh A_{jS_T}^\T \beta\right).
\end{equation}
The KKT condition corresponding to (\ref{def_MLE_oracle_unknown}) states
\begin{align}\label{eq_kkt_oracle_1_unknown}
	&N\sum_{j\in J} X_j {\wh A_{jS_T} \over \wh A_{jS_T}^\T \wt T_{S_T}} + \wt \lambda_{S_T}  + \wt \mu \1_s = 0;\\\label{eq_kkt_oracle_2_unknown}
	& \wt\lambda_k \ge 0, \quad \wt\lambda_k \wt T_k = 0, ~ \forall k\in S_T,\quad  \wt T_{S_T}^\T \1_s = 1.
\end{align}
By similar reasoning as the proof of Theorem \ref{thm_supp}, we define
\[
\wt \mu = -N,\quad \wt \lambda_k = N\left(
1 - \sum_{j\in J} X_j {\wh A_{jk} \over \wh A_{jS_T}^\T \wt T_{S_T}}
\right),\qquad \forall k\in [K].
\]
The verification that $(\wt T, \wt \lambda, \wt \mu)$ satisfies (\ref{eq_kkt_1_unknown}) -- (\ref{eq_kkt_2_unknown}) in lieu of $(\wh T, \lambda, \mu)$ boils down to show
\begin{equation}\label{lambda_Sc_pos_unknown}
	\wt \lambda_{k} > 0,\qquad \forall k\in S_T^c.
\end{equation}
We show this on the event defined in Theorem \ref{thm_supp_unknown}. Notice that conditions in Theorem \ref{thm_supp_unknown} imply that both (\ref{bd_A_hat_pi_hat}) and (\ref{bd_sandwich_A_hat_T}) hold. 

Pick any $k\in S_T^c$. Adding and subtracting terms yields 
\begin{align*}
	\sum_{j\in J} X_j {\wh A_{jk} \over \wh A_{jS_T}^\T \wt T_{S_T}} &= \sum_{j\in J}X_j {\wh A_{jk} - A_{jk} \over \wh A_{jS_T}^\T \wt T_{S_T}} + 
	\sum_{j\in J}X_j{A_{jk} \wh A_{jS_T}^\T (T_{S_T} - \wt T_{S_T}) \over \wh A_{jS_T}^\T \wt T_{S_T}\wh A_{jS_T}^\T T_{S_T}}\\
	&\quad + \sum_{j\in J}X_j {A_{jk}(A_{jS_T} - \wh A_{jS_T})^\T T_{S_T} \over \wh A_{jS_T}^\T T_{S_T} A_{jS_T}^\T T_{S_T}} + \sum_{j\in J}(X_j - \Pi_j) {A_{jk} \over \Pi_j} + \sum_{j\in J}A_{jk}\\
	&= R_{1,k} + R_{2,k} + R_{3,k} + R_{4,k} + \sum_{j\in J}A_{jk}.
\end{align*}
We bound each term separately. 
For $R_{1,k}$, using (\ref{bd_A_hat_pi_hat}), $J\subseteq \oJ$  and the proof of Theorem \ref{thm_mle_fast_unknown} gives, on the event $\E$ in (\ref{def_event}),
\begin{align}\label{bd_R1}\nonumber
	|R_{1,k}| & \lesssim \sum_{j\in J} {X_j\over \Pi_j} |\wh A_{jk} - A_{jk}|\\\nonumber &\lesssim \|\wh A_{\oJ} - A_{\oJ}\|_{1,\i} + \sum_{j \in \oJ\setminus \uJ} 
	{\|\wh A_{j\cdot}-A_{j\cdot}\|_\i \over \Pi_j} {\log(p)\over  N}\\
	&\lesssim \|\wh A_{\oJ} - A_{\oJ}\|_{1,\i} + {\log(p)\over  N} && \textrm{by }|\oJ\setminus \uJ|\le C'.
\end{align}
To bound $R_{2,k} $, by writing $\Delta = \wt T_{S_T} - T_{S_T}$ for simplicity and using similar arguments in the proof of Theorem \ref{thm_supp}, we have 
\begin{align*}
	|R_{2,k}| &= \left|
	\sum_{j\in J} X_j {A_{jk} \wh A_{jS_T}^\T \Delta\over \wh A_{jS_T}^\T \wt T_{S_T}\wh A_{jS_T}^\T T_{S_T}}
	\right| \le \|\D\|_1 \max_{j\in \oJ}{A_{jk} \over \wh A_{jS_T}^\T T_{S_T}} \max_{a\in S_T}
	\sum_{j\in J} X_j {\wh A_{ja}\over \wh A_{jS_T}^\T \wt T_{S_T}}.
\end{align*}
From (\ref{eq_kkt_oracle_1_unknown}) -- (\ref{eq_kkt_oracle_2_unknown}), we deduce that 
\[
\max_{a\in S_T}
\sum_{j\in J} X_j {\wh A_{ja}\over \wh A_{jS_T}^\T \wt T_{S_T}}
\le 1.
\]
Also, by using (\ref{bd_sandwich_A_hat_T}) together with 
\[
\|\Delta\|_1 \lesssim  \sqrt{s\log(p)\over N} + \rho\|\wh A_{\oJ} - A_{\oJ}\|_{1,\i}
\]
deduced from Theorem \ref{thm_mle_fast_unknown} with $K = s$ and  $\kappa^{-1}(A_{\oJ},s) \le C''$, 
we conclude that
\begin{equation}\label{bd_R2}
	|R_{2,k}| \lesssim \rho_{S_T^c} \left(\sqrt{s\log(p)\over N} + \rho \|\wh A_{\oJ} - A_{\oJ}\|_{1,\i}\right).
\end{equation}	
holds with probability at least $1-8p^{-1}$.
For $R_{3,k}$, by the arguments of bounding $R_{1,k}$ and $R_{2,k}$, it is easy to see that 
\begin{align}\label{bd_R3}
	|R_{3,k}| &\le  \rho_{S_T^c} \max_{a\in S_T}\sum_{j\in J}{X_j \over \Pi_j} |\wh A_{ja} - A_{ja}| \lesssim \rho_{S_T^c}\left(\|\wh A_{\oJ} - A_{\oJ}\|_{1,\i} +  {\log(p)\over  N}\right).
\end{align}
Regarding $R_{4,k}$, invoking Lemma \ref{lem_oracle_error} with $t= 1/p$ and taking a union bounds over $k\in S_T^c$ yields 
\begin{align}\label{bd_R4}
	\max_{k\in S_T^c} |R_{4,k}| &\le \sqrt{2\rho_{S_T^c}\log (p) \over N} + {2\rho_{S_T^c}\log (p)\over 3N}
\end{align}
with probability $1-2p^{-1}.$
Finally, since 
\[
1 - \sum_{j\in J}A_{jk} \ge \sum_{j\in \oJ^c}A_{jk}, 
\]
by collecting terms in (\ref{bd_R1}), (\ref{bd_R2}), (\ref{bd_R3}) and (\ref{bd_R4}),
the desired result follows provided that 
\begin{align*}
	\min_{k\in S_T^c}\sum_{j\in \oJ^c} A_{jk} & \gtrsim  \rho_{S_T^c} \sqrt{s\log(p)\over N} + \sqrt{\rho_{S_T^c}\log (p) \over N}  + (1 + \rho_{S_T^c}\rho)\|\wh A_{\oJ} - A_{\oJ}\|_{1,\i}+  {\log(p)\over  N} 
\end{align*}
which is ensured by the condition in  Theorem \ref{thm_supp_unknown} coupled with the fact $\rho_{S_T^c} \le \xi /\Tm$ and $\rho \le (1\vee \xi)/\Tm$. The proof is then complete. \qed

\section{Proof of Proposition \ref{allbounds} in Section \ref{sec:w bounds}}\label{proof:allbounds}

We prove (\ref{word}) and (\ref{topicw}) -- (\ref{topictv}) separately in this section. We collect technical lemmas that are used in the proofs at the end of this section.

\subsection*{Proof of (\ref{word})}

Using the triangle inequality for $W_1$ (Lemma \ref{thm:w props} below), 
\[
W_1(\wt \Pi^{(i)}, \wt \Pi^{(j)};\Dw) \le W_1(\wt \Pi^{(i)}, \Pi^{(i)}_*;\Dw) + W_1(\Pi^{(i)}_*,  \Pi^{(j)}_*;\Dw) + W_1(\Pi^{(j)}_*, \wt \Pi^{(j)};\Dw),
\]
and thus
\[W_1(\wt \Pi^{(i)}, \wt \Pi^{(j)};\Dw) - W_1(\Pi^{(i)}_*,  \Pi^{(j)}_*;\Dw)\le  \sum_{k\in \{i,j\}} W_1(\wt \Pi^{(k)}, \Pi^{(k)}_*;\Dw).\]
Combining this with a second application of the triangle inequality with the roles of $\wt\Pi^{(k)}$ and $\Pi^{(k)}$ switched for $k\in \{i,j\}$, we find
\begin{align}
	\left|W_1(\wt \Pi^{(i)}, \wt \Pi^{(j)};\Dw) - W_1( \Pi^{(i)}_*, \Pi^{(j)}_*;D_w)\right| &\le \sum_{k\in \{i,j\}} W_1(\wt \Pi^{(k)}, \Pi^{(k)}_*;\Dw)\nonumber\\
	&\le \|\Dwb\|_\i \frac{1}{2}\sum_{k\in \{i,j\}} \|\wt \Pi^{(k)}- \Pi^{(k)}_*\|_1,\label{pi tilde tv bnd dword}
\end{align}
where the second step follows from Lemma \ref{thm:w props} below. For $k\in \{i,j\}$, we find
\begin{align}
	\|\wt \Pi^{(k)}- \Pi^{(k)}_{*}\|_1 &= \|\wh A\wh T^{(k)} - AT_*^{(k)}\|_1\nonumber\\
	&= \|\wh A\wh T^{(k)} - A\wh T^{(k)}+A\wh T^{(k)} - A T_*^{(k)}\|_1\nonumber\\
	&\le \|(\wh A - A)\wh T^{(k)}\|_1 + \|A(\wh T^{(k)} -  T_{*}^{(k)})\|_1 \nonumber\\
	&\le \max_{l\in [K]}\|\wh A_{\cdot l} - A_{\cdot l}\|_1\|\wh T^{(k)}\|_1 + \max_{l\in [K]}\|A_{\cdot l}\|_1\|\wh T^{(k)} -  T_{*}^{(k)}\|_1 \label{pi tilde bnd step}\\
	&= \max_{l\in [K]}\|\wh A_{\cdot l} - A_{\cdot l}\|_1 + \|\wh T^{(k)} -  T_{*}^{(k)}\|_1,\nonumber
\end{align}
where (\ref{pi tilde bnd step}) follows from the fact that for any $v\in \R^K$,
\[\|Av\|_1 = \sum_{i=1}^p\left|\sum_{k=1}^K A_{ik}v_k\right|\le \sum_{k=1}^K |v_k| \sum_{i=1}^p|A_{ik}|\le \max_{l\in [K]}\|A_{\cdot l}\|\|v\|_1,\]
and in the final step we use that $A_{\cdot l} \in \Delta_p$ for all $l\in [K]$ and $\wh T^{(k)}\in \Delta_K$.
Plugging this into (\ref{pi tilde tv bnd dword}) we find
\begin{align}
	&\left|W_1(\wt \Pi^{(i)}, \wt \Pi^{(j)};\Dw) - W_1( \Pi^{(i)}_*, \Pi^{(j)}_*;D_w)\right| \nonumber\\
	&\qquad\qquad \le \|\Dwb\|_\i \left\{\max_{l\in [K]}\|(\wh A - A)\be_l\|_1 +\frac{1}{2}\sum_{k\in \{i,j\}} \|\wh T^{(k)} -  T_*^{(k)}\|_1\right\}.\label{eqn:pi tilde final except P}
\end{align}
Finally, note that for any $P\in \H_K$, $PP^\top = I_K$, so for $k\in \{i,j\}$,
\[\Pi^{(k)}_* = AT_*^{(k)} = APP^\top T_*^{(k)}.\]
Furthermore, $AP\in \Delta_p$ and $P^\top T_*^{(k)}\in \Delta_K$.	Thus, (\ref{eqn:pi tilde final except P}) holds when $A$ and $T^{(k)}_*$ are replaced by $AP$ and $P^\top T_*^{(k)}$, respectively, for any $P\in \H_K$. We can thus take the maximum over $P\in \H_K$, which completes the proof of (\ref{word}).
\qed 

\subsection*{Proof of (\ref{topicw}) and (\ref{topictv})}
We will prove the bound
\begin{align}
	&\left|W_1(\wh T^{(i)},\wh T^{(j)}; \hDt) - W_1(T^{(i)}_*,T^{(j)}_*; \Dt)\right| \nonumber\\
	&\quad \le 2\max_{k\in [K]}d(\wh A_{\cdot k}, A_{\cdot k}) + \|\Dtb\|_\i\frac{1}{2}\sum_{k\in \{i,j\}}\|\wh T^{(k)}- T^{(k)}_*\|_1\label{t bnd d},
\end{align}
where $d$ is any metric on $\Delta_p$, and
\begin{equation}\label{eqn:hdt def}
	\hDt(k,l) \coloneqq d(\wh A_{\cdot k}, \wh A_{\cdot l}), \quad \Dt(k,l) \coloneqq d(A_{\cdot k}, A_{\cdot l})\quad \forall k,l\in [K].
\end{equation}
Combining this with Lemma \ref{thm:Hk topic} below, (\ref{t bnd d}) yields
\begin{align}
	&\left|W_1(\wh T^{(i)},\wh T^{(j)}; \hDt) - W_1(T^{(i)}_*,T^{(j)}_*; \Dt)\right| \nonumber\\
	&\quad \le \max_{P\in \H_K}\left\{2\max_{k\in [K]}d(\wh A_{\cdot k}, (AP)_{\cdot k}) + \|\Dtb\|_\i\frac{1}{2}\sum_{k\in \{i,j\}}\|\wh T^{(k)}- P^\top T^{(k)}_*\|_1\right\}\label{t bnd d w Hk}.
\end{align}
Equation (\ref{topictv}) follows immediately from (\ref{t bnd d w Hk}) using $d(a,b) = \frac{1}{2}\|a-b\|_1$ for $a,b\in \Delta_p$, and noting that for this choice of $d$, $\|\Dtb\|_\i\le 1$. To prove (\ref{topicw}), choose $d(a,b) = W_1(a,b;\Dw)$ for $a,b\in \Delta_p$, and note that by Lemma \ref{thm:w props},
\begin{equation}\label{topic w 1}
	W_1(\wh A_{\cdot k},(AP)_{\cdot k};\Dw)\le \|\Dw\|_\i\frac{1}{2}\|\wh A_{\cdot k}-(AP)_{\cdot k}\|\quad \forall k\in [K],
\end{equation}
and for this choice of $d$,
\begin{equation}\label{topic w 2}
	\|\Dtb\|_\i = \max_{k,l\in[K]}W_1(A_{\cdot k}, A_{\cdot l};\Dw) \le \|\Dw\|\max_{k,l\in [K]}\| A_{\cdot k}-A_{\cdot l}\|\le \|\Dw\|.
\end{equation}
Combining (\ref{topic w 1}) and (\ref{topic w 2}) with (\ref{t bnd d w Hk}) proves (\ref{topicw}).

\subsubsection*{Proof of (\ref{t bnd d})} We first find
\begin{align}
	W_1(\wh T^{(i)},\wh T^{(j)}; \hDt) &= \inf_{w\in \Gamma(\wh T^{(i)},\wh T^{(j)})} \tr( w\hDtb)\nonumber\\
	&= \inf_{w\in \Gamma(\wh T^{(i)},\wh T^{(j)})}\left\{ \tr( w\Dtb) + \tr(w[\hDtb - \Dtb])\right\}\nonumber\\
	&=  \inf_{w\in \Gamma(\wh T^{(i)},\wh T^{(j)})} \tr( w\Dtb) +  \|\hDtb - \Dtb\|_\i \label{eqn: d inf tr bnd} \\
	&= \|\hDtb - \Dtb\|_\i+  W_1(\wh T^{(i)},\wh T^{(j)}; \Dt),\label{eqn:d inf init bound}
\end{align}
where in (\ref{eqn: d inf tr bnd}) we use that for any $w\in \Gamma(\wh T^{(i)},\wh T^{(j)})$,
\begin{align}
	\tr(w[\hDtb - \Dtb]) &= \sum_{t,l=1}^K w_{tl}(\hDtb_{lt} - \Dtb_{lt})\nonumber\\
	&\le \|\hDtb - \Dtb\|_\i \cdot \sum_{t,l=1}^Kw_{tl} &\text{since } w_{lt}\ge 0 \text{ for } t,l\in [K]\nonumber\\
	&=  \|\hDtb - \Dtb\|_\i. \label{d hat d bnd}
\end{align}
Using the triangle inequality for $W_1$ (Lemma \ref{thm:w props}), we find
\begin{align*}
	W_1(\wh T^{(i)},\wh T^{(j)}; \Dt) \le  W_1(\wh T^{(i)},T_*^{(i)}; \Dt)+ W_1( T_*^{(i)}, T_*^{(j)}; \Dt)+ W_1(\wh T^{(j)}, T_*^{(j)}; \Dt).
\end{align*}
Plugging this into (\ref{eqn:d inf init bound}) we find
\begin{align}
	&W_1(\wh T^{(i)},\wh T^{(j)};\hDt) - W_1(T^{(i)}_*,T^{(j)}_*;\Dt)\nonumber \\
	&\qquad \qquad\le \|\hDt- \Dt\|_\i + \sum_{k\in \{i,j\}}W_1(\wh T^{(k)},T^{(k)}_*;\Dt).\label{eqn:g hat g}
\end{align}
Using the triangle inequality again,
\begin{align*}
	W_1(T^{(i)}_*,T^{(j)}_*;\Dt) &\le  W_1(\wh T^{(i)},\wh T^{(j)};\Dt) + \sum_{k\in \{i,j\}}W_1(\wh T^{(k)},T^{(k)}_*;\Dt)\\
	&\le  \|\hDtb - \Dtb\|_\i+ W_1(\wh T_i,\wh T_j;\hDt) + \sum_{k\in \{i,j\}}W_1(\wh T_k,T_k;\Dt),
\end{align*}
where in the second line we used the same argument as in (\ref{eqn:d inf init bound}) with the roles of $\hDtb$ and $\Dtb$ reversed. Combining this with (\ref{eqn:g hat g}), we find
\begin{align}
	&\left|W_1(\wh T^{(i)},\wh T^{(j)}; \hDt) - W_1(T^{(i)}_*,T^{(j)}_*; \Dt)\right| \nonumber\\
	&\quad \le \|\hDtb - \Dtb\|_\i + \sum_{k\in \{i,j\}}W_1(\wh T_k,T_k;\Dt)\nonumber\\
	&\quad \le \|\hDtb - \Dtb\|_\i + \|\Dtb\|_\i\frac{1}{2}\sum_{k\in \{i,j\}}\|\wh T^{(k)}_*- T^{(k)}_*\|_1,\label{final but dhat - d}
\end{align}
where we use Lemma \ref{thm:w props} in the last line.

Next, we find by the triangle inequality and (\ref{eqn:hdt def}) that
\[
\hDtb_{tl} \le d(\wh A_{\cdot t},  A_{\cdot t}) + d( A_{\cdot t},  A_{\cdot l}) + d( A_{\cdot l},  \wh A_{\cdot l}), 
\]
and 
\[\Dtb_{tl} \le d(A_{\cdot t},  \wh A_{\cdot t}) + d( \wh A_{\cdot t}, \wh A_{\cdot l}) +d( \wh A_{\cdot l},   A_{\cdot l}), \]
which together give
\[\|\hDtb - \Dtb\|_\i \le 2\max_{t\in [K]} d(A_{\cdot t},  \wh A_{\cdot t}).\]
Plugging this into (\ref{final but dhat - d}) completes the proof of (\ref{t bnd d}).

\qed \\

We use the following simple lemma in the proof of (\ref{topicw}) and (\ref{topictv}).
\begin{lemma}\label{thm:Hk topic}
	For any metric $d: \Delta_p\times \Delta_p\to \R$, $T,T'\in \Delta_K$, and $A\in \R^{p\times K}$ with columns in $\Delta_p$, and any $P\in \H_K$,
	\[\inf_{w\in \Gamma(  T,  T')} \sum_{k,l=1}^{ K}w_{kl}d( A_{\cdot k},  A_{\cdot l})=\inf_{w\in \Gamma(P^\top  T, P^\top  T')} \sum_{k,l=1}^{ K}w_{kl}d(( A P)_{\cdot k}, (A P)_{\cdot l}).\]
\end{lemma}
\begin{proof}
	Fix $P\in \H_K$ and let $\pi : [K]\to [K]$ be the associated bijection, so $P_{kl} = 1[k=\pi(l)]$ for all $k,l\in [K]$. From this it follows that
	\begin{equation}\label{perm iden}
		(P^\top T)_k = T_{\pi(k)},\quad (P^\top T')_k = T'_{\pi(k)}, \quad (AP)_{\cdot k} = A_{\cdot \pi(k)}. 
	\end{equation}
	Let $w\in \Gamma(T,T')$, and define $w^\pi$ by $w^\pi_{kl} = w_{\pi(k)\pi(l)}$. Then for $l\in [K]$,
	\begin{align*}
		\sum_{k=1}^K w_{kl}^\pi &= \sum_{k=1}^K w_{\pi(k)\pi(l)} \\
		&= T'_{\pi(l)} &\text{since } w\in \Gamma(T,T')\\
		&= (P^\top T')_l. & \text{by } (\ref{perm iden})
	\end{align*}
	A similar calculation shows $\sum_{l}w^\pi_{kl} = (P^\top T)_{k}$, and we thus conclude that $w^\pi\in \Gamma(P^\top T,P^\top T')$. 
	
	Next note that
	\begin{align*}
		\sum_{k,l=1}^K w_{kl}d(A_{\cdot k}, A_{\cdot l}) &= \sum_{k,l=1}^K w_{\pi(k)\pi(l)} d(A_{\cdot \pi(k)},A_{\cdot \pi(l)}) & \text{since }\pi\text{ is a bijection}\\
		&= \sum_{k,l=1}^K w^\pi_{kl} d((AP)_{\cdot k},(AP)_{\cdot l}) & \text{by } (\ref{perm iden})\\
		&\ge \inf_{w\in \Gamma(P^\top  T, P^\top  T')} \sum_{k,l=1}^{ K}w_{kl}d(( A P)_{\cdot k}, (A P)_{\cdot l}). & \text{since } w^\pi\in \Gamma(P^\top T,P^\top T')
	\end{align*}
	Since this holds for all $w\in \Gamma(T,T')$, we find
	\begin{equation}\label{perm ineq}
		\inf_{w\in \Gamma(  T,  T')} \sum_{k,l=1}^{ K}w_{kl}d( A_{\cdot k},  A_{\cdot l})\ge\inf_{w\in \Gamma(P^\top  T, P^\top  T')} \sum_{k,l=1}^{ K}w_{kl}d(( A P)_{\cdot k}, (A P)_{\cdot l}).
	\end{equation}
	Applying (\ref{perm ineq}) with $P$, $T$, $T'$ and $A$ replaced by $P^\top$, $P^\top T$, $P^\top T'$, and $AP$, respectively, gives the opposite inequality. Combined with (\ref{perm ineq}), this completes the proof.
\end{proof}

We also use the following standard results on the $1$-Wasserstein distance in the proofs in this section (see, for example, \citep{gibbs2002metrics, villani2003}).

\begin{lemma}\label{thm:w props}
	Let $D$ be a metric on a finite, non-empty, set $\cX$. Then,
	\begin{enumerate}
		\item $W_1(\cdot, \cdot; D)$ is a metric on $\Delta_{|\cX|}$.
		\item For any $a,b\in \Delta_{|\cX|}$,
		\begin{equation*}
			W_1(a,b;D)\le \max_{x,y\in\cX}D(x,y)\cdot  \frac{1}{2}\|a-b\|_1.
		\end{equation*}
	\end{enumerate}
\end{lemma}

\section{Technical lemmas}\label{app_tech_lemma}

\begin{lemma}\label{lem_basic}
	For any $t\ge 0$, 
	with probability $1-2pe^{-t/2}$,
	\[
	|X_j - \Pi_j| \le \sqrt{\Pi_j t \over N} + {2t\over 3N},\qquad \text{uniformly over }1\le j\le p.
	\]
\end{lemma}
\begin{proof}
	The proof follows by a simple application of the Bernstein's inequality for bounded random variables (see, for instance, the proof of Lemma 15 in \cite{TOP}). 
\end{proof}

\begin{lemma}\label{lem_oracle_error}
	Pick any $k\in [K]$. For any $t\ge 0$,  with probability $1-2e^{-t/2}$, 
	\[
	\left|
	\sum_{j\in  \oJ} {X_j -  \Pi_j \over \Pi_j}A_{jk}
	\right| \le \sqrt{\rho_k t \over N} + {2\rho_k t\over 3N},
	\]
	with $\rho_k = \max_{j\in\oJ}A_{jk}/\Pi_j$.
\end{lemma}
\begin{proof}
	For any $j\in \oJ$, notice that 
	\[
	X_j  - \Pi_j = {1\over N}\sum_{i=1}^N (B_{ij} - \Pi_j)
	\]
	where $B_{ij}\sim \textrm{Bernoulli}(\Pi_j)$ and $(B_{i1},\ldots, B_{ip})^\T \sim \textrm{Multinomial}(1, \Pi_j)$ for $i\in [N]$. Let 
	\[
	\sum_{j\in  \oJ} {X_j -  \Pi_j \over \Pi_j}A_{jk}  = {1\over N}\sum_{i=1}^N Z_i
	\]
	with 
	$$
	Z_i = \sum_{j\in  \oJ} (B_{ij} - \Pi_j){A_{jk} \over \Pi_j}
	$$
	such that $\EE[Z_i] = 0$, $|Z_i|\le 2\max_{j\in \oJ}A_{jk}/\Pi_j = 2\rho_k$ and 
	\[
	\EE[Z_i^2] = \textrm{Var}\left(
	\sum_{j\in  \oJ}{A_{jk} \over \Pi_j}B_{ij}
	\right) \le \sum_{j\in  \oJ} {A_{jk}^2 \over \Pi_j} \le  \rho_k \sum_{j\in  \oJ} A_{jk} \le \rho_k.
	\]
	Then an application of  Bernstein's inequality gives
	\[
	\PP\left\{
	{1\over N}\left|
	\sum_{i=1}^n Z_i
	\right| \ge \sqrt{\rho_k t \over N} + {2\rho_k t \over 3N}
	\right\}\le 2e^{-t/2},\quad \forall t\ge 0,
	\] 
	which completes the proof.
\end{proof}

\bigskip

Recall that 
\[
H = \sum_{j\in \oJ}{A_{j\cdot}A_{j\cdot}^\T \over \Pi_j} .
\]

\begin{lemma}\label{lem_oracle_error_whitening}
	For any $t\ge 0$,  with probability $1-2e^{-t/2+ K\log 5}$, 
	\[
	\left\|
	\sum_{j\in  \oJ} {X_j - \Pi_j \over \Pi_j} H^{-1/2} A_{j\cdot}
	\right\|_2 \le 2\sqrt{t \over N} + {2(1\vee \xi) \over 3\kappa(A_{\oJ},K) \Tm}\cdot {t\over N}.
	\]
\end{lemma}
\begin{proof}
	First note that 
	\[
	\left\|
	\sum_{j\in  \oJ} {X_j - \Pi_j \over \Pi_j} H^{-1/2} A_{j\cdot}
	\right\|_2  = \sup_{v: \|v\|_2 = 1} \left|
	\sum_{j\in  \oJ} {X_j - \Pi_j \over \Pi_j} A_{j\cdot}^\T H^{-1/2} v
	\right|.
	\]
	Let $\cN$  be a minimal $(1/2)$-net of $\{v:\|v\|_2=1\}$.  By definition and the property of $(1/2)$-net, we have 
	\[
	\left\|
	\sum_{j\in  \oJ} {X_j - \Pi_j \over \Pi_j} H^{-1/2} A_{j\cdot}
	\right\|_2 \le 2\sup_{v\in \cN }\left|
	\sum_{j\in  \oJ} {X_j - \Pi_j \over \Pi_j} A_{j\cdot}^\T H^{-1/2} v
	\right|.
	\]
	Pick any $v\in \cN$. 
	By similar reasoning as in the proof of Lemma \ref{lem_oracle_error}, we have 
	\[
	\sum_{j\in  \oJ} {X_j - \Pi_j \over \Pi_j} A_{j\cdot}^\T H^{-1/2} v = {1\over N}\sum_{i=1}^N Z_i
	\]
	with 
	$$
	Z_i = \sum_{j\in  \oJ} {B_{ij} - \Pi_j \over \Pi_j} A_{j\cdot}^\T H^{-1/2} v.
	$$
	Note that $\EE[Z_i] = 0$ and 
	\begin{align*}
		|Z_i| & \le \max\left\{\max_{j\in \oJ}{|A_{j\cdot}^\T H^{-1/2}v|  \over \Pi_j}, ~\left|
		\sum_{j\in \oJ}A_{j\cdot}^\T H^{-1/2} v \right|\right\}\\
		&  \le \max\left\{\max_{j\in \oJ}{|A_{j\cdot}^\T H^{-1/2}v|  \over \Pi_j}, ~\max_{j\in \oJ}\left|
		A_{j\cdot}^\T H^{-1/2} v \over \Pi_j \right|\sum_{j\in \oJ}\Pi_j\right\}\\
		&\le \max_{j\in \oJ}{\|A_{j\cdot}\|_\i \|H^{-1/2} v\|_1  \over \Pi_j}\\
		&\le \rho ~ \|H^{-1/2}v\|_1.
	\end{align*}
	Since, for any $u\in \RR^K$, one has 
	\begin{align*}
		u^\T H u &= \sum_{j \in \oJ} {(A_{j\cdot}^\T u)^2 \over \Pi_j}\\
		&=  \sum_{j \in \oJ} {(A_{j\cdot}^\T u)^2 \over \Pi_j} \sum_{j \in \oJ} \Pi_j &\textrm{ by }\sum_{j \in \oJ}\Pi_j = 1 \\
		&\ge \left(\sum_{j \in \oJ} |A_{j\cdot}^\T u|\right)^2 \\
		&\ge \kappa^2(A_{\oJ}, K) \|u\|_1^2, &\textrm{ by }(\ref{def_kappa_A})
	\end{align*}
	from which, we deduce that 
	\begin{equation}\label{bd_I_inv}
		\|H^{-1/2}u\|_1 \le \kappa^{-1}(A_{\oJ}, K) \|u\|_2,\qquad \forall u\in \RR^K.
	\end{equation}
	We thus conclude $ |Z_i| \le \rho \kappa^{-1}(A_{\oJ}, K)$. Furthermore, observe that
	\[
	\EE[Z_i^2] = \textrm{Var}\left(
	\sum_{j\in  \oJ} {B_{ij} \over \Pi_j} A_{j\cdot}^\T H^{-1/2} v
	\right) \le  v^\T H^{-1/2} \sum_{j\in \oJ}{A_{j\cdot}A_{j\cdot}^\T\over \Pi_j} H^{-1/2} v= 1.
	\]
	An application of the Bernstein's inequality gives
	\[
	\PP\left\{
	{1\over N}\left|
	\sum_{i=1}^n Z_i
	\right| \ge \sqrt{t \over N} + {\rho t \over 3\kappa(A_{\oJ}, K)N}
	\right\}\le 2e^{-t/2},\quad \forall t\ge 0.
	\]
	The proof follows immediately by using $\rho \le (1\vee \xi) / \Tm$ and taking a union bound over $v\in \cN$ together with $|\cN| \le 5^K$.
\end{proof}

Recall that 
\[
\wh H = \sum_{j\in \oJ}{X_j \over \Pi_j^2} A_{j\cdot}A_{j\cdot}^\T.
\]
The following lemma provides a concentration inequality of 
$H^{-1/2}(\wh H - H)H^{-1/2}$ via an application of the Matrix Bernstein inequality \citep{tropp2015introduction}.

\begin{lemma}\label{lem_I_deviation}
	For any $t\ge 0$, one has 
	\begin{equation*}
		\PP\left\{
		\left\|H^{-1/2}(\wh H - H)H^{-1/2}\right\|_{\rm op} \le  \sqrt{2B t\over N} + {B t \over 3 N}
		\right\}
		\ge 1 - 2K e^{-t/2},
	\end{equation*}
	with 
	$$
	B =  {1\vee \xi \over \kappa^{2}(A_{\oJ},K)\Tm^2}\left(1 + \xi\sqrt{K-s} \right).
	$$
	Moreover, if 
	\[
	N \ge C B \log K
	\]
	for some sufficiently large constant $C>0$, then, with probability $1-2K^{-1}$, one has 
	\[
	\lambda_{\min}(\wh H) \ge c \lambda_{\min}(H)
	\]
	for some constant $c>0$. 
\end{lemma}
\begin{proof}
	By similar arguments in the proof of Lemma \ref{lem_oracle_error}, we have 
	\[
	H^{-1/2}(\wh H - H)H^{-1/2} = {1\over N}\sum_{i=1}^N  Z_i
	\]
	with 
	\[
	Z_i = \sum_{j\in \oJ}{B_{ij} - \Pi_j \over \Pi_j^2} H^{-1/2}A_{j\cdot} A_{j\cdot}^\T H^{-1/2}.
	\]
	Notice $\EE[Z_i] = 0$.
	To apply the Matrix Bernstein inequality, we first find the bound for $\|Z_i\|_{\op}$ as 
	\begin{align*}
		\|Z_i\|_{\op} &\le  
		\max\left\{
		\max_{j\in \oJ}\left\|{H^{-1/2}A_{j\cdot} A_{j\cdot}^\T H^{-1/2}\over \Pi_j^2} \right\|_{\op},~ 
		\left\|\sum_{j\in \oJ}{H^{-1/2}A_{j\cdot} A_{j\cdot}^\T H^{-1/2}\over \Pi_j} \right\|_{\op}
		\right\}\\
		&\le \max\left\{
		\max_{j\in \oJ}{A_{j\cdot}^\T H^{-1} A_{j\cdot}\over \Pi_j^2},~ 1
		\right\}.
	\end{align*}
	Since 
	\[
	A_{j\cdot}^\T H^{-1} A_{j\cdot} \le \|A_{j\cdot}\|_\i\|H^{-1/2}\|_{1,\i}\|H^{-1/2}A_{j\cdot}\|_1,
	\]
	by using (\ref{bd_I_inv}), we obtain
	\[
	A_{j\cdot}^\T H^{-1} A_{j\cdot} \le \|A_{j\cdot}\|_\i \|A_{j\cdot}\|_2 \kappa^{-2}(A_{\oJ},K).
	\]
	Since $\Pi_j \ge \Tm \|A_{jS_T}\|_1$, we conclude
	\begin{align}\label{bd_AjIAj}\nonumber
		\max_{j\in \oJ}{A_{j\cdot}^\T H^{-1} A_{j\cdot}\over \Pi_j^2} &\le \rho \kappa^{-2}(A_{\oJ},K)\max_{j\in \oJ}{ \|A_{j\cdot}\|_2 \over \Pi_j}\\\nonumber
		&\le  {\rho \over \kappa^{2}(A_{\oJ},K)\Tm} \max_{j\in \oJ}\left({ \|A_{jS_T}\|_2 \over \|A_{jS_T}\|_1} +{ \|A_{jS_T^c}\|_2 \over \|A_{jS_T}\|_1}\right)\\\nonumber
		&\le  {\rho \over \kappa^{2}(A_{\oJ},K)\Tm}\left(1 + \xi\sqrt{|S_T|^c} \right)\\
		&\le {1\vee \xi \over \kappa^{2}(A_{\oJ},K)\Tm^2}\left(1 + \xi\sqrt{K-s} \right) = B.
	\end{align}
	We used the definition (\ref{def_xi}) in the penultimate step and used $\rho \le (1\vee \xi) / \Tm$ in the last step. Thus, 
	$
	\|Z_i\|_{\op}  \le B.
	$
	For the second moment of $Z_i$, we have 
	\begin{align*}
		\EE[Z_iZ_i^\T] & = \EE\left[
		\sum_{j\in \oJ}{B_{ij} - \Pi_j \over \Pi_j^2} H^{-1/2}A_{j\cdot} A_{j\cdot}^\T H^{-1/2}\sum_{j\in \oJ}{B_{ij} - \Pi_j \over \Pi_j^2} H^{-1/2}A_{j\cdot} A_{j\cdot}^\T H^{-1/2}
		\right]\\
		&= \EE\left[
		\sum_{j\in \oJ}{B_{ij}\over \Pi_j^2} H^{-1/2}A_{j\cdot} A_{j\cdot}^\T H^{-1/2}\sum_{j\in \oJ}{B_{ij} \over \Pi_j^2} H^{-1/2}A_{j\cdot} A_{j\cdot}^\T H^{-1/2}
		\right]\\
		&\quad + \EE\left[
		\sum_{j\in \oJ}{1 \over \Pi_j} H^{-1/2}A_{j\cdot} A_{j\cdot}^\T H^{-1/2}\sum_{j\in \oJ}{1\over \Pi_j} H^{-1/2}A_{j\cdot} A_{j\cdot}^\T H^{-1/2}
		\right]\\
		&\preceq \EE\left[
		\sum_{j\in \oJ}{B_{ij}\over \Pi_j^2} H^{-1/2}A_{j\cdot} A_{j\cdot}^\T H^{-1/2}
		\right]\max_{j\in \oJ}{H^{-1/2}A_{j\cdot} A_{j\cdot}^\T H^{-1/2} \over \Pi_j^2}  + \bI_K\\
		&= \bI_K \max_{j\in \oJ}{H^{-1/2}A_{j\cdot} A_{j\cdot}^\T H^{-1/2} \over \Pi_j^2}  + \bI_K.
	\end{align*}
	Since 
	\[
	\max_{j\in \oJ}{H^{-1/2}A_{j\cdot} A_{j\cdot}^\T H^{-1/2} \over \Pi_j^2} = \max_{j\in \oJ}{A_{j\cdot}^\T H^{-1} A_{j\cdot}\over \Pi_j^2},
	\]
	by (\ref{bd_AjIAj}), we conclude 
	\[
	\left\|
	\EE[Z_iZ_i^\T]
	\right\|_{\op} = \left\|
	\EE[Z_i^\T Z_i]
	\right\|_{\op}  \le  1 +  B \le 2B.
	\]
	The first result then follows from an application of the Matrix Bernstein inequality. 
	The second result follows immediately by using the first result with $t = 2\log K$ and noting that
	\begin{align*}
		\lambda_{\min}(\wh H) \ge \lambda_{\min}(H) \lambda_{\min}(H^{-1/2}\wh H H^{-1/2}) \ge \lambda_{\min}(H)\left(
		1 - \|H^{-1/2}(\wh H - H)H^{-1/2}\|_{\op}\right).
	\end{align*}
	We use Weyl's inequality in the second step. 
\end{proof}

\section{Algorithm of estimating the word-topic matrix $A$}\label{app_alg_A}

We recommend the following procedure for estimating the word-topic matrix $A$ under Assumption \ref{ass_sep}. It consists of two parts: (a) estimation of  the partition of anchor words, and (b) estimation of the word-topic matrix $A$.  
Step (a) uses the procedure proposed in \cite{TOP}, stated in Algorithm \ref{alg_I} while step (b) uses the procedure proposed in \cite{bing2020optimal}, summarized in Algorithm \ref{alg_1}.

Recall that $\bX = (X^{(1)}, \ldots, X^{(n)})$ with $N_i$ denoting the length of document $i$. 
Define 
\begin{equation}\label{est_Theta}
	\wh\Theta = {1\over n}\sum_{i =1}^n\left[
	{N_i \over N_i - 1}X^{(i)}X^{(i)\top}  - {1\over N_i-1} \textrm{diag}(X^{(i)})\right]
\end{equation}
and
\begin{equation}\label{def_R_hat}
	\wh R = D_X^{-1} 	\wh\Theta  D_X^{-1}
\end{equation}
with $D_X = n^{-1}\diag(\bX\1_n)$.

\subsection{Estimation of the index set of the anchor words, its partition and the number of topics}

We write the set of anchor words as $	I = \cup_{k\in [K]}I_k$ and its partition $\I = \{I_1,\ldots, I_K\}$ where
\[
I_k = \{j\in [p]: A_{jk} > 0, \ A_{\ell k} = 0, \ \forall \ \ell \ne j\}.
\] 
Algorithm \ref{alg_I} estimates the index set $I$, its partition $\mathcal{I}$ and the number of topics $K$ from the input matrix $\wh R$. The choice $C_1 = 1.1$ is recommended and is empirically verified to be robust in \cite{TOP}. A data-driven choice of $\delta_{j\ell}$ is specified in \cite{TOP} as 
\begin{equation}\label{delta3}
	\wh \delta_{j\ell} =  {n^2 \over\| \bX_{j\cdot}\|_1 \|  \bX_{\ell \cdot}\|_1 }
	\!\left\{\wh \eta_{j\ell}  +2 \wh\Theta_{j\ell}   \sqrt{\log M \over n }
	\!\left[\! \frac{n}{\|\bX_{j\cdot}\|_1} \!\left( \frac{1}{n} \sum_{i=1}^n \frac{\bX_{ji} }{ N_i} \right)^{\rs{1\over 2}}\!\!\!+  \!\frac{n}{\|\bX_{\ell\cdot}\|_1}  \!\left( \frac{1}{n} \sum_{i=1}^n \frac{\bX_{\ell i} }{ N_i} \right)^{\rs{1\over 2}}
	\right]\right\}
\end{equation}
with $M = n \vee p \vee \max_i N_i$ and 
\begin{align}\label{def_eta3}
	\wh \eta_{j\ell} =  &~3\sqrt{6}\left( \left\| \bX_{j\cdot}\right\|_\infty^{1\over 2} +\left\| \bX_{\ell\cdot}\right\|_\infty^{1\over 2}\right) \sqrt{ \log M\over n}\left(\frac{1}{n} \sum_{i=1}^n \frac{ \bX_{ji} \bX_{\ell i} }{N_i}\right)^{1\over 2}+
	\\\nonumber &+  {2 \log M \over n}\left( \| \bX_{j\cdot}\|_\infty + \| \bX_{\ell\cdot}\|_\infty \right) \frac1n \sum_{i=1}^n {1\over N_i}
	+ 31 \sqrt{(\log M)^4 \over n}\left({1\over n}\sum_{i =1}^n{\bX_{ji} + \bX_{\ell i} \over N_i^3} \right)^{\rs \frac12} 
\end{align}

\begin{algorithm}[ht]
	\caption{Estimate the partition of the anchor words $\I$ by $\wh \I$}\label{alg_I}
	\begin{algorithmic}[1]
		\Require matrix $\wh R\in\RR^{p\times p}$, $C_1$ and $Q\in\RR^{p\times p}$ such that $Q[j,\ell] := C_1\delta_{j\ell}$ 
		\Procedure{FindAnchorWords}{$\wh R$, $Q$}
		\State initialize $\wh \I = \emptyset$
		\For{$i\in [p]$} 
		\State $ \wh a_i = \argmax_{1\le j\le p}\wh R_{ij}$
		\State set $\wh I^{(i)} =  \{\ell\in [p]: \wh R_{i\wh a_i}-\wh R_{il} \le Q[i,\wh a_i]+ Q[i,\ell]\}$ and $\textsc{Anchor}(i) = \textsc{True}$
		\For {$j \in \wh I^{(i)}$}
		\State $\wh a_j = \argmax_{1\le k\le p}\wh R_{jk}$
		\If {$\Bigl|\wh R_{ij}-\wh R_{j\wh a_j}\Bigr| > Q[i,j] + Q[j, \wh a_j]$}   
		\State $\textsc{Anchor}(i) =\textsc{False}$
		\State \textbf{break}
		\EndIf	
		\EndFor
		\If {$\textsc{Anchor}(i) $}
		\State $\wh \I = \textsc{Merge}(\wh I^{(i)}$, $\wh \I$)
		\EndIf
		\EndFor
		\State\Return $\wh \I = \{ \wh I_1, \wh I_2, \ldots, \wh I_{\wh K}\}$ 
		\EndProcedure
		\Statex
		
		\Procedure{Merge}{$\wh I^{(i)}$, $\wh\I$}
		\For {$G \in \wh \I$}
		\If {$G \cap \wh I^{(i)}\ne \emptyset$} 
		\State replace $G$ in $\wh \I$ by $G\cap \wh I^{(i)}$
		\State\Return $\wh \I$
		\EndIf
		\EndFor
		\State {$\wh I^{(i)} \in \wh \I$}
		\State\Return $\wh \I$
		\EndProcedure
	\end{algorithmic}
\end{algorithm}

\subsection{Estimation of the word-topic matrix $A$ with a given partition of anchor words}

Given the estimated partition of anchor words 
$\wh \I = \{\wh I_1, \ldots, \wh I_{\wh K}\}$ and its index set $\wh I = \cup_{k\in [\wh K]}\wh I_k$, 
Algorithm \ref{alg_1} below estimates the matrix $A$.

\cite{bing2020optimal} recommends to set $\lambda = 0$ whenever $\wh M$ is invertible and otherwise choose $\lambda$  large enough
such that $\wh M + \lambda \bI_K$ is invertible. Specifically, \cite{bing2020optimal} recommends to choose $\lambda$ as 
\begin{equation}\label{rate_lambda_data}
	\lambda(t^*) = 0.01 \cdot  t^*  \cdot K\left({K\log (n\vee p) \over [\min_{i\in \wh I}(D_{X})_{ii}]n}\cdot  \frac{1}{n}\sum_{i=1}^n {1\over N_i}\right)^{1/2}.\\
\end{equation}
where
\[
t^* = \arg\min\left\{t\in \{0,1,2,\ldots\}:\, \wh M + \lambda(t)\bI_K\text{ is invertible}\right\}.
\]

\begin{algorithm}[H]
	\caption{Sparse Topic Model solver (STM)
	}\label{alg_1}
	\begin{algorithmic}[1]
		\Require frequency data matrix $\bX\in\R^{p\times n}$ with document lengths $N_1, \ldots, N_n$;  the partition of anchor words $\{I_1,\ldots, \wh I_{\wh K}\}$ and its index set $\wh I = \cup_{k\in [\wh K]}\wh I_k$, the tuning parameter $\lambda\ge 0$
		\Procedure{}{}
		\State compute $D_X = n^{-1}\diag(\bX\1_n)$, $\wh \Theta$ from (\ref{est_Theta}) and $\wh R$ from (\ref{def_R_hat})
		\State compute $\wh B_{\wh I\cdot}$ by  
		$\wh B_{i\cdot} = \be_k$ for each $i\in \wh I_k$ and  $k\in [\wh K]$
		\State compute $\wh M = \wh B_{\wh I\cdot}^+
		\wh R_{\wh I\wh I} \wh B_{\wh I\cdot}^{+\T}$ and $\wh H = \wh B_{\wh I\cdot}^+\wh R_{\wh I\wh I^c}$ with 
		$\wh B_{\wh I\cdot}^+= (\wh B_{\wh I\cdot}^\top \wh B_{\wh I\cdot})^{-1} \wh B_{\wh I\cdot}^\top$
		and $\wh I^c = [p]\setminus \wh I$
		\State solve $\wh B_{\wh I^c\cdot}$ from 
		\begin{alignat*}{2}
			\wh B_{j\cdot } &= 0, &&\quad \text{if }(D_X)_{jj}\le  {7\log(n\vee p) \over n} \left({1\over n} \sum_{i=1}^n{1 \over N_i}\right),\\
			\wh B_{j\cdot } &=\argmin_{\beta \ge 0,\ \|\beta\|_1 = 1}\beta^\top  (\wh M  + \lambda \bI_K)\beta - 2\beta^\top  \wh h^{(j)}, &&\quad \text{otherwise,}
		\end{alignat*}
		\indent for each $j\in \wh I^c$, with $\wh h^{(j)}$ being the corresponding column of $\wh H$.
		\State compute $\wh A$ by normalizing $D_X\wh B$ to unit column sums
		\State \Return $\wh A$
		\EndProcedure
	\end{algorithmic}
\end{algorithm}

\section{Some existing results on estimation of $A$}\label{app_upperbounds_A}

For completeness, we state the upper bounds of the estimator $\wh A$ of $A$ proposed in \cite{TOP} as well as the conditions under which $\wh A$ is optimal in the minimax sense, up to a logarithmic factor.

Let $N = N_i$ for all $i\in [n]$ for simplicity and write $M = n\vee p\vee N$. Under Assumption \ref{ass_sep}, recall that $I$ denotes the index set of anchor words and $I^c = [p]\setminus I$. From Corollary 8 of \cite{TOP}, under the conditions stated in Appendix \ref{app_sec_A_cond}, the following holds with probability at least  $1-8M^{-1}$,
\begin{align}\label{bd_upper_Ahat_inf_1}
	&\min_{P\in \P_K}\|\wh A - AP\|_{1,\i} \lesssim \sqrt{(|I|+ K |I^c|)\log M \over n N},\\\label{bd_upper_Ahat_1}
	&\min_{P\in \P_K}\|\wh A - AP\|_{1} \lesssim K\sqrt{(|I| + K |I^c|)\log M \over n N},
\end{align}
On the other hand, the minimax lower bounds in Theorem 6 of \cite{TOP} further imply that the rates in (\ref{bd_upper_Ahat_inf_1}) -- (\ref{bd_upper_Ahat_1}) are minimax optimal, up to the $\log(M)$ factor.

\subsection{Conditions under which   (\ref{bd_upper_Ahat_inf_1}) -- (\ref{bd_upper_Ahat_1}) hold}\label{app_sec_A_cond}

Let  $\bT := \bT_*$ and $\bPi:= \bPi_*$. Define
\[
\nu := n\zeta_i\zeta_j \left[
{\zeta_i \over \zeta_j} \wedge {\zeta_j \over \zeta_i} - \cos(\angle(\bT_{i\cdot}, \bT_{i\cdot})) 
\right]
\]
with $\zeta_i = \|\bT_{i\cdot}\|_2 / \|\bT_{i\cdot}\|_1$. This quantity quantifies the incoherence between rows of $\bT$.
\begin{enumerate}
	\item[(1)]  The matrix $A$ satisfies 
	\begin{enumerate}
		\item  the anchor word assumption in Assumption \ref{ass_sep}, 
		\item the balancing condition
		$\max_{i\in I} \|A_{i\cdot}\|_\i \asymp \min_{i\in I} \|A_{i\cdot}\|_\i$ and 
		\[
		{1\over |I^c|} \sum_{j \in I^c} {\|A_{j\cdot}\|_\i \over \max_{i\in I} \|A_{i\cdot}\|_\i}\lesssim 1,
		\]
		\item  the separation condition between anchor and non-anchor words
		\[
		\min_{i\in I, j\in I^c}\|\wt A_{i\cdot} - \wt A_{j\cdot}\|_1 \ge 8\delta / \nu
		\]
		where $A = D_\Pi^{-1}A D_T$ with $D_\Pi =  \diag(\bPi \1_n)$ and $D_T =   \diag(\bT \1_n)$ and the expression of $\delta$ is stated below. 
	\end{enumerate}
	
	\item[(2)] The matrix $\bT := \bT_*$ satisfies 
	\begin{enumerate}
		\item ${\rm rank}(\bT) = K$,
		\item  
		the incoherence condition $\nu > 4\delta$,
		\item  the balancing condition $\max_{k\in [K]} \sum_{i =1}^n\bT_{ki} \asymp \min_{k\in [K]} \sum_{i =1}^n\bT_{ki}$,
		\item the weak dependency condition $\sum_{k'\ne k}\sqrt{C_{kk'}} \lesssim \sqrt{C_{kk}}$ for all $k\in [K]$ with 
		$C = n^{-1}\bT\bT^\T$;
	\end{enumerate}
	\item [(3)] The matrix $\bPi:= \bPi_*$ satisfies $$\min_{j\in  [p]}{1\over n}\sum_{i=1}^n \bPi_{ji} \ge {2\log M \over 3N},\qquad \min_{j\in  [p]} \max_{1\le i\le n} \bPi_{ji} \ge {(3\log M)^2 \over N}.$$
\end{enumerate}

For detailed interpretation and justification of the above conditions, we refer the reader to Remarks 2, 3, 9, 10 \& 11 of \cite{TOP}. 
The quantity $\delta$ mentioned above represents the noise level in the context of estimating $A$, defined as $\delta = \max_{j,\ell \in [p]}\delta_{j\ell}$, where
\begin{equation}\label{delta_wt}
	\delta_{j\ell} := {p^2\eta_{j\ell} \over \mu_j \mu_{\ell}} +{p^2\Theta_{j\ell}\over \mu_j\mu_{\ell}}\left(\sqrt{p \over \mu_j} + \sqrt{p\over \mu_\ell}\right) \sqrt{\log M \over nN}.
\end{equation}
with 	
\begin{align}\label{def_eta_wt}\nonumber
	\eta_{j\ell} & = \sqrt{\Theta_{j\ell}\log M\over nN}\sqrt{{m_j+m_\ell \over p} \vee {\log^2M \over N} }+{2(m_j+m_\ell)\over p}{ \log M \over nN}\\
	&\qquad + \sqrt{\log^4 M \over nN^3}\sqrt{{\mu_j + \mu_{\ell}\over p} \vee {\log M \over N}}.
\end{align}
Here, $\Theta = n^{-1}\bPi\bPi^\T$,
\[
{m_{j} \over p} = \max_{1\le i\le n}\bPi_{ji},\qquad {\mu_j \over p} = {1\over n} \sum_{i=1}^n\bPi_{ji},\quad \ \forall \ j\in [p].
\]

\bigskip

Let $\wh A$ be the estimator obtained the procedure proposed in \cite{TOP}. Write $|I_{\max}| = \max_{k\in [K]} |I_k|$ and recall that $M = n\vee p \vee N$.

\begin{thm}\label{thm_A_sup_norm}
	Under conditions in (1) -- (3) stated in Appendix \ref{app_sec_A_cond}, assume 
	\begin{equation}\label{cond_Imax_KJ}
		(|I_{\max}| + K|I^c|)\log (M) \le c nN
	\end{equation}
	for some absolute constant $c\in (0,1)$. Then
	there exists some permutation matrix $P\in \P_K$ such that, with probability $1-8M^{-1}$, we have 
	\begin{align*}
		\|(\wh A P)_{j\cdot} - A_{j\cdot}\|_\i ~ \lesssim ~ \sqrt{\|A_{j\cdot}\|_\i{K\log(M) \over nN}}\left(
		1 \vee \sqrt{p\|A_{j\cdot}\|_\i}
		\right),\qquad \forall\ j\in [p].
	\end{align*}
\end{thm}
\begin{proof}
	The proof repeatedly uses the results and proofs in the supplement of \cite{TOP}. We only go through the major steps here and refer the reader to \cite{TOP} for detailed notation and formal statements. 
	
	We work on the event $\E := \E_1\cap \E_2\cap \E_2$ defined in page 8 of the supplement of \cite{TOP}. Recall that $\wh A = T^{-1}\sum_{i=1}^T \wh A^i$. It suffices to prove the desired result for any $i\in [T]$. We follow the arguments in the proof of Theorem 7 of \cite{TOP} and write $\wh A = \wh A^i$ for simplicity. 
	
	We first recall from Theorem 7 of \cite{TOP} that, by assuming the identity permutation without loss of generality, $\wh K = K$ and $\wh I_k = I_k$ for all $k\in [K]$ hold on $\E$.  As a result, we have $\wh L = L$ with $L = \{i_1, \ldots, i_K\}$ and $i_k \in I_k$ for each $k\in [K]$. 
	From page 28 of the supplement of \cite{TOP}, we have, for any $j\in [p]$ and $k\in [K]$,
	\begin{align*}
		|\wh A_{jk} - A_{jk}| &\le 
		{\left| \|\wh B_{\cdot k}\|_1 -  \|B_{\cdot k}\|_1\right|\over \|B_{\cdot k}\|_1}\wh A_{jk} + {|\wh B_{jk} - B_{jk}| \over \|B_{\cdot k}\|_1}\\
		&\le { \|\wh B_{\cdot k}- B_{\cdot k}\|_1\over \|B_{\cdot k}\|_1}\left(A_{jk} + |\wh A_{jk}-A_{jk}|\right) + {|\wh B_{jk} - B_{jk}| \over \|B_{\cdot k}\|_1}
	\end{align*}
	and $\|B_{\cdot k}\|_1 = p / \alpha_{i_k}$, where, following \cite{TOP}, we define  
	\begin{equation}\label{def_alpha_gamma}
		\alpha_j := p\max_{1\le k\le K} A_{jk}, \qquad {\gamma_k} := {K\over n}\sum_{i=1}^nW_{ki},\qquad \text{ for each }j\in[p], \ k\in[K].
	\end{equation}
	Since Corollary 8 of \cite{TOP} ensures that 
	\[
	\max_{k\in [K]}{\|\wh B_{\cdot k}- B_{\cdot k}\|_1\over \|B_{\cdot k}\|_1} \lesssim \sqrt{(|I_{\max}| + K|I^c|)\log (M)\over nN}
	\]
	with $|I_{\max}| = \max_k |I_k|,$ condition (\ref{cond_Imax_KJ}) implies 
	\[
	(1-c)|\wh A_{jk} - A_{jk}| \le  A_{jk} \sqrt{(|I_{\max}| + K|I^c|)\log (M)\over nN} + {\alpha_{i_k}\over p}|\wh B_{jk} - B_{jk}|,
	\]
	hence 
	\begin{equation}\label{disp_Ahatj_Aj}
		\|\wh A_{j\cdot} - A_{j\cdot}\|_\i \lesssim  \|A_{j\cdot}\|_\i \sqrt{(|I_{\max}| + K|I^c|)\log (M)\over nN} + \max_k{\alpha_{i_k}\over p}|\wh B_{jk} - B_{jk}|,
	\end{equation}
	It thus remains to bound the second term. We distinguish two cases: 
	
	(i) If $j\in I_k$ for some $k\in [K]$, then by using the fact that $|\wh B_{jk}-B_{jk}| = \|\wh B_{Ik} - B_{Ik}\|_1$ whenever $|I_k| = 1$, invoking the bound of $\|\wh B_{Ik}-B_{Ik}\|_1$ in display (58) of the supplement of \cite{TOP} with $I_k = \{j\}$ yields 
	\[
	|\wh B_{jk}-B_{jk}| ~ \lesssim ~ {\alpha_j \over \alpha_{i_k} \sqrt{\ua_I \gamma_k}} \sqrt{pK\log(M) \over nN} 
	\]
	where
	\[
	\ua_I = \min_{i\in I}\alpha_i,\qquad \oa_I = \max_{i\in I}\alpha_i.
	\]
	As a result, by using $\gamma_k \asymp 1$ 
	which is implied by condition (2) c) in Appendix \ref{app_sec_A_cond}, we have
	\[
	\max_k{\alpha_{i_k}\over p}|\wh B_{jk} - B_{jk}| ~ \lesssim ~ {\alpha_j \over  \sqrt{\ua_I \gamma_k}} \sqrt{K\log(M) \over npN} = \|A_{j\cdot}\|_\i \sqrt{pK\log (M)\over 
		nN} 
	\]
	where we also use $\ua_I\asymp \alpha_j$ for $j\in I$ from condition (1) b).
	
	(ii) If $j\in I^c$, then following the arguments in page 27 of the supplement of \cite{TOP}, one can deduce 
	\begin{align*}
		|\wh B_{jk}- B_{jk}|
		& \le \|\wh\omega_k\|_1 \|\wh \Theta_{jL} - \Theta_{jL}\|_\i + \|B_{j\cdot}\|_\i
		\left(
		\|\wh\Theta_{LL}\wh\omega_k - \be_k\|_1 + \|\wh\omega_k\|_1\|\wh\Theta_{LL}-\Theta_{LL}\|_{\i, 1}
		\right)\\
		&\le C_0\|\omega_k\|_1\max_{i\in L}\eta_{ij} + 
		2\|B_{j\cdot}\|_\i \|\omega_k\|_1 \lambda \\
		&\lesssim {p^2\over \alpha_{i_k}\ua_I}\|C^{-1}\|_{\i,1}\left(
		\max_{i\in L}\eta_{ij} + {p\over \ua_I}\|A_{j\cdot}\|_\i\max_{i\in L}\sum_{j\in L}\eta_{ij} 
		\right)\\
		&\lesssim {p^2K\over \alpha_{i_k}\ua_I}\left(
		\max_{i\in L}\eta_{ij} + {p\over \ua_I}\|A_{j\cdot}\|_\i \sqrt{\oa_I^3\overline{\gamma}\log M \over Knp^3N}
		\right)
	\end{align*}
	with $\overline{\gamma} = \max_k \gamma_k$,
	where in the penultimate step we have used $\|C^{-1}\|_{\i,1}\lesssim K$ and the bound for $\max_{i\in L}\sum_{j\in L}\eta_{ij}$  in the proof of Corollary 8 of \cite{TOP}. 
	Then by $\max_k \gamma_k \asymp 1$, $\oa_I\asymp \ua_I$, we obtain 
	\begin{align*}
		\max_k{\alpha_{i_k}\over p}|\wh B_{jk} - B_{jk}| ~ \lesssim ~ {pK \over \oa_I}
		\max_{i\in L}\eta_{ij} + \|A_{j\cdot}\|_\i \sqrt{p K\log M \over \ua_I nN}.
	\end{align*}
	Finally, to bound $\max_{i\in L}\eta_{ij}$, recalling from display (57) of the supplement of \cite{TOP}, we have, for any $i\in I_a$ and $a\in [K]$,
	\[
	\eta_{ij} ~ \lesssim ~\sqrt{{1\over n}\langle \bT_{a\cdot}, \bPi_{j\cdot}\rangle}\sqrt{\alpha_i (\alpha_i+\alpha_j)\log M\over np^2N} + {(\alpha_i + \alpha_j)\log M\over npN} +  \sqrt{(\alpha_i+\alpha_j)(\log M)^4\over npN^3}.
	\]
	Since 
	\[
	\langle \bT_{a\cdot}, \bPi_{j\cdot}\rangle = A_{j\cdot}^\T {1\over n}\bT\bT_{a\cdot}\le \|A_{j\cdot}\|_\i {1\over n}\sum_{t=1}^n\bT_{at} = \|A_{j\cdot}\|_\i {\gamma_a \over K}\lesssim {\|A_{j\cdot}\|_\i\over K},  
	\]
	we have
	\begin{align*}
		{pK \over \oa_I} \max_{i\in L}\eta_{ij} &~ \lesssim ~ \sqrt{\|A_{j\cdot}\|_\i\left(1 + {\alpha_j \over \ua_I}\right)}\sqrt{K\log (M)\over nN}
		+ \left(1 + {\alpha_j \over \ua_I}\right){K\log (M)\over nN}\\
		&\quad + \sqrt{\oa_I +\alpha_j \over \ua_I^2}\sqrt{pK^2\log^4(M) \over nN^3}.
	\end{align*}
	By using the same arguments in the proof of Lemma 13 of \cite{TOP}, one can show the last two terms are smaller in order than the first term under condition (3) in the Appendix \ref{app_sec_A_cond}. Therefore, we conclude 
	\[
	{pK \over \oa_I} \max_{i\in L}\eta_{ij} ~ \lesssim ~ \sqrt{\|A_{j\cdot}\|_\i\left(1 + {\alpha_j \over \ua_I}\right)}\sqrt{K\log (M)\over nN} \lesssim \|A_{j\cdot}\|_\i\sqrt{pK\log (M)\over (\alpha_j \wedge  \ua_I)nN}.
	\]
	Since condition (2) b) implies 
	\[
	1\le {1\over p}\sum_{i=1}^p \alpha_i \lesssim {1\over p}\left(|I^c| + |I|\right) \ua_I  = \ua_I
	\]
	and $\|A_{j\cdot}\|_\i = \alpha_j/ p$, we conclude
	\[
	\max_k{\alpha_{i_k}\over p}|\wh B_{jk} - B_{jk}| ~ \lesssim ~ \max\left\{
	\|A_{j\cdot}\|_\i \sqrt{pK\log (M)\over nN} ,\ \sqrt{\|A_{j\cdot}\|_\i {K\log (M)\over nN}}
	\right\}
	\]
	for any $j\in I^c$, which together with case (i), display (\ref{disp_Ahatj_Aj}) and the fact that $|I_{\max}|+K|I^c| \le pK$ completes the proof. 
\end{proof}

\section{Error bounds for $\mle - T_*$ in $\ell_2$ norm}\label{app_rate_ell_2}
In this section we state the results on $\|\mle - T_*\|_2$ with $\mle$ defined in (\ref{def_MLE}) for known $A$. 

Assume the conditions in Theorem \ref{thm_mle_fast}.  
The display (\ref{bd_IDelta_ell_2}) in the proof of Theorem \ref{thm_mle_fast} yields the following $\ell_2$ norm convergence rate of $\Tm - T^*$:
\begin{equation}\label{bd_ell_2}			
	\|\Tm - T^*\|_2 = \cO_\PP
	\left(\sigma^{-1}(I, s)\sqrt{K\over N}\right)
\end{equation}
where 
\[
\sigma^2(I, s) = \min_{S\subseteq [K], |S|\le s}\sup_{v\in \cC(S)} {v^\T I v  \over  \|v\|_2^2},
\qquad\textrm{with }\quad  
I = \sum_{j\in \oJ}{A_{j\cdot}A_{j\cdot}^\T \over \Pi_j}.
\]
On the event $\E_{supp}$,  (\ref{bd_ell_2}) could be improved to 
\begin{equation}\label{bd_ell_2_fast}			
	\|\Tm - T^*\|_2 = \cO_\PP
	\left(\sigma^{-1}(I, s)\sqrt{s\over N}\right)
\end{equation}
When $\sigma^{-1}(I, s)= \cO(1/\sqrt s)$, the rate in (\ref{bd_ell_2_fast}) is minimax optimal according to Theorem \ref{thm_minimax}. Indeed, since $\wh T - T_*\in \cC(s)$ for any $T_*\in \cT'(s)$ and $\wh T \in \Delta_K$, we have $ \|\wh T - T_*\|_1\le 2\|[\wh T - T_*]_{S_*}\|_1 \le 2\sqrt{s}\|\wh T - T_*\|_2$ with $S_* = \supp(T_*)$ and $|S_*| = s$. Consequently, Theorem \ref{thm_minimax} implies 
\[
\inf_{\wh T}\sup_{T_* \in \cT'(s)}\PP\left\{
\|\wh T - T_*\|_2 \ge 
c_0\sqrt{1\over N}
\right\} \ge c_1.
\]

\begin{remark}[Discussion on $\sigma(I, s)$]	
	To understand the magunitude of $\sigma(I, s)$, it is helpful to consider $s = K$, in which case $\sigma^2(I, K)$ is simply the smallest eigenvalue of $I$. We then have 
	\[
	\sigma^2(I, K) \le \min_{k}  \sum_{j\in \oJ}{A_{jk}^2 \over \Pi_j} \le \min_{k}  \sum_{j\in \oJ}{A_{jk}^2 \over A_{jk}T_k}  \le \max_k {1\over T_k} \le K,
	\]
	where the last step uses $\max_k T_k\ge 1/K$.  We also note that this upper bound is attainable  (in terms of rates), for instance, when all words are anchor words and the numbers of anchor words of all topics are the same order.  Immediately, when $\sigma(I, K)  \asymp \sqrt{K}$, (\ref{bd_ell_2}) yields 
	\[
	\|\mle - T^*\|_2 = \cO_\PP\left(
	\sqrt{1/ N}
	\right). 
	\]
	
	Next, we connect  $\sigma(I, s)$ to a quantity that is only related with $A$ and $s$. By (\ref{kappa_I}), we have 
	\[
	\sigma(I, s) \ge \kappa_2(A_{\oJ}, s)
	\]
	where 
	\[
	\kappa_2(A_{\oJ}, s) = \min_{S\subseteq [K], |S|\le s}\sup_{v\in \cC(S)} {\|A_{\oJ} v\|_1 \over \|v\|_2}.
	\]
	Since 
	$$
	{\kappa_2(A_{\oJ}, s) \over \sqrt s} \le  \kappa(A_{\oJ}, s)\le \kappa_2(A_{\oJ}, s),
	$$
	the ideal case is $\kappa_2(A_{\oJ}, s) \asymp \sqrt{s} ~ \kappa(A_{\oJ}, s) $, whence 
	\[
	\|\mle - T^*\|_2 = \cO_\PP\left(
	\kappa^{-1}(A_{\oJ}, s)\sqrt{1/ N}
	\right). 
	\]
\end{remark}

{\small 
	\setlength{\bibsep}{0.85pt}{
\bibliographystyle{ims}
\bibliography{ref}
}
}

\end{document}